\documentclass{amsart}
\usepackage{amssymb}
\usepackage{amsbsy}
\usepackage{amscd}
\usepackage{amsmath}
\usepackage{amsthm}
\usepackage{amsxtra}
\usepackage{latexsym}
\usepackage[mathscr]{eucal}
\usepackage{upgreek}
\usepackage{wasysym}
\usepackage[all,cmtip]{xy}

\usepackage[dvipdfm, colorlinks]{hyperref}

\newcommand{\Vb}{V^{\kappa_c}(\g)^{\*_{\mf{z}(\g)} b}}
\newcommand{\QZ}{Q^{\LZ}}

\newcommand{\Slo}{\mathscr{S}}
\newcommand{\V}{\mathbf{V}^\mathcal{S}}
\newcommand{\LZ}%{\mc{Z}(\affg)}%
{\mc{Z}}
\newcommand{\eqW}{\mathbf{W}}
\newcommand{\g}{\mathfrak{g}}
\newcommand{\Dch}{\mathcal{D}^{ch}}

\newcommand{\D}{\mc{D}}

\newcommand{\bw}[1]{\bigwedge\nolimits^{#1}}

\newcommand{\mc}{\mathcal}
\newcommand{\mf}{\mathfrak}

\newcommand{\on}{\operatorname}
\newcommand{\KL}{\on{KL}}

\newcommand{\isomap}{{\;\stackrel{_\sim}{\to}\;}}

\newcommand{\W}{\mathscr{W}}

\newcommand{\h}{\mathfrak{h}}
\renewcommand{\a}{\mathfrak{a}}

\newcommand{\affg}{\widehat{\mathfrak{g}}}
\newcommand{\affa}{\widehat{\mathfrak{a}}}

\newcommand{\fing}{\mathfrak{g}}

\newcommand{\Q}{\mathbb{Q}}
\renewcommand{\1}{{\mathbf{1}}}
\newcommand{\bra}{{\langle}}
\newcommand{\ket}{{\rangle}}

\newcommand{\lam}{\lambda}
\newcommand{\ra}{\rightarrow}
\newcommand{\+}{\mathop{\oplus}}
\newcommand{\Z}{\mathbb{Z}}
\newcommand{\Mod}{\text{-}\mathrm{Mod}}

\newcommand{\cprime}{$'$}

\renewcommand{\*}{{\otimes}}
\newcommand{\C}{\mathbb{C}}

\theoremstyle{plain}
\newtheorem{Th}{Theorem}[section]

\newtheorem{Pro}[Th]{Proposition}

\newtheorem{Lem}[Th]{Lemma}

\newtheorem{Co}[Th]{Corollary}

\theoremstyle{definition}

\theoremstyle{remark}
\newtheorem{Def}[Th]{Definition}
\newtheorem{Rem}[Th]{Remark}
\newtheorem{Conj}{Conjecture}

\newcommand{\semiinf}{\frac{\infty}{2}}
\newcommand{\Zhu}{\on{Zhu}}%{A}%{\mathscr{Z}hu}

\DeclareMathOperator{\im}{Im}

\DeclareMathOperator{\ch}{ch}

\DeclareMathOperator{\id}{id}

\DeclareMathOperator{\op}{op}
\DeclareMathOperator{\End}{End}
\DeclareMathOperator{\gr}{gr}

\DeclareMathOperator{\Hom}{Hom}

\DeclareMathOperator{\ad}{ad}

\DeclareMathOperator{\Spec}{Spec}

\title[Chiral algebras of class $\mathcal{S}$ and  symplectic varieties]{Chiral algebras of class $\mathcal{S}$ and Moore-Tachikawa symplectic varieties}
\author{Tomoyuki Arakawa}
\address{Research Institute for Mathematical Sciences, Kyoto University,
 Kyoto 606-8502 JAPAN}

\email{arakawa@kurims.kyoto-u.ac.jp}

\thanks{This work is partially  supported 
by JSPS KAKENHI Grant Number
No.\ 20340007 and No.\ 23654006.}

\begin{document}
\maketitle

\begin{abstract}
We give a functorial construction
of 
the genus zero {\em chiral algebras of class $\mc{S}$},
that is,
the vertex algebras corresponding to the {\em theory of class $\mc{S}$}
associated with genus zero punctured Riemann surfaces
%in the four-dimensional $\mathcal{N}=2$ superconformal field theoeis
 via the 4d/2d duality
discovered by 
Beem, Lemos, Liendo, Peelaers, Rastelli and van Rees in physics.
%\cite{BeeLemLie15}.
We show that there is a unique family of vertex algebras
satisfying the required conditions
and that they are all simple and conformal.
In fact, our construction works for any
complex semisimple group $G$ that is not necessarily simply laced.
Furthermore,
we show that the associated varieties 
of these vertex algebras 
are exactly the genus zero {\em Moore-Tachikawa symplectic varieties} %\cite{MooTac12}
that have been recently constructed by 
Braverman, Finkelberg and Nakajima %\cite{BFN17}
using the geometry of the affine Grassmannian for the Langlands dual group.
\end{abstract}

\section{Introduction}
Let $G$ be a simply connected semisimple linear algebraic group over $\C$,
$\g=\on{Lie}(G)$.
In \cite{BFN17},
Braverman, Finkelberg and Nakajima have constructed
a new family of 
the (possibly singular) symplectic varieties
\begin{align*}
W^b_G=\on{Spec}H^*_{\check{G}[[t]]}(\on{Gr}_{\check{G}},i_{\Delta}^!(\boxtimes_{k=1}^b\mc{A}_R)),\quad b\in \Z_{\geq 1},
\end{align*}
%of dimension $b \dim \g-(b-2)\on{rk}\g$
equipped with a Hamiltonian action of $\prod_{k=1}^b G$.
%$\overbrace{G\times G\times \dots \times G}^{b\text{ times}}$.
Here 
$\check{G}$ is the Langlands dual group of $G$,
$\on{Gr}_{\check{G}}$ is the 
affine Grassmannian for $\check{G}$,
$i_{\Delta} : \on{Gr}_{\check{G}}\ra \prod_{k=1}^b  \on{Gr}_{\check{G}}$
is the diagonal
embedding,
and $\mc{A}_R$ is the perverse sheaf corresponding to the regular representation $\mc{O}(G)$ of  $G$ under the geometric Satake correspondence \cite{Mirkovic:2007vl}. 
These symplectic varieties 
%are
% called (genus zero) {\em Moore-Tachiakwa symplectic varieties},
%and
satisfy the following properties:
\begin{align*}
W^{b=1}_G\cong G\times \Slo,\quad W^{b=2}_G\cong T^*G,\quad 
(W^b_G \times W^{b'}_G)/\!/\!/\Delta(G)\cong W_G^{b+b'-2},
\end{align*}
where $\Slo$ is a Kostant-Slodowy slice in $\g^*$
and the left-hand side in the last isomorphism is the symplectic reduction
of $W^b_G\times W^{b'}_G$ with respect to the diagonal action of $G$.
The existence of 
the symplectic varieties satisfying the above conditions 
was conjectured by Moore and Tachikawa \cite{MooTac12}
%on the existence of a 2d TQFT whose  targets holomorphic symplectic varieties.
 and  $W^b_G$  is called 
a (genus zero) {\em Moore-Tachikawa symplectic variety}.
We note that $W^b_G$ is conical for $b\geq 3$.

In this article
we perform a
chiral quantization of the above construction.
More precisely,
 we construct a family of  vertex algebras
$\V_{G,b}$, $b\in \Z_{\geq 1}$,
equipped with a vertex algebra homomorphism
(chiral quantum moment map)
$
\bigotimes_{i=1}^b V^{\kappa_c}(\g)\ra  \V_{G,b}
$
satisfying the following properties:
\begin{align}
&\V_{G,b=1}\cong H_{DS}^0(\Dch_G),\quad \V_{G,b=2}\cong \Dch_G,
\label{eq:initial}\\
&H^{\semiinf+i}(\affg_{-\kappa_\g},\g,\V_{G,b}\* \V_{G,b'})\cong \delta_{i,0}\V_{G,b+b'-2}
\quad  \text{(associativity)},
\label{eq:ass-property}
\end{align}
where 
$V^{\kappa_c}(\g)$ is the universal affine vertex algebra associated with $\g$ at the critical level $\kappa_c$,
$H_{DS}^0(?)$ is the quantized Drinfeld-Sokolov reduction functor \cite{FF90},
$\Dch_G$ is the algebra of chiral differential operators \cite{MalSchVai99,BeiDri04} on $G$ at $\kappa_c$,
$\affg_{-\kappa_\g}$ is the affine Kac-Moody algebra associated with $\g$ and the minus of the Killing form 
$\kappa_\g$ of $\g$,
and
$H^{\semiinf+\bullet}(\affg_{-\kappa_\g},\g,\V_{G,b}\* \V_{G,b'})$ is the relative semi-infinite cohomology 
with coefficients in $\V_{G,b}\* \V_{G,b'}$,
which can be regarded as an affine analogue of the Hamiltonian reduction with respect to the diagonal $G$-action,
see \eqref{eq;initial},
\eqref{eq:main-def},
\eqref{eq:cylinder}
and Theorem \ref{Th:fusion}
 for the details.
We show that
$\V_{G,b}$,
$b\in \Z_{\geq 1}$,
is the unique
family of 
vertex algebra objects in $ \on{KL}^{\otimes b}$ that satisfies the conditions
\eqref{eq:initial} and
\eqref{eq:ass-property}
(Remark~\ref{Rem:uniquness})
and that
they are all simple and conformal (Theorem \ref{Th:simple} and Proposition \ref{Pro:central-charge}).
The central charge of 
$\V_{G,b}$ is given by 
$$b\dim \g -(b-2)\on{rk}\g-24 (b-2)(\rho|\rho^{\vee})
%=\dim W^b_G-24 (b-2)(\rho|\rho^{\vee})
,$$
and  the character of $\V_{G,b}$ is given by the following formula (Proposition \ref{Pro:conic}):
\begin{align}
\on{tr}_{\V_{G,b}}(q^{L_0}z_1z_2\dots z_b)=\sum_{\lam\in P_+}\left(\frac{q^{\bra \lam,\rho^{\vee}\ket}
\prod\limits_{j=1}^{\infty}(1-q^j)^{\on{rk}\g}}{\prod\limits_{\alpha\in \Delta_+}(1-q^{\bra \lam+\rho,\alpha^{\vee}\ket})}\right)^{b-2}
\prod_{k=1}^b\on{tr}_{\mathbb{V}_{\lam}} 
(q^{-D}z_k),
\label{eq:ch-formula}
\end{align}
where 
$(z_1,\dots,z_b)\in T^b$,
$T$ is a maximal torus of $G$,
%$\h$ is a Cartan subalgebra of $\g$,
$\Delta_+$ is the set of positive roots of $\g$,
$\rho=\frac{1}{2}\sum_{\alpha\in \Delta_+}\alpha$,
$\rho^{\vee}=\frac{1}{2}\sum_{\alpha\in \Delta_+}\alpha^{\vee}$,
$P_+$ is the set of integrable dominant weights of $\g$,
$\mathbb{V}_{\lam}$ is the Weyl module of 
the affine Kac-Moody algebra
$\affg_{\kappa_c}$ at  $\kappa_c$ with highest weight $\lam$
(see \eqref{eq:Weyl-module}, \eqref{eq:critical-Weyl}),
and $D$ is the standard degree operator
of the affine Kac-Moody algebra
that acts as zero on the highest weight
of $\mathbb{V}_{\lam}$.

%\smallskip

We construct $\V_{G,b}$ as a cohomology of a certain BRST complex, see \eqref{eq:main-def}.
Our construction is an affine analogue of 
that of Ginzburg and Kazhdan \cite{GinKaz}.

In the case that $G$ is simply laced,
the above stated results were
conjectured  by
Beem, Lemos, Liendo, Peelaers, Rastelli and van Rees
\cite{BeeLemLie15},
and
Beem, Peelaers, Rastelli and van Rees
\cite{Beem:2015yu}
in the theory of
the 4d/2d duality in physics discovered in \cite{BeeLemLie15},
as we explain below.

The  4d/2d duality \cite{BeeLemLie15}
associates a conformal vertex algebra $V_{\mc{T}}$ to any
four-dimensional $\mathcal{N}=2$ superconformal field theory
(4d $\mathcal{N}=2$ SCFT)  $\mc{T}$.
There is a distinguished class of 4d $\mathcal{N}=2$ SCFTs called 
the {\em theory of class $\mc{S}$} \cite{Gai12,GaiMooNei13},
which is labeled by a complex semisimple group $G$
and
a punctured Riemann surface $\Sigma$.
The  vertex algebras
associated to the theory of class $\mc{S}$ 
by the 
4d/2d duality are called 
the {\em chiral algebras of class $\mc{S}$} \cite{Beem:2015yu}.
%By the uniqueness result,
The vertex algebra $\V_{G,b}$ 
constructed in this article is exactly the
chiral algebra of class $\mc{S}$
corresponding to the group $G$
and a $b$-punctured Riemann surface of genus zero.
Chiral algebras
of class $\mc{S}$
associated with higher genus Riemann surfaces
are obtained by glueing $\V_{G,b}$'s and 
such a gluing procedure 
%
%Moreover,
%chiral algebras of class $\mc{S}$  in general
%%associated with 
%%a pointed Riemann surfaces 
is described in terms of 2d TQFT whose targets are vertex algebras
(\cite{Beem:2015yu}),
which is well-defined by 
the properties
\eqref{eq:initial} and
\eqref{eq:ass-property},
see \cite{TachikawaVOA, Tac18}
for mathematical expositions.
%The existence of such a 2d TQFT
%is assured by
%the existence of the family of
%vertex algebras $\{\V_{G,b}\}$
%satisfying

It is known that
4d $\mathcal{N}=2$ SCFTs
have several important invariants (or observables).
One of them is the {\em Schur index},
which is a formal series.
The 4d/2d duality  \cite{BeeLemLie15}
is constructed
in such a way 
that 
the  Schur index of a 4d $\mathcal{N}=2$ SCFT $\mc{T}$
is obtained as the character of the corresponding vertex algebra $V_{\mc{T}}$.
A
recent remarkable conjecture
of  Beem and Rastelli \cite{BeeRas}
states  that one can also recover the 
geometric invariant  $Higgs(\mc{T})$ of $\mc{T}$, called
the {\em Higgs branch},
which is a  possibly singular symplectic variety, 
from the vertex algebra $V_{\mc{T}}$.
More precisely, 
they expect that we have an isomorphism
\begin{align}
Higgs(\mc{T})\cong X_{V_{\mc{T}}}
\label{eq;Higgs-branch-conj}
\end{align}
of Poisson varieties
for any 4d $\mathcal{N}=2$ SCFT $\mc{T}$,
where 
%$Higgs(\mc{T})$ is the
%the geometric invariant of the 4d theory $\mc{T}$ called
%the {\em Higgs branch},
%which is a  possibly singular symplectic variety, and
$X_V$ is the {\em associated variety} \cite{Ara12}
of a  vertex algebra $V$,
 see \cite{A.Higgs} for a survey.

In this paper we prove the conjecture  \eqref{eq;Higgs-branch-conj} for the genus zero class $\mc{S}$ theories as well.

According to Moore and Tachikawa \cite{MooTac12},
the Higgs branches of the theory of class $\mc{S}$
 are
exactly\footnote{There is a subtlety
for higher genus cases due to the non-flatness of the moment map.}
 the Moore-Tachikawa symplectic varieties
that have been constructed mathematically 
as explained above
%Braverman, Finkelberg and Nakajima
(\cite{BFN17}).
 Therefore
 showing  the isomorphism
between the Higgs branches and the associated varieties
%for  genus zero class $\mc{S}$ theories
is equivalent 
to the following statement (Theorem \ref{Th:Higgs-branch-conj}):
\begin{align}
X_{\V_{G,b}}\cong W^b_G\quad \text{%as Poisson varieties 
for all }b\geq 1.
\label{eq:higgs-branch}
\end{align}
In other words,
$\V_{G,b}$ is a {\em chiral quantization} (Definition \ref{def:cq}) of the
Moore-Tachikawa symplectic variety $W^b_G$.
Note that this in particular proves that 
the associated variety of $\V_{G,b}$ is symplectic for all $b\geq 1$ (Corollary \ref{Co:symplectic}).

We remark that there is a close relationship between
the Higgs branches of 4d $\mathcal{N}=2$ SCFTs and the {\em Coulomb branches} \cite{BraFinNak}
of 3d $\mathcal{N}=4$ SUSY gauge theories.
Indeed,
it is known \cite[Theorem 5.1]{BFN17} 
 that  in type $A$
the Moore-Tachikawa variety $W_G^b$ is isomorphic to 
 the Coulomb branch
 of a  star shaped quiver gauge theory.

%We remark that $W^b_G$ is coninal for all $b\geq 3$.

In view of \cite{BeeRas},
we conjecture that 
$\V_{G,b}$ is {\em quasi-lisse}  (\cite{Arakawam:kq}),
that is,
$W^b_G$ has finitely many symplectic leaves for all $b\geq 1$
(Conjecture \ref{Conj}).

Although  
it is very difficult to describe the vertex algebra $\V_{G,b}$ explicitly in general,
conjectural descriptions of   $\V_{G,b}$
have been given in several cases in \cite{BeeLemLie15},
and it is possible to confirm them 
 using 
Remark \ref{Rem:showing-isomorphisms} below.
For instance,
$\V_{G=SL_2,b=3}$ is isomorphic to the $\beta\gamma$-system
$SB((\C^2)^{\otimes 3})$ associated with the 
symplectic vector space $W^{b=3}_{G=SL_2}=(\C^2)^{\otimes 3}$
(Theorem \ref{Th:SL23pt});
$\V_{G=SL_2,b=4}$ is the simple affine vertex algebra $L_{-2}(D_4)$ associated with $D_4$ at level $-2$
(Theorem \ref{Th:SL24pt}).
(Note that $-2$ is the critical level for $\widehat{\mathfrak{sl}_2}$.)
The statement \eqref{eq:higgs-branch}
for $G=SL_2$, $b=4$  reproves the fact \cite{AM15} that
$X_{L_{-2}(D_4)}$ is isomorphic to the minimal nilpotent orbit closure $\overline{\mathbb{O}_{min}}$ in $D_4$,
and
\eqref{eq:ass-property} gives a non-trivial isomorphism
\begin{align*}
L_{-2}(D_4)\cong H^{\semiinf+0}((\widehat{\mf{sl}}_{2})_{-4},\mf{sl}_2,
SB((\C^2)^{\otimes 3})\* SB((\C^2)^{\otimes 3}))
\end{align*}
that was  conjectured in \cite{BeeLemLie15}.
Also,
\eqref{eq:ch-formula}
provides a non-trivial identity of 
the normalized character \cite{KasTan00} of $L_{-2}(D_4)$,
whose homogeneous specialization is known to be
$E_4'(\tau)/240\eta(\tau)^{10}$
 (\cite{Arakawam:kq}).
 It would be interesting to compare \eqref{eq:ch-formula}
 with \cite[Conjecture 3.4]{KWnegative}.
 We also show that
 $\V_{G=SL_3,b=3}$ is the simple affine vertex algebra $L_{-3}(E_6)$ associated with $E_6$ at level $-3$
 (Theorem \ref{Th:E6}).
 The statement \eqref{eq:higgs-branch}
for $G=SL_3$, $b=3$  reproves the fact \cite{AM15} that
$X_{L_{-3}(E_6)}$ is isomorphic to the minimal nilpotent orbit closure $\overline{\mathbb{O}_{min}}$ in $E_6$.

In general,
$\V_{G,b}$ is a $W$-algebra in the sense that it is not generated by a Lie algebra
(see \cite{BeeLemLie15,Beem:2015yu,LemPee15}),
and the associativity isomorphism \eqref{eq:ass-property} 
provides non-trivial isomorphisms involving such simple vertex algebras.

\medskip

This paper is organized as follows.
In Section \ref{section:Vertex algebras and associated varieties}
we fix some notation and collect some facts on vertex algebras
 needed in the later sections.
In Section \ref{section:Some variants of Chiral Hamiltonian reductions}
we discuss the BRST formalism of the Hamiltonian reduction and its chiral analogue
that is used in later sections.
Most of the results here are slight modifications of those in \cite{Ara09b}.
In Section \ref{Sectioin:Feigin-Frenkel center}
we fix notation for the Feigin-Frenkel center $\mf{z}(\affg)$,
which plays an important role in our construction.
In fact,
we construct the vertex algebra
$\V_{G,b}$
from $b$ copies of $\V_{G,1}$ by ^^ ^^ letting them talk to each other through 
Feigin-Frenkel center",
see \eqref{eq:main-def}.
Section \ref{Section:cdoonG}
is devoted to the chiral differential operators $\Dch_{G,\kappa}$ on $G$.
By the property \eqref{eq:identify-cdo}
that has been proved by Arkhipov  and Gaitsgory
\cite{ArkGai02},
$\Dch_{G,\kappa}$  at the critical level $\kappa=\kappa_c$
is identified with
 the {\em cylinder chiral algebra} $\V_{G,2}$ of class $\mathcal{S}$ (see \cite[Conjecture 4]{Beem:2015yu}).  
%We in particular recall  that the associativity 
%\eqref{eq:ass-property} for $b=2$
%has been already proved by see 
The main purpose of
Section \ref{section:Quantum Drinfeld-Sokolov reduction and equivariant affine $W$-algebras}
is to introduce the {\em equivariant affine $W$-algebra} $\eqW_{G,f}^\kappa$,
which is an
affinization of Losev's  equivariant finite affine $W$-algebra
(see Theorem \ref{Th:Zhu-eq-W}).
The {\em cap chiral algebra} 
$\V_{G,b=1}$ (\cite[Section 5]{Beem:2015yu})
of class $\mathcal{S}$  is the equivariant affine $W$-algebra
$\eqW_{G,f_{prin}}^{\kappa_c}$ for
a principal nilpotent element $f_{prin}$ at the critical level.
We show that
the usual affine $W$-algebra $\W^\kappa(\g,f)$
can be defined as the  commutant of the affine vertex algebra inside  $\eqW_{G,f}^\kappa$ (Proposition \ref{Pro:g[t]t-invariant-part})
and 
gives a yet another realization  of the quantized Drinfeld-Sokolov reduction in terms of $\eqW_{G,f}^\kappa$
(Theorem \ref{Th:DS-realization}).
This  in particular proves the  fundamental property
\begin{align}
H^{\semiinf+i}(\affg_{-\kappa_\g},\g,\V_{G,1}\* \V_{G,b})\cong \delta_{i,0}H_{DS}(\V_{G,b})
\cong H^{\semiinf+i}(\affg_{-\kappa_\g},\g,\V_{G,b}\* \V_{G,1})
\label{eq:ass-red}.
\end{align}
%In class $\mathcal{S}$ theory,
%one can also input the ramification (?)  data labeled by nilpotent elements of $\g$
%for  each puncture of the Riemann surface (see \cite[Section 2.2]{BeeLemLie15}).
%The corresponding vertex algebras 
%can be obtained straightforwardly  from $\V_{G,b}$
%using $\eqW_{G,f}^{\kappa_c}$ associated various nilpotent element $f$.
%We show that $\eqW_{G,f}^\kappa$ is an affinization of Losev's  equivariant finite affine $W$-algebra
%(Theorem \ref{Th:Zhu-eq-W}).
Section \ref{sec:ch-uv}
is a short digression
to the {\em chiral universal centralizer} $\mathbf{I}_G^\kappa$,
which 
 is a strict chiral quantization of the {\em universal centralizer}
$(G\times \Slo)\times_{\g^*}\Slo$.
For $G=SL_2$ and a generic $\kappa$,
$\mathbf{I}_G^\kappa$ is identified with the {\em modified regular representation}
of the Virasoro algebra introduced by 
Frenkel and Styrkas \cite{FreSty06}.
At the critical level  $\kappa=\kappa_c$,
the  chiral universal centralizer $\mathbf{I}_G^{\kappa_c}$
is noting but the ^^ ^^ vacuum chiral algebra" $\V_{G,b=0}$.
However, the role of this interesting vertex algebra is not clear to us at this point.
In Section \ref{section:Drinfeld-Sokolov reduction at the Critical level},
we collect some results on the Drinfeld-Sokolov reduction at the critical level.
%which plays an important role.
%Indeed,
The Drinfeld-Sokolov reduction functor $H_{DS}^0(?)$ is important
in our construction since
the associativity \eqref{eq:ass-property} 
and \eqref{eq:ass-red}
imply that
\begin{align}
\V_{G,b}=H_{DS}^0(\V_{G,b+1}).
\label{eq:DS-reduction-reduces-b}
\end{align}
It follows from
\eqref{eq:DS-reduction-reduces-b}
that it is enough to construct an inverse functor
to $H_{DS}^0(?)$ in order to define $\V_{G,b}$,
since we already know $\V_{G,2}$ and $\V_{G,1}$.
It is indeed possible to construct such a functor at
the critical level  $\kappa=\kappa_c$,
and this is done in
Section \ref{Sec:main-construction}.
More precisely,
let 
\begin{align*}
 H^{\semiinf+\bullet}(\LZ,\eqW_{G,f_{prin}}^{\kappa_c}\* M),
\end{align*}
denote the semi-infinite $\mf{z}(\affg)$-cohomology  with coefficients
$\eqW_{G,f_{prin}}^{\kappa_c}\* M$,
where $\mf{z}(\affg)$ is viewed as a commutative Lie algebra
and $\mf{z}(\affg)$ acts on $\eqW_{G,f_{prin}}^{\kappa_c}\* M$ diagonally
(see Section \ref{Sec:main-construction} for the precise definition).
We show that
\begin{align*}
M\cong  H^{\semiinf+0}(\LZ,\eqW_{G,f_{prin}}^{\kappa_c}\* H_{DS}^0(M))
\end{align*}
for any $M\in \on{KL}$ (Theorem \ref{Th:it-is-an-inverse-to-DS}).
Finally in Section \ref{section:chiral-algebras-of-classs-S},
 we define the genus zero chiral algebra  $\V_{G,b}$ of class $\mathcal{S}$,
 %the genus zero chiral algebras of class $\mathcal{S}$, 
 and show that they satisfy the above stated properties.
 In Appendix \ref{appendix},
 we give some examples of
 $\V_{G,b}$.
 
%From 
% Section \ref{section:More on cdo on $G$ at the critical level}
% we restrict ourselves to the critical level case.
% In Section \ref{section:More on cdo on $G$ at the critical level}
%we study more on the structure of the {\em cylinder vertex algebra}
%$\V_{G,b=2}=\Dch_G$ that is needed for the later purposes.
%In Section \ref{section:Drinfeld-Sokolov reduction at the Critical level}
%we collect the known facts about the Drinfeld-Sokolov reduction at the critical level
%that is needed to for the later purposes.

\subsection*{Acknowledgments}
The author 
is grateful to
Christopher Beem, 
%Damien Calaque,
Gurbir Dhillon,
Boris Feigin,
Michael Finkelberg,
Davide Gaiotto,
Victor Ginzburg,
Madalena Lemos,
Victor Kac, 
Anne Moreau, Hiraku Nakajima, Takahiro Nishinaka, Wolfger Peelaers, Leonardo Rastelli, Shu-Heng Shao
and Yuji Tachikawa
for very useful discussions.
The results in this paper has been announced 
in part
in 
^^ ^^ Conference in Finite Groups and Vertex Algebras dedicated to Robert L. Griess on the occasion of his 71st birthday",
Academia Sinica, Taiwan, August 2016,
^^ ^^ 60th Annual Meeting of the Australian Mathematical Society",
Camberra, December 2016,
^^ ^^ Exact operator algebras in superconformal field theories",
Perimeter Institute for Theoretical Physics,
Waterloo,
December 2016, 
^^ ^^ 7th Seminar on Conformal Field Theory",
Darmstadt, February 2017,
^^ ^^ Representation Theory XV",
Dubrovnik,
June 2017,
^^ ^^ Affine, Vertex and W-algebras",
Rome, December 2017,
^^ ^^ The 3rd KTGU Mathematics Workshop for Young Researchers",
Kyoto, February 2018,
^^ ^^ Gauge theory, geometric langlands and vertex operator algebras",
Perimeter Institute for Theoretical Physics, Waterloo, March 2018,
^^ ^^ Vertex Operator Algebras and Symmetries",
RIMS, Kyoto, July 2018,
^^ ^^ Vertex algebras, factorization algebras and applications",
IPMU, Kashiwa,  July 2018,
^^ ^^ The  International Congress of Mathematics",
Rio de Janeiro, August, 2018,
^^ ^^ Workshop on Mathematical Physics",
ICTP-SAIFR,
 S\~{a}o Paulo, August 2018,
and
^^ ^^ Geometric and Categorical Aspects of CFTs",
the Casa Matem\'{a}tica Oaxaca,
September 2018.
He thanks  the organizers of these conferences and
 apologies for the delay of the paper.
The author 
is partially supported in part by JSPS KAKENHI Grant Numbers 17H01086, 17K18724.
%and by Perimeter Institute for Theoretical Physics.

\section{Vertex algebras and associated varieties}
\label{section:Vertex algebras and associated varieties}
A \emph{vertex algebra} 
\cite{Bor86}
consists of a vector space $V$ with a distinguished vacuum vector $|0\rangle  \in V$ and a vertex operation, which is a linear map $V \otimes V \rightarrow V((z))$, written $a \otimes b \mapsto a(z)b = (\sum_{n \in \Z} a_{(n)} z^{-n-1})b$, such that the following are satisfied:
\begin{itemize}
\item (Unit axioms) $(|0\rangle)(z) = \id_V$ and $a(z)|0\rangle  \in a + zV[[z]]$ for all $a \in V$;

\item (Locality)
$(z-w)^n[a(z),b(w)]=0$ for a sufficiently large $n$
for all $a, b\in V$.
\end{itemize}
The operator $\partial: a \mapsto a_{(-2)}|0\rangle $ is called the {\em translation operator} and it satisfies $(\partial a)( z) =
[\partial,a(z)]= \partial_z a(z)$. The operators $a_{(n)}$ are called \emph{modes}.

For elements $a,b$ of a vertex algebra $V$
 we have the following
{\em Borcherds identity}
for any
$m,n\in \Z$:
\begin{align}
&[a_{(m)},b_{(n)}]=\sum_{j\geq 0}\begin{pmatrix}m\\j\end{pmatrix}(a_{(j)}b)_{(m+n-j)},
\label{eq:B1}\\
& (a_{(m)}b)_{(n)}=\sum_{j\geq 0}(-1)^j\begin{pmatrix}m\\j\end{pmatrix}
(a_{(m-j)}b_{(n+j)}-(-1)^m b_{(m+n-j)}a_{(j)}).
\end{align}
By regarding the Borcherds identity
as fundamental relations,
 representations of a vertex algebra are naturally defined
(see \cite{Kac98,FreBen04} for the details).

We  write 
\eqref{eq:B1}
as 
\begin{align*}
a(z)b(w)\sim \sum_{j\geq 1}\frac{1}{(z-w)^j}(a_{(j)}b)(w)
\end{align*}
and call it the operator product expansion (OPE) of $a$ and $b$.

For a vertex algebra $V$,
let $V^{op}$ denote the opposite vertex algebra of $V$,
that is, the vertex algebra $V$ equipped with the new vertex operation
$a(z)^{op}b=a(-z)b$
and the new translation operator $\partial^{op}=-\partial$.

A vertex algebra $V$ is called {\em commutative}
if the left side (or the right side) of \eqref{eq:B1} are zero for all $a,b\in V$,
$m,n\in \Z$.
If this is the case, $V$ can be regarded as a {\em differential algebra} (=a unital commutative algebra with a derivation)
 by the multiplication
$
a.b=a_{(-1)}b
$
and the derivation
$\partial$.
Conversely,
any differential algebra can be naturally equipped with the structure of a commutative vertex algebra.
Hence,
commutative vertex algebras are the same\footnote{However,
their modules are different.} as differential algebras (\cite{Bor86}).

For an affine scheme $X$,
let
$J_{\infty}X$ be the
{\em arc space} of $X$ that is defined by the functor of points
$\on{Hom}(\on{Spec}R,J_{\infty}X)=\on{Hom}(\on{Spec}R[[t]], X)$
for all $\C$-algebra $R$.
The ring $\C[J_{\infty}X]$ is naturally a differential algebra,
and hence is a commutative vertex algebra.
In the case that 
$X$ is an affine Poisson scheme,
 $\C[J_{\infty}X]$ has \cite{Ara12} the structure of 
a {\em Poisson vertex algebra},
see
\cite[16.2]{FreBen04}, 
\cite{Kaclecture17}
for the definition of Poisson vertex algebras.

It is known by  Haisheng Li \cite{Li05} that
any vertex algebra $V$ is canonically filtered,
and hence can be regarded
 as a quantization of 
the associated graded Poisson vertex algebra 
$\on{gr} V=\bigoplus_p F^pV/F^{p+1}V$,
where 
$F^{\bullet}V$ is the canonical filtration of $V$.
By definition,
\begin{align*}
F^pV=\on{span}_\C\{(a_1)_{(-n_1-1)}\dots (a_r)_{(-n_r-1)}|0\rangle\mid a_i\in V,\ n_i\geq 0,\ \sum_i n_i\geq p\}.
\end{align*}
A vertex algebra is called {\em separated} if
the filtration $F^\bullet V$ is separated, that is, 
$\bigcap_p F^p V=0$.
For instance, any positively graded vertex algebra is separated.

The subspace 
 $$R_V:=V/F^1V=F^0V/F^1V\subset \on{gr}V$$
is called 
 {\em Zhu's $C_2$-algebra of $V$}.
The Poisson vertex algebra structure of $\on{gr} V$ restricts to a Poisson algebra structure
of $R_V$,
which  is given by
\begin{align*}
\bar a.\bar b=\overline{a_{(-1)}b},\quad
\{\bar a,\bar b\}=\overline{a_{(0)}b},
\end{align*}
where $\bar a$ is the image of $a\in V$
in $R_V$.
The {\em associated variety}
of $V$
is by definition the 
Poisson variety
\begin{align*}
X_V=\on{Specm}(R_V)
\end{align*}
and the {\em associated scheme} of $V$
is the Poisson scheme $\tilde{X}_V=\on{Spec}(R_V)$ 
(\cite{Ara12}). 
We have a surjective homomorphism
\begin{align}
\mc{O}(J_{\infty}\tilde{X}_V)\twoheadrightarrow \gr V
\label{eq:surj-map-PVA}
\end{align}
of Poisson vertex algebras
which is the identity map on $R_V$
(\cite{Li05,Ara12}).

The scheme $\on{Spec}(\gr V)\subset J_{\infty}\tilde{X}_V$ is called the {\em singular support} of $V$
and is denoted by $SS(V)$.

\begin{Def}\label{def:cq}
 Let $X=\on{Spec}R$ be an affine Poisson scheme.
A {\em chiral quantization} of $X$
is a separated vertex algebra $V$ such that $X_V\cong X$
as Poisson varieties.
A {\em strict chiral quantization} of $X$ is a 
chiral quantization of $X$ such that 
$\tilde{X}_V\cong X$ as schemes and
\eqref{eq:surj-map-PVA} is an isomorphism, so that
the singular support of $V$ identifies with $J_{\infty}X$.
\end{Def}

\begin{Th}[\cite{AraMorCore}]\label{Th:AMoreau}
Let $V$ be a  separated vertex algebra such that 
$\tilde{X}_V$ is a reduced, smooth symplectic variety.
Then 
\begin{itemize}
\item $\gr V$ is simple as a vertex Poisson algebra;
\item $V$ is simple;
\item $V$ is a strict chiral quantization of $\tilde{X}_V$.
\end{itemize}
\end{Th}

Let $\phi:V\ra W$
be a vertex algebra homomorphism.
Then $\phi(F^pV)\subset F^pW$,
and 
$\phi$ induces a homomorphism of vertex Poisson algebras
$\gr V\ra \gr W$,
which we denote by $\gr \phi$.
It restricts to a Poisson algebra
homomorphism
$R_V\ra R_W$,
and therefore induces a morphism
$X_W\ra X_V$ of Poisson varieties.

\smallskip

A vertex algebra is called {\em conformal}
if there exists a vector $\omega$,
called the {\em conformal vector},
such that the corresponding field $\omega(z)=\sum_{n\in \Z}L_nz^{-n-2}$
satisfies the following conditions:
(1) $[L_m, L_n]=(m-n)L_{m+n}+\frac{m^3-m}{12}\delta_{m+n,0}c\on{id}_V$,
 where $c$ is a constant called the {\em central charge}  of V;
 (2)
$L_0$ acts semisimply on V;
(3) $L_{-1}=\partial$.
For a conformal vertex algebra
$V$ we set $V_{\Delta}=\{v\in V\mid L_0v=\Delta V\}$,
so that $V=\bigoplus_{\Delta}V_{\Delta}$.
We write
$\Delta_a=\Delta$ if $a\in V_{\Delta}$
and call $\Delta_a$ the {\em conformal weight} of $a$.
A conformal vertex algebra is called {\em conical} 
if $V_{\Delta}=0$ unless $\Delta\in \frac{1}{m}\Z_{\geq 0}$ for some $m\in \Z_{\geq 1}$
and $V_{0}=\C$
(\cite{Arakawam:kq}).
%A {\em positive energy representation} $M$
%of a conformal vertex algebra $V$ is a $V$-module  $M$
%on which $L_0$ acts diagonally and the $L_0$-eivenvalues on $M$ is bounded from below.
%An  {\em ordinary  representation} is a positive energy representation such that
%each $L_0$-eivenspaces are finite-dimensional.
%For a finitely generated ordinary representation $M$,
%the normalized character
%\begin{align*}
%\chi_V(q)=\on{tr}_V (q^{L_0-c/24})
%\end{align*}
%is well-defined.

For a conformal vertex algebra $V=\bigoplus_{\Delta}V_{\Delta}$,
one defines {\em Zhu's algebra} \cite{FreZhu92} $\on{Zhu}(V)$ of  $V$ by
\begin{align*}
\on{Zhu}(V)=V/V\circ V,\quad V\circ V=\on{span}_{\C}\{a\circ b\mid a,b\in V\},
\end{align*}
where $a\circ b=\sum_{i\geq 0}\begin{pmatrix}
\Delta_a\\
i\end{pmatrix}a_{(i-2)}b$ for $a\in V_{\Delta_a}$.
The algebra
$\on{Zhu}(V)$ is a unital associative algebra by the multiplication
$a* b=\sum_{i\geq 0}\begin{pmatrix}
\Delta_a\\
i\end{pmatrix}a_{(i-1)}b$.
The grading of $V$ gives a filtration on $\on{Zhu}(V)$
which makes it 
almost-commutative,
 and
there is a surjective map
\begin{align}
R_V\twoheadrightarrow \on{gr}\on{Zhu}(V)
\label{eq:surj-RV}
\end{align}
of Poisson algebras.
The map \eqref{eq:surj-RV}
 is an isomorphism
if  $V$ admits a PBW basis (\cite{A2012Dec}).

Let $V\Mod$
be the category of positively graded
$V$-modules.
For $M\in V\Mod$,
the top component  $M_{top}$ is naturally a $\Zhu(V)$-modules.
There is a left adjoint
$$\on{Ind}_{\Zhu(V)}^V:\Zhu(V)\Mod\ra V\Mod$$
to the functor $M\mapsto M_{top}$ (\cite{Zhu96}, see also \cite[5.1]{Ara07},\cite[6.2.1]{Rosellen}).
 
 For a vertex subalgebra $W$ of $V$,
let
\begin{align*}
\on{Com}(W,V)&=\{v\in V\mid [v_{(m)}, w_{(n)}]=0\ \forall w\in W,\ n\in \Z\}\\
&=\{v\in V\mid w_{(n)}v=0\ \forall w\in W,\ n\geq 0\}.
\end{align*}
 Then $ \on{Com}(W,V)$ is a vertex subalgebra of $V$,
 called a {\em coset vertex algebra},
 of a {\em commutant vertex algebra} (\cite{FreZhu92}).
 The vertex subalgebras $W_1$ and $W_2$ are said to form a dual pair
 in $V$ if $W_1=\on{Com}(W_2,V)$ and $W_2=\on{Com}(W_1,V)$.
 The coset vertex algebra 
$\on{Com}(V,V)$ is called the {\em center} of $V$
and is denoted by $Z(V)$.

Let $\mf{a}$ be a finite-dimensional Lie algebra.
For
 an invariant symmetric   bilinear form  $\kappa$ on $\mf{a}$,
 let $\affa_{\kappa}$ be the corresponding  Kac-Moody  affinization of $\mf{a}$:
\begin{align*}
\affa_{\kappa}=\mf{a}((t))\+ \C \mathbf{1},
\end{align*}
where the commutation relation of $\affa_{\kappa}$ is given by
\begin{align*}
[xf,yg]=[x,y]fg+\on{Res}_{t=0}(gdf)\kappa(x,y)\mathbf{1},
\quad [\mathbf{1},\affa_{\kappa}]=0.
\end{align*}
Define
\begin{align*}
V^\kappa(\mf{a})=U(\affa_{\kappa})\*_{U(\mf{a}[[t]]\+\C \mathbf{1})}\C,
\end{align*}
where $\C$ is the one-dimensional representation of $\mf{a}[t]]\+\C \mathbf{1}$
on which $\mf{a}[[t]]$ acts trivially and $\mathbf{1}$ acts as the identity.
There is a unique vertex algebra structure on $V^\kappa(\mf{a})$
such that
$|0\ket =1\* 1$ is the vacuum vector and
$$x(z)=\sum_{n\in \Z}(xt^n)z^{-n-1},$$
for $z\in \mf{a}$,
where on the left-hand-side 
we have regarded $\g$ as a  subspace of $V^{\kappa}(\mf{a})$ by the embedding
$\g\hookrightarrow V^{\kappa}(\mf{a})$, $x\mapsto (xt^{-1})|0\ket$.
The vertex algebra 
$V^\kappa(\mf{a})$ is a
strict chiral quantization of $\mf{a}^*$
and is 
called
the {\em universal affine vertex algebra associated with $\mf{a}$ at level $\kappa$}.

%The vertex algebra 
%$V^\kappa(\mf{a})$  is naturally $\Z_{\geq 0}$-graded:
%$V^\kappa(\mf{a})=\bigoplus_{\Delta \in \Z_{\geq 0}}V^\kappa(\mf{a})_{\Delta}$,
%where  $\deg |0\ket=0$ and $\deg (xt^n)=-n$.
%We have
%\begin{align}
%\on{Zhu}(V^\kappa(\mf{a}))\cong U(\mf{a})\label{eq:Zhu-affineVA}
%\end{align}
%(\cite{FreZhu92})
%and under this identification the filtration on $\on{Zhu}(V^\kappa(\mf{a}))$
%coincides with the PBW filtration of $U(\mf{a})$.
%The map
%\eqref{eq:surj-RV} for $V^\kappa(\mf{a})$ is the PBW isomorphism
%$\mc{O}(\mf{a}^*)\isomap \gr U(\mf{a})$.
%
Note that we have $V^\kappa(\mf{a})^{op}\isomap V^{\kappa}(\mf{a})$,
$x\mapsto -x$,  ($x\in \mf{a}$).

A $V^\kappa(\mf{a})$-module is the same as 
 a smooth $\affa_{\kappa}$-module,
 that is, a $\affa_{\kappa}$-module $M$ such that 
 $xt^n m=0$ for $n\gg 0$ for all $x\in \a$, $m\in M$.

\section{Some variants of Chiral Hamiltonian reductions}
\label{section:Some variants of Chiral Hamiltonian reductions}

For an algebraic group
$\mathbf{G}$ and 
an   $\mathbf{G}$-scheme
$X$, 
let
$\on{QCoh}^{\mathbf{G}}(X)$ be the category of 
quasi-coherent sheaves on $X$ equivariant under the adjoint action of $\mathbf{G}$.
When $X$ is affine
we do not 
distinguish
between a coherent sheaf on an affine scheme and the module of its global sections.

 Let $G$ be a simply connected semisimple algebraic group,
$\g=\on{Lie}(G)$.
An object of $\on{QCoh}^{G}(\g^*)$
is the same as a Poisson $\mc{O}(\g^*)$-module
on which the adjoint action of $\g$ is locally finite.

 A {\em Poisson  algebra object} in $\on{QCoh}^{G}(\g^*)$
is a Poisson algebra $R$ equipped with
a Poisson algebra homomorphism
$\mu_R:\mc{O}(\g^*)\ra R$
such that 
the adjoint action of $\g$ 
on $R$
is locally finite.
If $R$ is a Poisson algebra object in  $\on{QCoh}^{G}(\g^*)$,
$X=\on{Spec}(R)$ is a $G$-scheme and
the $G$-equivariant morphism $\mu_R^*:X\ra \g^*$  is the moment map
for the $G$-action.
We denote by $R^{op}$ the Poisson  algebra object
in $\on{QCoh}^{G}(\g^*)$ with the opposite Poisson structure,
$X^{op}=\Spec R^{op}$.
We have  $\mu_{R^{op}}=-\mu_R$.

Let $R$ be a Poisson algebra object in $\on{QCoh}^G(\g^*)$,
$X=\Spec R$.
Suppose that
the moment map
$\mu_X:X\ra \g^*$ is flat,
and
that
there exists a closed subscheme $S$ of $X$
such that the action map $G\times S\ra X$, $(g,s)\ra gs$,
is an isomorphism of $G$-schemes.
Then
$\mc{O}(S)\cong \mc{O}(X)^G$ is a Poisson subalgebra of $\mc{O}(X)$,
and hence
$S$ is a Poisson subvariety of $X$.
Let $R'$ be an another Poisson algebra object in $\on{QCoh}^G(\g^*)$,
$X'=\Spec R'$,
$\mu_{X'}: X'\ra \g^*$, the moment map.
Then $R^{op}\* R$ is a Poisson algebra object in $\on{QCoh}^G(\g^*)$
with the moment map
$\mu:X^{op}\times X'\ra \g^*$, $(x,x')\mapsto -\mu_X(x)+\mu_{X'}(x')$
for the diagonal $G$-action.
Since $\mu_X$ is flat, so is $\mu$.
Moreover,
$\mu^{-1}(0)\cong X\times_{\g^*}X'\cong G\times (S\times_{\g^*} X')$,
where 
$S\ra \g^*$ is the restriction of $\mu_X$ to $S$.
Hence,
the Hamiltonian reduction
$(X^{op}\times X')/\!/\!/ \Delta(G)$ is well-defined,
and we have
\begin{align}
(X^{op}\times X')/\!/\!/ \Delta(G)=\mu^{-1}(0)/G\cong S\times_{\g^*}X'.
\end{align}
According to \cite{KosSte87,Kuw15},
this construction is realized by a BRST cohomology
as follows.
Let $\overline{Cl}(\g)$ be the classical Clifford algebra associated with  $\g\+ \g^*$,
which is a Poisson superalgebra generated by odd elements
$\bar \psi_i$,
$\bar \psi_i^*$,
$i=1,\dots,\dim \g$,
with the relations
$\{\bar \psi_i,\bar \psi_j^*\}=\delta_{ij}$,
$\{\bar \psi_i,\bar \psi_j^*\}=\{\bar \psi_i^*,\bar \psi_j^*\}=0$.
For a Poisson algebra  object $A$  in $\on{QCoh}^G(\g^*)$,
set
\begin{align*}
C(\g^{cl},A)=A\* \overline{Cl}(\g).
\end{align*}
Then $C(\g^{cl},A)$ is naturally a Poisson superalgebra.
Define  the odd element
\begin{align*}
\bar Q=\sum_{i=1}^{\dim \g}\mu_A(x_i)\* \bar \psi_i^*- 1\* \frac{1}{2}c_{ij}^k\bar \psi_i^*\bar \psi_j^*\bar \psi_k\in C(\g^{cl},A),
\end{align*}
where 
$\{x_i\}$ is a basis of $\g$ and $c_{ij}^k$ is the corresponding structure constant.
Then  $\{\bar Q,\bar Q\}=0$,
and hence, $(\ad \bar Q)^2=0$.
It follows that $(C(\g^{cl},A),\ad \bar Q)$
is a differential graded Poisson algebra,
where
its grading is defined by
$\deg a=0$ for $a\in A$,
$\deg \psi_i^*=1$,
$\deg \psi_i=-1$.
The corresponding cohomology is denoted by
$H^{\semiinf+\bullet}(\g^{cl}, A)$,
which is naturally a Poisson superalgebra.
We have
\begin{align}
H^{\semiinf+\bullet}(\g^{cl}, R^{op}\*R)\cong \mc{O}((X^{op}\times X')/\!/\!/ \Delta(G))\* H^\bullet(\g,\C)
\label{eq:BRST-cl}%\\
%\cong \mc{O}((X^{op}\times X')/\!/\!/ \Delta(G))\* H^\bullet_{DR}(G)
%\nonumber 
\end{align}
as Poisson superalgebras
(\cite{Kuw15}).
Here
 $H^\bullet(\g,\C)$ is the  Lie algebra $\g$-cohomology
with coefficients in the trivial representation $\C$,
which
is isomorphic to the De Rham cohomology ring $H^\bullet_{DR}(G)$ 
of $G$,
and to the ring of invariant forms $\bw{\bullet}(\g^*)^\g$.
In particular,
$H^{\semiinf+0}(\g^{cl}, R^{op}\*R)=0$ for $i<0$ or $i>\dim G$,
%where $w_0$ is the longest element of the Weyl group $W$ of $\g$
%and $\ell:W\ra \Z_{\geq 0}$ is the length function,
and
$H^{\semiinf+0}(\g^{cl}, R^{op}\*R)\cong \mc{O}((X^{op}\times X')/\!/\!/ \Delta(G))\cong 
\mc{O}(S)\*_{\mc{O}(\g^*)}R=R|_{S}$ as Poisson algebras.
More generally,
for $M\in \on{Coh}^G(\g^*)$,
$C(\g^{cl},R^{op}\* M)=R^{\op}\* M\* \overline{Cl}(\g)$
is naturally a differential graded Poisson $C(\g^{cl},R^{op}\*\mc{O}(\g^*))$-module.
The corresponding cohomology
space
$H^{\semiinf+\bullet}(\g^{cl},R^{op}\*M)$
is a Poisson module over $H^{\semiinf+\bullet}(\g^{cl},R^{op}\*\mc{O}(\g^*))\cong \mc{O}(S)\* H^\bullet(\g,\C)$,
and we have
$$H^\bullet(\g^{cl},R^{op}\*M)\cong \left(\mc{O}(S)\*_{\mc{O}(\g^*)}M\right)\* H^{\bullet}(\g,\C).$$

The arc space $J_{\infty}G$ of $G$
 is the
 proalgebraic group $G[[t]]$,
 and
the Lie algebra
$J_{\infty}\g$
of $J_{\infty}G$ equals to $\g[[t]]$.
An object $M$ of $\on{Coh}^{J_{\infty}G}(J_{\infty}\g^*)$
is 
the same as 
 a Poisson
vertex $\mc{O}(J_{\infty}\g^*)$-module
on which the action of $J_{\infty}\g^*$ is locally finite.

%There is an arc space analogue of the above construction.
A {\em Poisson vertex algebra object} in $\on{QCoh}^{J_{\infty}G}(J_{\infty}\g^*)$
is a Poisson vertex algebra $V$
equipped with a Poisson vertex algebra homomorphism
$\mu_{V}:\mc{O}(J_{\infty}\g^*)\ra V$
on which
the action of $J_{\infty}\g$ is locally finite.
If $V$ is a Poisson vertex algebra object in $\on{QCoh}^{J_{\infty}G}(J_{\infty}\g^*)$,
$\Spec (V)$ is a $J_{\infty}G$-scheme.
The $J_{\infty}G$-equivariant morphism
$\mu_{V}^*:\Spec V\ra J_{\infty}\g^*$
%induced by $\mu_V$ 
is
 called the {\em chiral  moment map}.
 For a Poisson algebra object $R$ in  $\on{QCoh}^{G}(\g^*)$,
 $J_{\infty}R:=\mc{O}(J_\infty \Spec R)$ is a Poisson vertex algebra object in $\on{QCoh}^{J_{\infty}G}(J_{\infty}\g^*)$
 with the chiral moment map $(J_{\infty}\mu_{R})^*$.
 
 For a Poisson algebra object $V$ in $\on{QCoh}^{J_{\infty}G}(J_{\infty}\g^*)$,
 set
 \begin{align*}
C^{cl}(\affg^{cl},V)=V\* \overline{\bigwedge}^{\semiinf+\bullet}(\g),
\end{align*}
where $\overline{\bigwedge}^{\semiinf+\bullet}(\g)$ is the Poisson vertex superalgebra generated by odd elements
$\bar \psi_i$, $\bar \psi_i^*$,
$i=1,\dots,\dim \g$,
with $\lam$-brackets (see \cite{Kaclecture17})
$\{(\bar \psi_i)_{\lam} \bar \psi_j^*\}=\delta_{ij}$,
$\{(\bar \psi_i)_{\lam}\bar \psi_j^*\}=\{(\bar \psi_i^*)_{\lam}\bar \psi_j^*\}=0$.
Set
\begin{align*}
Q^{cl}=\sum_{i=1}^{\dim \g}\mu_V(x_i)\* \bar \psi_i^*- 1\* \frac{1}{2}c_{ij}^k\bar \psi_i^*\bar \psi_j^*\bar \psi_k\in C(\affg^{cl},V).
\end{align*}
Then $(Q^{cl}_{(0)})^2=0$,
and $(C(\affg^{cl},V), Q_{(0)}^{cl})$ is a differential graded Poisson vertex algebra,
where 
its cohomological grading is defined by
$\deg a=$ for $a\in V$,
$\deg \partial^n\bar \psi_i^*=1$,
$\deg \partial^n \bar \psi_i=-1$,
$n\geq 0$.
The corresponding cohomology 
$H^{\semiinf+\bullet}(\affg^{cl}, V)$
is naturally a Poisson vertex superalgebra.

Let $X\cong G\times S$ be as above.
Then $J_{\infty}X\cong J_{\infty}G\times J_{\infty}S$.
Assume further that
the chiral moment map $J_{\infty}\mu_X:J_{\infty}X\ra J_{\infty}\g^*$ is flat.
For a vertex Poisson algebra object $V$ in 
$\on{QCoh}^{J_{\infty}G}(J_{\infty}\g^*)$,
$\mc{O}(J_{\infty}X^{op})\* V$
is a vertex Poisson algebra object in 
$\on{QCoh}^{J_{\infty}G}(J_{\infty}\g^*)$
with chiral moment map
$\hat \mu:J_{\infty}X^{op}\times \Spec V\ra J_{\infty}\g^*$,
$(x,y)\mapsto -J_{\infty}\mu_X(x)+\mu_V^*(y)$.
We have
$\hat\mu^{-1}(0)\cong J_{\infty}X\times_{J_{\infty}\g^*}\Spec V\cong 
J_\infty G\times (J_{\infty}S\times_{J_{\infty}\g^*}\Spec V) $.
It follows in the same way as  \eqref{eq:BRST-cl}
that 
\begin{align}
H^{\semiinf+\bullet}(\affg^{cl},  \mc{O}(J_{\infty}X^{op})\* V)
\cong \left(\mc{O}(J_{\infty}S)\*_{\mc{O}(J_{\infty*}\g^*)} V\right)\* H^\bullet(\g,\C).
\label{eq:cintg}
\end{align}
In particular,
if $R$ is an Poisson algebra object in $\on{Coh}^G(\g^*)$,
then
\begin{align}
H^{\semiinf+\bullet}(\affg^{cl},  \mc{O}(J_{\infty}X^{op}\times J_{\infty} X'))
\cong \mc{O}\left(J_{\infty}\left((X^{\op}\times X)/\!/\!/ \Delta(G)
\right)\right)\* H^\bullet(\g,\C).
\end{align}
More generally,
the same argument gives that 
\begin{align}
H^{\semiinf+\bullet}(\affg^{cl},  \mc{O}(J_{\infty}X^{op})\otimes M)
\cong
\left(\mc{O}(J_{\infty}S)\*_{\mc{O}(J_{\infty}\g^*)}M\right)\* H^\bullet(\g,\C)
\label{eq:cintgm}
\end{align}
for any $M\in \on{Coh}^{J_{\infty}G}(J_\infty\g^*)$.

%Let $G$, $\g$ be as in introduction,
%and let
Let $\kappa$ be
an invariant symmetric bilinear form  on $\g$.
Denote by $\on{KL}_{\kappa}$  the full subcategory of 
the category of  graded $V^\kappa(\g)$-modules consisting of objects $M$ on which 
$\g[[t]]t$ acts locally nilpotently and
$\g$ acts locally finitely.
For $\lam\in P_+$, 
let 
\begin{align}
\mathbb{V}_\lam^\kappa=U(\affg_{\kappa})\*_{U(\g[t]\+ \C \mathbf{1})}V_{\lam}\in \KL_{\kappa},
\label{eq:Weyl-module}
\end{align}
where
$V_{\lam}$ is the irreducible finite-dimensional representation of $\g$ with highest weight $\lam$
that is regarded as a $\g[t]\+ \C \mathbf{1}$-module on which $\g[t]t$ acts trivially and 
$\mathbf{1}$ acts as the identity.
Note that $\mathbb{V}_0^\kappa\cong V^\kappa(\g)$ as $\affg_{\kappa}$-modules.
%Denote by $\mathbb{L}_{\lam}^\kappa$ the unique simple graded quotient of $\mathbb{V}_\lam^\kappa$.
%Let $D:\KL_{\kappa}\ra \KL_{\kappa}$ denote the duality functor.

Let $\on{KL}_{\kappa}^{ord}$ be the 
full subcategory of $\on{KL}_{\kappa}$
consisting of objects 
that are positively graded and each homogenous subspaces are finite-dimensional.
For $M\in \on{KL}_{\kappa}^{ord}$,
the Li filtration $F^\bullet M$ is separated.
Note that any object of $\on{KL}_{\kappa}$
is an inductive limit of objects of $\on{KL}_{\kappa}^{ord}$.

 A  {\em vertex algebra object}
in $\on{KL}_{\kappa}$
is a vertex algebra $V$
equipped with a vertex algebra homomorphism
$\mu_{V}:V^{\kappa}(\g)\ra V$ such that $V$
is a direct sum of objects in $\on{KL}_{\kappa}^{ord}$
 as a $V^{\kappa}(\g)$-module.
 The map $\mu_{V}$ is called the {\em chiral quantum moment map}.

Let $V$ be a  vertex algebra object 
 in $\on{KL}_{\kappa}$.
 Then $V$ is separated,
 and
$R_V$ and $\gr V$ are 
a Poisson algebra object in 
$\on{QCoh}^{G}(\g^*)$ and a Poisson vertex algebra object in $\on{QCoh}^{J_{\infty} G}(J_{\infty}\g^*)$,
respectively.
A {\em conformal vetex algebra object} 
in  $\on{KL}_{\kappa}$ is 
a vertex algebra object $V$ in $\on{KL}_\kappa$ which  is conformal 
with conformal vector $\omega_V$
and
 $\mu_V(x)$ is primary with respect to $\omega_V$ for all $x\in \g\subset V^\kappa(\g)$,
 that is,
 $(\omega_V)_{(n)}\mu_V(x)=0$ for $n\geq 2$.
%and the grading of $V$
%is given by $(\omega_V)_{(1)}$.

 If $V$ is a vertex algebra object in $\KL_\kappa$,
 $V^{op}$ is also an vertex algebra object in $\KL_\kappa$
 with $\mu_{V^{op}}(x)=-\mu_{V}(x)$ for $x\in \g$.

Let $\kappa_\g$ be the Killing form of $\g$.
For a vertex algebra  object in $\KL_{-\kappa_\g}$,
let 
\begin{align*}
C(\affg_{-\kappa_\g},V)=V\* \bw{\semiinf+\bullet}(\g),
\end{align*}
where ${\bigwedge}^{\semiinf+\bullet}(\g)$ is the vertex superalgebra generated by odd elements
$ \psi_i$, $ \psi_i^*$,
$i=1,\dots,\dim \g$,
with OPEs %$\lam$-brackets
$\psi_i(z)\psi_j^*(w)\sim \frac{\delta_{ij}}{z-w}$,
$\psi_i(z)\psi_j(w)\sim \psi_i^*(z)\psi_j^*(w)\sim 0$.
%$[( \psi_i)_{\lam}  \psi_j^*]=\delta_{ij}$,
%$[( \psi_i)_{\lam} \psi_j^*]=[( \psi_i^*)_{\lam} \psi_j^*]=0$.
The vertex algebra $\bw{\semiinf+\bullet}(\g)$
is conformal
of central charge 
$-2\dim \g$.
We have
\begin{align*}
\bw{\semiinf+\bullet}(\g)=\bigoplus_{\Delta\in \Z_{\geq 0}}\bw{\semiinf+\bullet}(\g)_{\Delta},\quad
\bw{\semiinf+\bullet}(\g)_0\cong \bw{\bullet}(\g^*).
\end{align*}
Set
\begin{align*}
Q=\sum_{i=1}^{\dim \g}\mu_V(x_i)\*  \psi_i^*- 1\* \frac{1}{2}c_{ij}^k \psi_i^* \psi_j^* \psi_k\in C(\affg_{-\kappa_\g},V).
\end{align*}
Then $(Q_{(0)})^2=0$,
and $(C(\affg_{-\kappa_\g},V), Q_{(0)})$ is a differential graded  vertex algebra,
where 
its cohomological  grading is defined by
$\deg a=0$ for $a\in V$,
$\deg \partial^n\psi_i^*=1$,
$\deg \partial^n \psi_i=-1$,
$n\geq 0$.
The corresponding cohomology 
$H^{\semiinf+\bullet}(\affg_{-\kappa_\g}, V)$
is the semi-infinite $\affg_{-\kappa_\g}$-cohomology
with coefficients in $V$,
and 
is naturally a  vertex superalgebra.
If $V$ is conformal with central charge $c_V$,
then $H^{\semiinf+\bullet}(\affg_{-\kappa_\g}, V)$
is conformal with central charge $c_V-2\dim \g$.

More generally,
if $M$ is an object in $\KL_{-\kappa_\g}$,
then 
$C(\affg_{-\kappa_\g},M)=M\* \bw{\semiinf+\bullet}(\g)$
is a 
$C(\affg_{-\kappa_\g},V^{-\kappa_\g}(\g))$-module.
Hence,
 $(C(\affg_{-\kappa_\g},M), Q_{(0})$ is naturally a differential graded vertex algebra module over 
 $C(\affg_{-\kappa_\g},V^{-\kappa_\g}(\g))$.
 The corresponding cohomology
 $H^{\semiinf+\bullet}(\affg_{-\kappa_\g}, M)$ is naturally a module over 
 $H^{\semiinf+\bullet}(\affg_{-\kappa_\g}, V^{-\kappa_\g}(\g))$.

Set 
\begin{align}
\kappa^*=-\kappa-\kappa_\g.
\label{eq:dual-kappa}
\end{align}
For $M\in \KL_{\kappa}$,
$N\in \KL_{\kappa^*}$,
$M\* N$ is an object of $\KL_{-\kappa_\g}$ with respect to the diagonal action of $\affg_{-2\kappa}$.
\begin{Th}\label{Th:vanishing-BRST-g}
Let $V$ 
be a vertex algebra object in $\on{KL}_\kappa$,
which is a strict quantization of 
$X$.
Assume that
(1) there exists a closed subscheme $S$ of $X$
such that the action map $G\times S\ra X$, $(g,s)\ra gs$,
is an isomorphism of $G$-schemes,
(2) the chiral moment map $\gr\mu_{V}^*;J_{\infty}X\ra J_{\infty}\g^*$ is flat.
Then,
\begin{align*}
\gr H^{\semiinf+\bullet}(\affg_{-\kappa_\g},V^{op}\* M)\cong
 \left(\mc{O}(J_{\infty}S)\*_{\mc{O}(J_{\infty*}\g^*)} \gr M\right)\* H^\bullet(\g,\C)
\end{align*}
for any $M\in \KL_{\kappa^*}^{ord}$.
If $W$ is 
 a  vertex algebra object $W$ in $\KL_{\kappa^*}$,
the vertex algebra
$H^{\semiinf+\bullet}(\affg_{-\kappa_\g}, V^{op}\* W)$
is separated,
and we have
\begin{align*}
SS(H^{\semiinf+0}(\affg_{-\kappa_\g}, V^{op}\* W))\cong J_{\infty}S\times_{J_\infty \g^*}SS(W)
\\\text{and }\quad \tilde{X}_{H^{\semiinf+0}(\affg_{-\kappa_\g}, V^{op}\* W))}\cong S\times_{ \g^*}\tilde{X}_W.
\end{align*}
\end{Th}
\begin{proof}
Note that 
$(C(\affg_{-\kappa_\g},V^{op}\* W), Q_{(0)})$
is a direct sum of finite-dimensional subcomplexes 
since $V^{op}\* W$  is a direct sum of objects in $\KL_{-\kappa_\g}^{ord}$.
Note also that
$\gr C(\affg_{-\kappa_\g},V^{op}\* M)\cong C(\affg^{cl},\gr V^{op}\* \gr M)
\cong C(\affg^{cl},\mc{O}(J_{\infty}X)\* \gr M)$.
Hence there is a spectral sequence $E_r\Rightarrow H^{\semiinf+\bullet}(\affg_{-\kappa_\g}, V^{op}\* M)$ such that 
\begin{align*}
\bigoplus_q E_{1}^{n-q,q}\cong H^{\semiinf+n}(\affg^{cl},\mc{O}(J_{\infty}X)\* \gr M)
\cong 
 \left(\mc{O}(J_{\infty}S)\*_{\mc{O}(J_{\infty*}\g^*)} \gr M\right)\* H^n(\g,\C).
\end{align*}
Because 
$\mc{O}(J_{\infty}S)$ is a trivial $J_{\infty}G$-module,
elements of $\mc{O}(J_{\infty}S)\*_{\mc{O}(J_{\infty*}\g^*)} \gr M$
 are  $\g$-invariant.
It follows that 
 $d_r: E_r\ra E_{r+1}$  is identically zero for all $r\geq 1$.
 Therefore,
 \begin{align*}
\gr H^{\semiinf+\bullet}(\affg_{-\kappa_\g},V^{op}\* M)\cong
 \left(\mc{O}(J_{\infty}S)\*_{\mc{O}(J_{\infty*}\g^*)} \gr M\right)\* H^\bullet(\g,\C).
\end{align*}
This completes the proof.
\end{proof}

We have a vertex algebra homomorphism
$V^0(\g)\ra C(\affg_{-\kappa_\g},V^{-\kappa_g}(\g))$,
$x_i\mapsto \widehat{x_i}:=Q_{(0)}\psi_i$,
$i=1,\dots, \dim \g$.
Thus
$C(\affg_{-\kappa_\g},M)$ 
is a $V^0(\g)$-module 
for $M\in \on{KL}_{-\kappa_\g}$.
Set
$$C(\affg_{-\kappa_g},\g,M)=\{c\in C(\affg_{-\kappa_\g},M)\mid
(\hat{x}_i)_{(0)}c=(\psi_i)_{(0)}c=0,\ \forall i=1,\dots,\dim \g\}.$$
Then 
$C(\affg_{-\kappa_\g},\g,M)$ is a subcomplex of 
$C(\affg_{-\kappa_\g},M)$,
and 
the cohomology of the complex
$(C(\affg_{-\kappa_\g},\g,M),Q_{(0)})$
is 
the  semi-infinite cohomology
$H^{\frac{\infty}{2}+\bullet}(\affg_{-\kappa_\g},\g,M)$
{\em relative to} $\g$
with coefficients in $M$.
If $V$ is a vertex algebra object in $\KL_{-\kappa_\g}$,
$H^{\frac{\infty}{2}+\bullet}(\affg_{-\kappa_\g},\g,V)$ is naturally a vertex superalgebra.

\begin{Pro}\label{Pro:relative-coh}
For $M\in \on{KL}_{-\kappa_\g}$,
\begin{align*}
H^{\semiinf+\bullet}(\affg_{-\kappa_\g},M)\cong H^{\semiinf+\bullet}(\affg_{-\kappa_\g},\g,M)\* H^\bullet(\g,\C).
\end{align*}
\end{Pro}
\begin{proof}
We may assume that
$M$ is finitely generated.
Consider the 
Hochschild-Serre spectral sequence
$E_r\Rightarrow H^{\semiinf+i}(\affg_{-\kappa_\g},M)$
for the subalgebra
$\g\subset \affg_{-\kappa_\g}$.
In the first term we have
\begin{align*}
E_1^{p,q}\cong C^p(\affg_{-\kappa_\g},\g,M)\* H^q(\g,\C).
\end{align*}
 In the second term we have
\begin{align}
E_2^{p,q}\cong H^p(\affg_{-\kappa_\g},\g,M)\* H^q(\g,\C).
\label{eq:E2-dec}
\end{align}
We can therefore represent classes in 
$E_2^{p,q}$
 as sums of tensor products $\omega_1\* \omega_2$ of a cocycle  $\omega_1$
 in $C^p(\affg_{-\kappa_\g},\g,M)$ representing a class in $H^p(\affg_{-\kappa_\g},\g,M)$ and a cocycle 
$\omega_2$ representing a class in $H^q(\g,\C)$.
Applying the differential to this class, 
we find that it is identically equal to zero because $\omega_1$
is $\g$-invariant. Therefore all the classes in $E_2$ survive. 
Moreover, all of the vectors of the two factors in the decomposition
\eqref{eq:E2-dec}
lift canonically to the cohomology 
$H^{\semiinf+\bullet}(\affg_{-\kappa_\g},M)$,
and so we obtain the desired isomorphism.
\end{proof}

The following assertion follows immediately 
from Theorem \ref{Th:vanishing-BRST-g} and Proposition~\ref{Pro:relative-coh}.
\begin{Th}\label{Th:vanishing-BRST-g-rel}
Let $V$, $S$,
$M$,
be as in Theorem \ref{Th:vanishing-BRST-g}.
Then
 we have
\begin{align*}
\gr H^{\semiinf+i}(\affg_{-\kappa_\g},\g, V^{op}\* M)\cong
\begin{cases}
 \mc{O}(J_{\infty}S)\*_{\mc{O}(J_{\infty*}\g^*)} \gr M&\text{for }i=0,\\
 0&\text{otherwise}.
 \end{cases}
\end{align*}
In particular,
\begin{align*}
SS(H^{\semiinf+0}(\affg_{-\kappa_\g},\g, V^{op}\* W))\cong J_{\infty}S\times_{J_\infty \g^*}SS(W),
\\\text{and }\quad \tilde{X}_{H^{\semiinf+0}(\affg_{-\kappa_\g},\g, V^{op}\* W))}\cong S\times_{ \g^*}\tilde{X}_W,
\end{align*}
for a vertex algebra object $W$ in $\on{KL}_{\kappa^*}$.
\end{Th}

 \begin{Pro}\label{Pro:vanishing-relative}
 Let $M\in \on{KL}_{-\kappa_\g}$.
 Suppose that
 $M$ is free over $U(t^{-1}\g[t^{-1}])$ and 
 cofree over $U(t\g[t])$.
 Then 
 $H^{\semiinf+i}(\affg_{-\kappa_\g},\g,M)=0$ for $i\ne 0$.
 \end{Pro}
 \begin{proof}
 Since $M$ is free over $U(t^{-1}\g[t^{-1}])$,
 we obtain that 
  $H^{\semiinf+i}(\affg_{-\kappa_\g},M)=0$ for $i<0$
  by considering the  Hochschild-Serre spectral sequence
  for the subalgebra
  $t^{-1}\g[t^{-1}]\subset \affg_{-\kappa_\g}$,
  see \cite[Theorem 2.3]{Vor93}.
  Therefore by Proposition \ref{Pro:relative-coh},
  $H^{\semiinf+i}(\affg_{-\kappa_\g},\g,M)=0$ for $i<0$.
  Next, consider the 
  Hochschild-Serre spectral sequence
  $E_r\Rightarrow H^{\semiinf+\bullet}(\affg_{-\kappa_\g},M)$
  for the subalgebra
  $\g[t]\subset \affg_{-\kappa_\g}$ as in \cite[Theorem 2.2]{Vor93}.
 By definition,
 we have
$  
E_1^{p,q}=H^q(\g[t], M\* \bw{-p}(\affg/\g[t]))
$.
 Since $M$ is cofree over
   $U(t\g[t])$,
   so is $M\* \bw{-p}(\affg/\g[t])$.
Hence 
\begin{align}
E_1^{p,q}\cong \left(M\* \bw{-p}(\affg/\g[t])\right)^{\g[t]}\* H^q(\g,\C).
\end{align}
It follows that
$E_1^{p,q}=0$ for $q>\dim G$,
and therefore 
 $H^{\semiinf+i}(\affg_{-\kappa_\g},M)=0$ for  $i>\dim G$.
By Proposition \ref{Pro:relative-coh},
this implies that
 $H^{\semiinf+i}(\affg_{-\kappa_\g},\g,M)=0$ for all $i>0$.

 \end{proof}

The form
\begin{align}
\kappa_c=-\frac{1}{2}\kappa_\g
\end{align}
 is called the {\em critical level}
for $\g$.
For $M,N\in \KL_{\kappa_c}$,
we have $M\* N\in \KL_{-\kappa_\g}$.
We set
\begin{align}
M\circ N:=H^{\semiinf+\bullet}(\affg_{-\kappa_\g},\g, M\* N).
\label{eq:MTproduct}
\end{align}

\section{Feigin-Frenkel center}
\label{Sectioin:Feigin-Frenkel center}
Let
%$\kappa_c=-\frac{1}{2}\kappa_\g$,
%and 
%let
 $$\mf{z}(\affg)=Z(V^{\kappa_c}(\g))=V^{\kappa_c}(\g)^{\g[t]},$$
the Feigin-Frenkel center of the vertex algebra $V^{\kappa_c}(\g)$ (\cite{FeiFre92}).
We have
 isomorphisms (\cite{EisFre01,Fre05})
\begin{align*}
\tilde{X}_{\mf{z}(\affg)}\cong  \g^*/\!/G,\quad
SS(\mf{z}(\affg))\cong J_{\infty}( \g^*/\!/G),\quad
\on{Zhu}(\mf{z}(\affg))\cong \mc{Z}(\g),
\end{align*}
where $ \mc{Z}(\g)$ denotes the center of $U(\g)$.
The grading of $V^{\kappa_c}(\g)$ induces a grading of
$\mf{z}(\affg)$:
 $\mf{z}(\affg)=\bigoplus_{\Delta\in \Z_{\geq 0}}\mf{z}(\affg)_{\Delta}$,
$\mf{z}(\affg)_{\Delta}=\mf{z}(\affg)\cap V^{\kappa}(\g)_{\Delta}$.
Let $d_1,\dots, d_{\on{rk}\g}$ be the 
exponents of $\g$,
where $\on{rk}\g$ is the rank of $\g$.
Choose homogeneous strong generators $P_1,\dots, P_{\on{rk}\g}\in \mf{z}(\affg)$ with $P_i\in \mf{z}(\g)_{d_i+1}$.
Their images
form homogeneous generators 
of $R_{\mf{z}(\affg)}=\mc{O}(\g^*)^G$,
$\gr \mf{z}(\affg)=\mc{O}(J_{\infty}\g^*)^{J_{\infty}G}$,
$\on{Zhu}(\mf{z}(\affg))=\mc{Z}(\g)$,
 respectively.
 We use the notation
 \begin{align*}
P_i(z)=\sum_{n\in \Z}P_{i,(n)}z^{-n-1}=\sum_{n\in \Z}P_{i,n}z^{-n-d_i-1},
\end{align*}
so that the operator $P_{i,n}$ has degree $-n$ on $V^{\kappa_c}(\g)$.

Let  $M\in \on{KL}_{\kappa_c}$.
Then
%The Feigin-Frenkel center 
$\mf{z}(\affg)$ naturally acts on $M$,
 and hence, 
 $M$ can be regarded as a module of the polynomial ring
 \begin{align}
\LZ=\C[P_{i,n};i=1,\dots,\on{rk}\g,n\in \Z].
\label{eq:LZ}
\end{align}
Set
\begin{align}
\LZ_{(\geq 0)}=\C[P_{i,(n)};i=1,\dots,\on{rk}\g,n\geq 0],\quad \LZ_{(<0)}=\C[P_{i,(n)};i=1,\dots,\on{rk}\g,n< 0],
\label{eq:Z(<0)}\\
\LZ_{>0}=\C[P_{i,n};i=1,\dots,\on{rk}\g,n>0],\quad \LZ_{<0}=\C[P_{i,n};i=1,\dots,\on{rk}\g,n< 0],\\
\LZ_0=\C[P_{i,0};i=1,\dots,\on{rk}\g].
\end{align}
Then 
$\LZ_{(<0)}\subset \LZ_{<0}$,
$\LZ_{(\geq 0)}\supset \LZ_{>0}$,
and 
$\LZ_{(<0)}\cong \mf{z}(\affg)$. 
We have the isomorphism
$$o:\LZ_0\isomap \Zhu(\mf{z}(\affg))\isomap \mc{Z}(\g).$$

Let $P_+$ be the set of dominant weights of $\g$ as in Introduction.
For $\lam\in P_+$,
set 
\begin{align}
\mathbb{V}_{\lam}=\mathbb{V}^{\kappa_c}_{\lam}\label{eq:critical-Weyl}
\end{align}
(see \eqref{eq:Weyl-module}).
We regard $\mathbb{V}_{\lam}$ a $\Z_{\geq 0}$-graded 
$\affg_{\kappa_c}$-module
by giving zero degree  to the highest weight vector.
Let  $\mathbb{L}_\lam$ be the unique simple graded quotient of $\mathbb{V}_{\lam}$,
and 
let $\chi_{\lam}:\LZ\ra \C$ be the evaluation at $\mathbb{L}_{\lam}$.
Since $\mathbb{L}_{\lam}$
is graded,
$\chi_{\lam}(P_{i,n})=\begin{cases}
\gamma_{\lam}(o(P_{i,n}))&\text{for }n=0,\\
0&\text{for }n\ne 0,\end{cases}$
 where $\gamma_{\lam}:\mc{Z}(\g)\ra \C$ is the evaluation at $V_{\lam}$.
   Let $\KL^{[\lam]}$ be the 
full subcategory of $\KL$ consisting of objects $M$ such that 
$\chi_{\lam}(P_{i,0})m=\gamma_{\lam}(o(P_{i,0}))m$ for all $i$.
Then
\begin{align}
\KL=\bigoplus_{\lam\in P_+}\KL^{[\lam]}.
\label{eq:dec-KL-critical}
\end{align}
For $M\in \KL$,
let $M=\bigoplus_{\lam\in P_+}M_{[\lam]}$, $M_{[\lam]}\in \KL_{[\lam]}$,
be the corresponding decomposition.
Clearly,
$\mathbb{V}_{\lam},\mathbb{L}_{\lam}\in \in \KL^{[\lam]}$
and
any simple object in $\KL^{[\lam]}$
is isomorphic to 
 $\mathbb{L}_{\lam}[d]$ for some $d\in \C$, 
where $\mathbb{L}_{\lam}[d]$ denotes the 
$\affg_{\kappa_c}$-module $\mathbb{L}_{\lam}$
whose grading is shifted as
$(\mathbb{L}_{\lam}[d])_{d'}=
(\mathbb{L}_{\lam})_{d+d'}$.

We denote by
$\LZ\Mod$
the category of positive energy
representations of the vertex algebra $\mf{z}(\affg)$,
that is, the category
of $\LZ$-modules $M$ that admits a grading
$M=\bigoplus_{d\in h+\Z_{\geq 0}}M_{d}$, $h\in \C$,
such that 
$P_{i,n}M_d\subset M_{d-n}$
for all $i$ and $n$.
For $M\in \LZ\Mod$,
we set
$\ch M=\sum_{d\in \C}q^d \dim M_d$
if $\dim M_d<\infty$ for all $d$.
Let $\LZ\Mod^{[\lam]}$ be the 
full subcategory of $\LZ\Mod$ consisting of objects $M$ 
on which $P_{i,0}$ acts as $\chi_{\lam}(P_{i,0})$ for all $i=1,\dots,
\on{rk}\g$,
and set $\LZ\Mod_{reg}=\bigoplus_{\lam\in P_+}
\LZ\Mod^{[\lam]}$.

% We denote by $o(P_i)$ the image of $P_i$ in 
% $\on{Zhu}(\mf{z}(\affg))=\mc{Z}(\g)$,
% so that $\mc{Z}(\g)=\C[o(P_1),\dots, o(P_r)]$.

%
%Let  $M\in \on{KL}_\kappa$.
%Then
% the Poisson module
% $\bar M=M/F^1 M$ over the Poisson algebra $\mc{O}(\g^*)$
% is an object of 
% $\on{QCoh}^{G}(\g^*)$.
% Also,
% the Poisson vertex module
% $\gr M$ over $\mc{O}(J_{\infty}\g^*)$
% is an object of  $\on{QCoh}^{J_{\infty}G}(J_{\infty}\g^*)$
%  (see e.g.\ \cite{Ara09b,A2012Dec}).

%
%We have the right exact functor 
%\begin{align*}
%\on{KL}_{\kappa}\ra \mathscr{C}oh^{J_{\infty}G}(J_{\infty}\g^*)
%\end{align*}
%that sends $M\in \on{KL}_\kappa$ to $\gr M=\bigoplus_{p\geq 0}F^p M/F^{p+1}M $,
%where $F^{\bullet}M$  denotes the Li's filtration on the $V^{\kappa}(\g)$-module $M$.
%
%We have the right exact functor 
%\begin{align*}
%\on{KL}_{\kappa}\ra \mathscr{C}oh^G(\g^*)
%\end{align*}
%that sends $M\in \on{KL}_\kappa$ to $M/F^1M$.
%
%
%
%
%
%
%
%Let $\mathscr{HC}(\g)$ be the category of the Harish-Chandra bimodules over $U(\g)$,
%that is, the category of $G$-modules equipped with an equivariant $U(\g)$-action.
%Since $\on{Zhu}(V^{\kappa}(\g))$ is canonically isomorphic to $U(\g)$,
%We have the right exact functor 
%\begin{align*}
%\on{KL}_{\kappa}\ra \mathscr{HC}(\g)
%\end{align*}
%that sends $M\in \on{KL}_\kappa$ to the Frenkel-Zhu's bimodule $\on{Zhu}(M)$ (\cite{FreZhu92}).
%
%

\section{Chiral differential operators on $G$}
\label{Section:cdoonG}
There are two commuting Hamiltonian actions on the cotangent bundle $T^*G=G\times \g^*$
of $G$
given by
$g.(h,x)=(hg^{-1},gx)$ and 
$g.(h,x)=(g h,x)$.
The corresponding moment maps
are 
\begin{align*}
\mu_L:G\times \g^*\ra \g^*,\quad (g,x)\mapsto x, \quad \text{and}\quad
\mu_R:G\times \g^*\ra \g^*,\quad (g,x)\mapsto g.x,
\end{align*}
respectively.
The algebra
$\mc{O}(T^*G)$ 
can be considered
as a Poisson algebra object in $\on{Coh}^G(\g^*)$
with the moment map $\mu_L$
or $\mu_R$.
Note that the isomorphism
$T^*G\isomap T^*G$,
$(g,x)\mapsto (g^{-1},gx)$,
exchanges $\mu_L$ and $\mu_R$.

We have 
\begin{align}
\mc{O}(T^*G)^{op}\cong \mc{O}(T^*G),\quad 
x\mapsto -x\ (x\in \g\subset \mc{O}(\g^*)),\
f\mapsto f\ (f\in \mc{O}(G)),
\end{align}
and so we do not 
distinguish
between $T^*G$ and $(T^*G)^{op}$.
\begin{Lem}
\begin{enumerate}
\item For any  Poisson algebra object $R$ in $\on{Coh}^G(\g^*)$,
we have
$$H^{\semiinf+0}(\g,\mc{O}(T^*G)\* R)\cong R.$$
\item 
For any  vertex Poisson algebra object $V$ in $\on{Coh}^{J_{\infty}G}(J_{\infty}\g^*)$,
we have
$$H^{\semiinf+0}(\affg^{cl},\mc{O}(J_{\infty}T^*G)\* V)\cong V.$$

\end{enumerate}

\end{Lem}
\begin{proof}
By \eqref{eq:BRST-cl},
we have $H^{\semiinf+0}(\g,\mc{O}(T^*G)\* R)\cong \mc{O}(\g^*)\*_{\mc{O}(\g^*)}R\cong R$.
Similarly,
by \eqref{eq:cintg},
$H^{\semiinf+0}(\affg^{cl},\mc{O}(J_\infty T^*G)\* V)\cong \mc{O}(J_{\infty}\g^*)\*_{\mc{O}(J_{\infty}\g^*)}V\cong V$.
\end{proof}

%The action 
%$x\mapsto x_L:=(J_{\infty}\mu_L)^*$
% 
% Let $x\mapsto x_L$
% and $x\mapsto x_R$
% denote the action of $x\in J_{\infty}\g$
% on $\mc{O}(J_{\infty}G)$ as  the left invariant vector fields and  the right invariant vector fields,
% respectively.

Let $\kappa$ be an invariant symmetric bilinear form of $\g$ as before.
Let 
\begin{align}
\Dch_{G,\kappa}=U(\affg_{\kappa})\otimes _{U(\fing[[t]]\+ \C \1)}\mathcal{O}(J_{\infty}G),
\label{eq:cdo-def}
\end{align}
where $\fing[[t]]=J_{\infty}\g$ acts on $\mathcal{O}(J_{\infty}G)$ 
via the Poisson vertex algebra homomorphism
$(J_{\infty}\mu_L)^*: \mc{O}(J_{\infty}\g^*)\ra \mathcal{O}(J_{\infty}G)$
and $\1$ acts as the identity.
There exists a unique vertex algebra structure on $\mc{D}_{G,\kappa}^{ch}$ 
such that 
\begin{align*}
\pi_L\colon&V^{\kappa}(\g)\hookrightarrow \Dch_{G,\kappa},\quad  u|0\ket\mapsto  u\*\mathbf{1}_{J_{\infty}G}
\quad (u\in U(\affg_{\kappa})),\\
&\mathcal{O}(J_{\infty}G)\hookrightarrow \Dch_{G,\kappa},\quad f\mapsto 1\* f,
\end{align*}
are homomorphisms of vertex algebras,
and 
\begin{align*}
x(z)f(w)\sim \frac{1}{z-w}(x_Lf)(w)\quad (x\in \g\subset V^{\kappa}(\g),
\quad f\in \mc{O}(G)\subset \mc{O}(J_{\infty}G))
\end{align*}
(\cite{GorMalSch01,ArkGai02}).
The vertex algebra $\Dch_{G,\kappa}$ is 
called 
the {\em algebra of chiral differential operators} (cdo) on $G$ at level $\kappa$,
which is the special case
of the chiral differential operators
on a smooth algebraic variety introduced  independently by Malikov,
Schechtman and Vaintrob
\cite{MalSchVai99} and Beilinson and Drinfeld \cite{BeiDri04}.
The cdo $\Dch_{G,\kappa}$ 
is a strict chiral quantization of the cotangent bundle $T^*G$ to $G$.
In particular,
$\Dch_{G,\kappa}$
is simple for any $\kappa$ by Theorem \ref{Th:AMoreau}.

We have 
$\Dch_{G,\kappa}\isomap (\Dch_{G,\kappa})^{op}$,
$x\mapsto -x$, $x\in \g$,
$f\mapsto f$, $f\in \mc{O}(G)$,
and so we do not 
distinguish
between $\Dch_{G,\kappa}$ and $(\Dch_{G,\kappa})^{op}$.

According to \cite{GorMalSch01,ArkGai02}, there is a vertex algebra embedding
\begin{align*}
\pi_R:V^{\kappa^*}(\g)\hookrightarrow  \mc{D}_{G,\kappa}^{ch}
\end{align*}
(see \eqref{eq:dual-kappa}),
and
$V^{\kappa}(\g)$ and $V^{\kappa^*}(\g)$ form a dual pair in $\mc{D}_{G,\kappa}^{ch}$,
i.e.,
\begin{align}
V^{\kappa}(\g)= \on{Com}(V^{\kappa^*}(\g),\Dch_{G,\kappa})=(\Dch_
{G,\kappa})^{\pi_R(\g[t])},
\label{eq:piL}
\\
V^{\kappa^*}(\g)=\on{Com}(V^{\kappa}(\g),\Dch_{G,\kappa})=(\Dch)^{\pi_L(\g[t])}.
\label{eq:piR}
\end{align}

Suppose that $\kappa\ne \kappa_c$,
%where
%\begin{align*}
%\kappa_c:=-\frac{1}{2}\kappa_g.
%\end{align*}
and let $\omega_L$ and  $\omega_R$ be the Sugawara conformal vector of
$V^{\kappa}(\g)$ and $V^{\kappa^*}(\g)$,
respectively.
Then
\begin{align*}
\omega_{\Dch_{G,\kappa}}=\omega_L+\omega_R
\end{align*}
gives a conformal vector  of $\mc{D}_{G,\kappa}^{ch}$ of central charge $2\dim G$.
In fact,
the left-hand side
 makes sense even at the critical level $\kappa=\kappa_c$,
and $\omega_{\Dch_{G,\kappa}}$ gives a  conformal vector of $\mc{D}_{G,\kappa_c}^{ch}$ of central charge $2\dim G$
for all $\kappa$
(\cite{GorMalSch00}).
We have
\begin{align*}
\mc{D}_{G,\kappa}=\bigoplus_{n\in \Z_{\geq 0}}(\mc{D}_{G,\kappa})_{\Delta},
\quad (\mc{D}_{G,\kappa})_{0}\cong \mc{O}(G)\cong \bigoplus_{\lam\in P_+}V_{\lam}\* V_{\lam^*},
\end{align*}
where $\lam^*=-w_0(\lam)$ and $w_0$ is the longest element of the Weyl group $W$ of $\g$.

We have  the canonical isomorphism
  (\cite{AraCheMal08})
\begin{align}
 \on{Zhu}\Dch_{G,\kappa}\cong \D_{G},
\label{eq:cdo-zhu}
\end{align}
where $\D_G$ denotes the algebra of global differential operators on $G$.
Under this identification, the filtration of $\on{Zhu}\Dch_{G,\kappa}$
coincides with the natural filtration on $\D_G$
and 
the map
\eqref{eq:surj-RV} for $V^\kappa(\g)$ recovers the well-known isomorphism
$\mc{O}(T^*G)\isomap \gr \D_G$.

The algebra
homomorphism
between Zhu's algebras
induced by $\pi_L$
and $\pi_R$ are
 the embeddings
$U(\g)\hookrightarrow
 \mc{D}_G$
 defined by
$\g\ni x\mapsto x_L$
and $\g\ni x\mapsto x_R$,
respectively.

For a $\mc{D}_G$-module $M$,
we have
\begin{align*}
\on{Ind}_{\mc{D}_G}^{\Dch_{G,\kappa}}M=U(\affg_{\kappa})\*_{U(\g[[t]]\+ \C \mathbf{1})}\left(\mc{O}(J_{\infty}G)\*_{\mc{O}(G)}M\right)
\end{align*}
as $\affg_{\kappa}$-modules,
where $\g[[t]]\+ \C \mathbf{1}$
acts on
$\mc{O}(J_{\infty}G)\*_{\mc{O}(G)}M$  by
$(u(t))(g\* m)=(u(t) .g)\* m+ g\* (u(0) m)$ for $u\in \g[[t]]$, $f\in \mc{O}(J_{\infty}G)$, $m\in M$.

\begin{Th}[{\cite[Theorem 5.2]{AraCheMal08}}]
The functor
$\on{Ind}_{\mc{D}_G}^{\Dch_{G,\kappa}}:\mc{D}_G\Mod\ra \Dch_{G,\kappa}\Mod$
gives an equivalence of categories.
\end{Th}

\begin{Co}\label{Co:free-cofree-modules}
Any $M\in \Dch_{G,\kappa}\Mod$ is free over $U(t^{-1}\g[t^{-1}])$
and cofree over $U(t\g[t])$.
\end{Co}

Sine
${\kappa}_c^*=\kappa_c$,
we have  a vertex algebra homomorphism
\begin{align}
\pi_L\* \pi_R: V^{\kappa_c}(\g)\otimes V^{\kappa_c}(\g)\ra  \mc{D}_{G,\kappa_c}^{ch}.
\label{eq:bimodule}
\end{align}
Note that 
$$\on{Com}((\pi_L\* \pi_R)(V^{\kappa_c}(\g)\* V^{\kappa_c}(\g)),\Dch_{G,\kappa_c})=
(\Dch_{G,\kappa_c})^{\g[t]\times \g[t]}
.$$

\begin{Lem}\label{Lem:inversection-is-center}
We have
$$\pi_L(V^{\kappa_c}(\g))\cap \pi_R(V^{\kappa_c}(\g))=\pi_L(\mf{z}(\affg))=\pi_R(\mf{z}(\affg))=(\Dch_{G,\kappa_c})^{\g[t]\times \g[t]}.$$
In particular, both $\pi_L$ and $\pi_R$ give the isomorphism
$$\mf{z}(\affg)\isomap  (\Dch_{G,\kappa_c})^{\g[t]\times \g[t]}$$
of vertex algebras.
\end{Lem}
\begin{proof}
By \eqref{eq:piL}
and \eqref{eq:piR},
$\pi_L(\mf{z}(\affg))=\pi_L(V^{\kappa_c}(\g))\cap  (\mc{D}_{G,\kappa_c}^{ch})^{\pi_L(\g[t])}
=\pi_L(V^{\kappa_c}(\g))\cap \pi_R(V^{\kappa_c}(\g))$.
Similarly,
$\pi_R(\mf{z}(\affg))=
\pi_L(V^{\kappa_c}(\g))\cap \pi_R(V^{\kappa_c}(\g))$.
Also,
$(\Dch_{G,\kappa_c})^{\g[t]\times \g[t]}=\pi_L(V^{\kappa_c}(\g))\cap \pi_R(V^{\kappa_c}(\g))$.
This completes the proof.
\end{proof}

%Below we identify $\mf{z}(\affg)$ with$ (\mc{D}_{G}^{ch})^{\g[t]\times \g[t]}$ 
%via $\pi_L$.

Let
$S:\mc{O}(G)\ra \mc{O}(G)$
be the  antipode 
of the Hopf algebra
$\mc{O}(G)$.
The map $S$ induces the 
antipode 
 $J_{\infty}S:\mc{O}(J_\infty G)\ra \mc{O}(J_\infty G)$  of the Hopf algebra
$\mc{O}(J_\infty G)$.
Since $J_{\infty}S\circ \pi_L(x)=\pi_R(x)\circ J_{\infty}S$,
$J_{\infty}S\circ \pi_R(x)=\pi_L(x)\circ J_{\infty}S$ for $x\in J_{\infty}\g$, 
$J_{\infty}S$ extends to the vertex algebra isomorphism
\begin{align*}
\tilde \tau:\Dch_{G,\kappa_c}\isomap \Dch_{G,\kappa_c}
\end{align*}
such that $\tilde \tau(f)=S(f)$, $f\in \mc{O}(G)$,
$\tilde{\tau}\circ \pi_L=\pi_R\circ \tilde{\tau}$,
$\tilde{\tau}\circ \pi_R=\pi_L\circ \tilde{\tau}$.
Thus,
$\tilde{\tau}$ restricts to the involutive  automorphism $\tau$ of the vertex subalgebra
$\mf{z}(\affg)= (\Dch_{G,\kappa_c})^{\g[t]\times \g[t]}$,
and
 we have
\begin{align}
\tau (\pi_L(z))=\pi_R(z),\quad \tau (\pi_R(z))=\pi_L(z)
\label{eq:effect-of-tau-to-FF}
\end{align}
for $z\in \mf{z}(\affg)$.
It follows that
\eqref{eq:bimodule} factors through the vertex algebra homomrophism
\begin{align}
V^{\kappa_{c}}(\g)\otimes_{\mf{z}(\affg)} V^{\kappa_c}(\g)\ra \Dch_{G,\kappa_c},
\label{eq:hom-V-V}
\end{align}
where 
$\mf{z}(\affg)$ acts on the first factor  by the natural inclusion
$\mf{z}(\affg)\hookrightarrow V^{\kappa_{c}}(\g)$,
and on the second by
twisting  the action by $\tau$.

The following assertion can be regarded as an affine analogue
of \cite[Theorem 2]{Ner81}.

\begin{Pro}\label{Pro:cdosub}
The vertex algebra homomorphism
\eqref{eq:hom-V-V} is injective.
\end{Pro}
\begin{proof}
The homomorphism
between the associated graded vertex Poisson algebras
induced by \eqref{eq:hom-V-V}
is dual to 
the morphism
\begin{align}
J_{\infty}G\times J_{\infty}\g_{}^*\ra  J_{\infty}\g_{}^*\times_{J_{\infty}(\g^*/\!/G)}J_{\infty}\g_{}^*,\quad (g,x)\mapsto (x,gx).
\label{eq:dominant?}
\end{align}
On the other hand,
 the morphism
\begin{align*}
G\times \g_{reg}^*\ra \g_{reg}^*\times_{\g^*/\!/G}\g_{reg}^*,\quad (g,x)\mapsto (x,gx).
\end{align*}
is smooth and surjective (\cite{Ric17}),
where $\g^*_{reg}$ is  the open dense subset of $\g^*$ consisting of regular elements.
Therefore,
the morphism
\begin{align*}
J_{\infty}G\times J_{\infty}\g_{reg}^*\ra  J_{\infty}\g_{reg}^*\times_{J_{\infty}(\g_{reg}^*/\!/G)}J_{\infty}\g_{reg}^*,\quad (g,x)\mapsto (x,gx),
\end{align*}
is formally smooth and surjective (\cite[Remark 2.10]{EinMus}).
Hence  \eqref{eq:dominant?}
is dominant.
\end{proof}

\begin{Pro}\label{Pro:cdo-is-good}
The cdo
$\Dch_{G,\kappa}$ is a direct sum of objects in $\KL^{ord}_{\kappa}$
(resp.\ $\KL^{ord}_{\kappa^*}$).
\end{Pro}
\begin{proof}
First, suppose that $\kappa=\kappa_c$.
Consider the decomposition 
\begin{align}
\Dch_{G,\kappa_c}=\bigoplus_{\lam\in P_+}(\Dch_{G,\kappa_c})_{[\lam]}
\label{eq:dec-cod-critical}
\end{align}
with respect to the action $\pi_R$ (see \eqref{eq:dec-KL-critical}).
By \eqref{eq:effect-of-tau-to-FF} each direct summand
$(\Dch_{G,\kappa_c})_{[\lam]}$ is closed under the action
of both $\pi_L$ and $\pi_R$ of $\affg_{\kappa_c}$.
Thus,
it is enough to show that 
$(\Dch_{G,\kappa_c})_{[\lam],\Delta}=(\Dch_{G,\kappa_c})_{[\lam]}\cap (\Dch_{G,\kappa_c})_{\Delta}$ 
is finite-dimensional 
for all $\lam$, $\Delta$.
%Note that
%by Lemma \ref{Lem:inversection-is-center},
By the definition \eqref{eq:cdo-def},
$\Dch_{G,\kappa_c}$ is cofree over
$U(t\g[t])$,
and we have
\begin{align}
(\Dch_{G,\kappa_c})^{t\g[t]}\cong U(\affg_{\kappa_c})\*_{U(\g[t]\+ \C \mathbf{1})}\mc{O}(G)\cong \bigoplus_{\lam\in P_+}\mathbb{V}_{\lam^*}\* V_{\lam}.
\label{eq:g[t]t-invariants-of-cdo}
\end{align}
Hence the direct summand
$(\Dch_{G,\kappa_c})^{[\lam]}$ is also cofree over $U(t\g[t])$
and we have
$(\Dch_{G,\kappa_c})_{[\lam]}^{t\g[t]}\cong \mathbb{V}_{\lam^*}\* V_{\lam}$.
It follows as
\begin{align}
(\Dch_{G,\kappa_c})_{[\lam]}\cong \mathbb{V}_{\lam^*}\* \mathbb{V}_{\lam}\quad\text{as 
graded vector spaces }
\label{ch-of-Dch}
\end{align}
 (not as $\affg_{\kappa_c}\+ \affg_{\kappa_c}$-modules),
and we are done.

Next suppose 
$\kappa\ne \kappa_c$.
Let $L^L(z)=\sum_{n\in \Z}L^L_nz^{-n-2}$
and 
 $L^R(z)=\sum_{n\in \Z}L^R_nz^{-n-2}$
 be the fields corresponding
 to the conformal vectors $\omega_L$ and $\omega_R$,
 respectively.
 Set
  $\Dch_{G,\kappa}[d_1,d_2]=\{v\in \Dch_{G,\kappa}\mid
 L_0^Lv=d_1v, \ L_0^Rv=d_2v\} $,
 so that
 $(\D_{G,\kappa})_{\Delta}=\bigoplus_{d_1+d_2=\Delta}\Dch_{G,\kappa}[d_1,d_2]$.
 We have
 \begin{align*}
\Dch_{G,\kappa}=\bigoplus_{d\in \C}\Dch_{G,\kappa}[\bullet,d]=\bigoplus_{d\in \C}\Dch_{G,\kappa}[d,\bullet],
\end{align*}
where
$\Dch_{G,\kappa}[\bullet,d]=\bigoplus_{d'\in \C}\Dch_{G,\kappa}[d',d]$,
$\Dch_{G,\kappa}[d,\bullet]=\bigoplus_{d'\in \C}\Dch_{G,\kappa}[d,d']$.
So
it is sufficient to show  that
$\Dch_{G,\kappa}[\bullet,d]$ and 
$\Dch_{G,\kappa}[d,\bullet]$ are objects of $\on{KL}_{\kappa}^{ord}$
and $\on{KL}_{\kappa^*}^{ord}$,
respectively.
This is equivalent to that 
 $\Dch_{G,\kappa}[d_1,d_2]$ is finite-dimensional for all $d_1,d_2\in \C$.
Since the argument is symmetric,
we may assume that
$(\kappa-\kappa_c)/\kappa_\g\not \in \Q_{\leq 0}$.
 By \cite{Zhu11},
there exists an increasing filtration
\begin{align}
0=M_1\subset M_1\subset \dots\dots ,\quad \ \Dch_{G,\kappa}=\bigcup_p M_p,
\label{eq:min-zhu-fil}
\end{align}
such that
\begin{align*}
\bigoplus_p M_p/M_{p-1}\cong \bigoplus_{\lam\in P_+}\mathbb{V}_{\lam}^\kappa\* D(\mathbb{V}_{\lam}^{\kappa^*}),
\end{align*}
where  $D(\mathbb{V}_{\lam}^{\kappa^*})$ is the contragredient dual of $\mathbb{V}_{\lam}^{\kappa^*}$.
On the other hand,
for a given $d_1$,
there exists only finitely many $\lam\in P_+$ such that 
$(\mathbb{V}_{\lam}^\kappa)_{d_1}\ne 0$,
where $(\mathbb{V}_{\lam}^\kappa)_{d_1}$
denotes the $L_0^L$-eigenspace
of $\mathbb{V}_{\lam}^\kappa$ with eigenvalue $d_1$.
It follows that
$\Dch_{G,\kappa}[d_1,d_2]$ is finite-dimensional for all $d_1,d_2\in \C$.
\end{proof}

By Proposition \ref{Pro:cdo-is-good},
$\Dch_{G,\kappa}$
is a conformal vertex algebra object in $\KL_{\kappa}$ (resp.\ $\KL_{\kappa^*}$)
with the chiral quantum moment map $\pi_L$ (resp.\ $\pi_R$).

Let $M\in \on{KL}_{\kappa'}$.
According to \cite[Theorem 6.5]{ArkGai02},
we have the canonical isomorphism
\begin{align}
\mc{D}_{G,\kappa}^{ch}\otimes M\cong \mc{D}_{G,\kappa-\kappa'}^{ch}\otimes  M
\label{eq:AG-iso}
\end{align}
of $\affg_{\kappa}\times \affg_{\kappa'-\kappa-\kappa_\g}$-modules.
Here,
on the left-hand side
$\affg_{\kappa}$ acts  on $\mc{D}_{G,\kappa}^{ch}$ via $\pi_L$
and $\affg_{\kappa'-\kappa-\kappa_0}$ diagonally  acts on  via $\pi_R$ and the $\affg_{\kappa'}$-action
on $M$.
On the right-hand side 
$\affg_{\kappa}$ diagonally  acts via $\pi_L$ and the $\affg_{\kappa'}$-action
on $M$,
and  $\affg_{\kappa'-\kappa-\kappa_0}$ acts on $\mc{D}_{G,\kappa}^{ch}$  via $\pi_R$.
Moreover,
it follows in the same manner as \cite[Proposition 3.4]{ACL17}
that
if $M$ is a vertex algebra object in $\KL_{\kappa'}$,
\eqref{eq:AG-iso} is a vertex algebra isomorphism.
Since
$H^{\semiinf+i}(\affg_{-\kappa_\g},\g,\Dch_{G,\kappa})\cong \delta_{i,0}\C$
by Theorem \ref{Th:vanishing-BRST-g-rel},
\eqref{eq:AG-iso}
induces 
 the isomorphism
\begin{align}
H^{\semiinf+i}(\affg_{-\kappa_\g},\g,\mc{D}_{G,\kappa}\* M)\cong \delta_{i,0}M
\end{align}
as $\affg_{\kappa}$-modules for any $M\in \on{KL}_{\kappa}$ (\cite{ArkGai02}).
This is an isomorphism of vertex algebras
 if $M$ is a vertex algebra object in $\on{KL}_{\kappa}$.
 In particular, in terms of the notation 
 \eqref{eq:MTproduct}
 we have
 \begin{align}
\Dch_{G,\kappa_c}\circ M\cong M\cong M\circ \Dch_{G,\kappa_c}
\label{eq:identify-cdo}
\end{align}
for $M\in \KL_{\kappa_c}$.

   \section{Quantum Drinfeld-Sokolov reduction and equivariant affine $W$-algebras}
 \label{section:Quantum Drinfeld-Sokolov reduction and equivariant affine $W$-algebras}
 Let 
$f$ be a nilpotent element of $\g$,
 $\{e,h,f\}$  an $\mf{sl}_2$-triple in $\g$ associated with
$f$,
$$\Slo_f=f+\g^e\subset \g\cong \g^*$$ the {\em Slodowy slice} at $f$,
where $\g^e$ is the centralizer of $e$ in $\g$.
It is known by \cite{GanGin02} that
$\Slo_f$ is a Poisson veriety,
where 
the Poisson structure of $\Slo_f$ is described as follows.
Set $$\fing_j=\{x\in \g\mid [h,x]=2jx\},$$
so that 
$\g=\bigoplus_{j\in \frac{1}{2}\Z}\fing_j$.
Choose a 	Lagrangian subspace $\mf{l}$ of 
$\g_{1/2}$ with respect to the symplectic form 
$\g_{1/2}\times \g_{1/2}\ra \C$,
$(x,y)\mapsto (f|[x,y])$.
Then $\mf{m}:=\mf{l}\+\bigoplus_{j\geq 1}\fing_j$ 
is a nilpotent Lie subalgebra of $\g$.
Let $M$ be the unipotent subgroup of $G$ corresponding to $\mf{m}$.
The projection
$\mu:\g^*\ra \mf{m}^*$ 
is the moment map for the $M$-action.
Let  $\chi:\mf{m}\ra \C$ be the character defined by
$\chi(x)= (x|f)$.
Then
  $\chi$ is a regular value of $\mu$.
Moreover,
we have the isomorphism of the affine varieties
$$M\times \Slo_f \isomap \mu^{-1}(\chi)=
\chi+\mf{m}^{\bot},\quad (g,s)\mapsto g.s.$$
Hence,
\begin{align*}
\Slo_f\cong \mu^{-1}(\chi)/M=\g^*/\!/\!/_{\chi}M
\end{align*}
has the structure of a reduced Poisson variety.

According to \cite{Gin08},
this construction can be generalized as follows.
Let  $R$ be a Poisson algebra object in $\on{QCoh}^{G}(\g^*)$,
$X=\Spec R$.
Then 
the moment map $\mu:X\overset{\mu_X}{\ra} \g^*\ra \mf{m}^*$ for the $M$-action is flat,
where the last map is the projection.
We get the reduced Poisson scheme $$ X/\!/\!/_{\chi}M=\mu^{-1}(\chi)/M=\Slo_f\times_{\g^*}X.$$
By applying this construction to $X=T^*G=G\times \g^*$
for the moment map $\mu_L$,
we obtain  the 
 {\em equivariant Slodowy slice}  \cite{Los07}
\begin{align*}
\mathbf{S}_{G,f}=G\times_M (f+\mf{m}^\bot)
\end{align*}
   at $f$. 
The Poisson variety $\mathbf{S}_{G,f}$ is an example of a
twisted contangent bundle  over $G/M$ (\cite[1.5]{ChrGin97}).
In particular, $ \mathbf{S}_{G,f}$ is a smooth symplectic variety.   
We have $\mathbf{S}_{G,f}^{op}\cong \mathbf{S}_{G,f}$
since $(T^*G)^{op}\cong T^*G$.
We identify 
$ \mathbf{S}_{G,f}$ with $ G\times \Slo_f$
via
the $G$-equivariant isomorphism
   \begin{align}
\nu:\mathbf{S}_{G,f}=G\times_{M}(\chi+\mf{m}^{\bot})\isomap G\times \Slo_f,
\quad (g,(n,s))\mapsto (gn,s),
\label{eq:Geqi}
\end{align}
($g\in G$, $n\in M$, $s\in \Slo_f$).
The $G$-action on  $ \mathbf{S}_{G,f}$ 
%given by $g.(h,s)=(gh,s)$
is Hamiltonian with the moment map
\begin{align}
\mu: \mathbf{S}_{G,f}=G\times\Slo_f\ra \g^*,\quad (g,s)\mapsto g.s
\label{eq:moment-map-for-equov-Slo},
\end{align}
which is flat 
 (\cite{Slo80}).
Hence, by \eqref{eq:BRST-cl}, \eqref{eq:Geqi},
we have
\begin{align}
H^{\semiinf+\bullet}(\g,\mc{O}(\mathbf{S}_{G,f})\* R)
\cong \mc{O}(\Slo_f\times_{\g^*}\Spec R)\* H^{\bullet}(\g,\C)
\end{align}
for a Poisson algebra object $R$ in $\on{QCoh}^G(\g^*)$.

The same argument as in \cite{Slo80}
shows that
\begin{align*}
J_n\mu: J_n\mathbf{S}_{G,f}=J_nG\times J_n\Slo_f\ra J_n\g^*
\end{align*}
is flat for  all $n\in \Z_{\geq 0}$,
and hence the same is true for $n={\infty}$.
Thus,
by \eqref{eq:cintg},
we have
\begin{align}
H^{\semiinf+\bullet}(\affg^{cl},\mc{O}(J_{\infty}\mathbf{S}_{G,f})\* V)
\cong \mc{O}(J_{\infty}\Slo_f\times_{\g^*} SS(V))\* H^{\bullet}(\g,\C)
\end{align}
for any Poisson vertex algebra object $V$ in $\on{Coh}^{J_\infty G}(J_{\infty}\g^*)$.
%In fact we have the following assertion.
%\begin{Th}
%$\mc{O}(\mathbf{S}_{G,f})$ is free over $\mc{O}(\g^*)$.
%\end{Th}
\medskip

For $M\in \on{KL}_{\kappa}$,
let $H_{DS,f}^{\bullet}(M)$ be the 
BRST cohomology of the
quantized Drinfeld-Sokolov reduction  \cite{FF90,KacRoaWak03}
associated with $f$ (and the Dynkin grading) with coefficients in $M$,
which is defined as follows.
Set $\g_{\geq 1}=\bigoplus_{j\geq 1}\g_j\subset \g_{>0}=
\bigoplus_{j>0}\g_j$.
Then 
 $\tilde\chi:\g_{\geq 1}[t,t^{-1}]\ra \C$,
$xt^n\mapsto \delta_{n,-1}(f|x)$,
defines a character.
Let $F_{\tilde\chi}=U(\g_{>0}[t,t^{-1}])\otimes_{U(\g_{>0}[t]+\g_{\geq 1}[t,t^{-1}])}\C_{\tilde\chi}$,
where $\C_{\tilde\chi}$ is the one-dimensional representation of 
$\g_{>0}[t]+\g_{\geq 1}[t,t^{-1}]$
on which $\g_{\geq 1}[t,t^{-1}]$ acts by  the character $\tilde\chi$
and  $\g_{>0}[t]$ acts trivially.
The space $F_{\chi}$ 
is isomorphic to the $\beta\gamma$-system associated with the symplectic vector space
$\g_{1/2}$.
In particular,
$F_{\chi}$  has the structure of a vertex algebra.
By definition,
\begin{align*}
H^\bullet_{DS,f}(M)=H^{\frac{\infty}{2}+\bullet}(\g_{>0}[t,t^{-1}],M\otimes  F_{\chi}),
\end{align*}
where 
$\g_{>0}[t,t^{-1}]$ acts diagonally on $M\otimes  F_{\chi}$
and
$H^{\frac{\infty}{2}+\bullet}(\g_{>0}[t,t^{-1}],N)$ is the semi-infinite  $\g_{>0}[t,t^{-1}]$-cohomology
\cite{Feu84}
with coefficients in a $\g_{>0}[t,t^{-1}]$-module $N$.

The (affine) {\em $W$-algebra} associated with $(\g,f)$ at level $\kappa$
is  
the vertex algebra
\begin{align*}
\W^\kappa(\g,f)=H_{DS,f}^0(V^{\kappa}(\g))
\end{align*}
(\cite{FF90,KacRoaWak03}).
 The $W$-algebra
 $\W^\kappa(\g,f)$ is a  strict chiral quantization of 
 the Slodowy slice $\Slo_f$ (\cite{De-Kac06,Ara09b}).
 
More generally,
if $V$ is a vertex algebra object  in $\KL_{\kappa}$,
then
$H_{DS,f}^0(V)$ 
is a vertex algebra,
and $\mu_V:V^\kappa(\g)\ra V$ induces a vertex algebra homomorphism
$\W^\kappa(\g,f)\ra H_{DS,f}^0(V)$.
If $V$ is a conformal vertex algebra object  in $\KL_{\kappa}$  with central charge $c_V$,
 it follows from  \cite[\S 2.2]{KacRoaWak03} that $H_{DS,f}^0(V)$ 
 is
conformal with central charge 
\begin{align}
c_V-\dim \mathbb{O}_f-\frac{3}{2}\dim \g_{1/2}+12(\rho|h)-3(k+h^{\vee})|h|^2,
%c_V-\dim \g+\dim \g_0-\frac{1}{2}\dim \g_{1/2}+12(\rho|h)-3(k+h^{\vee})|h|^2,
%\frac{(k+h^{\vee})}{4}|h|^2),
\label{eq:cc-of-reduction}
\end{align}
where 
$\mathbb{O}_f=G.f$ 
%$\rho$ is a half sum of positive root of $\g$
and $k=2h^{\vee}\kappa/\kappa_{\g}$.
Note that $\dim \mathbb{O}_f=\dim \g-\dim \g_0- \dim \g_{1/2}$.

\begin{Th}\label{Th:IMRN2016}
\begin{enumerate}
\item \rm{(}\cite{FreGai07,Ara09b}\rm{)}
For $M\in \on{KL}_{\kappa}$
we have 
$H_{DS,f}^i(M)=0$ for $i\ne 0$.
Therefore 
the functor $\on{KL}_{\kappa}\ra \W^\kappa(\g,f)\on{-Mod}$,
$M\mapsto H_{DS,f}^0(M)$, is exact.
\item \rm{(}\cite{Ara09b}\rm{)}
For  $M\in \on{KL}_{\kappa}^{ord}$,
we have
$$\gr H_{DS,f}^0(M)\cong \mc{O}(J_{\infty}\Slo_f)\*_{\mc{O}(J_\infty\g^*)}\gr M.$$
Hence, if 
 $V$ is a vertex algebra object in $\KL_\kappa$,
then
the vertex algebra 
$ H_{DS,f}^0(V)$ is separated, and
%The filtration $F^\bullet V$ induces the filtration
%$F^\bullet H_{DS,f}^0(V)$ of $V$,
%which is separated.
%We have
$\gr H_{DS,f}^0(V)\cong \mc{O}(J_{\infty}\mc{S}_f)\otimes_{\mc{O}(J_{\infty}\g^*)}(\gr V)$.
In particular,
$$\tilde{X}_{H_{DS,f}^0(V)}\cong \tilde{X}_V\times_{\g^*} \Slo_f.$$
%where $\tilde{X}_V\ra \g^*$ is the morphism induced by the vertex algebra homomorphism
%$\mu_V:V^\kappa(\g)\ra V$.
\item
\rm{(}\cite{A2012Dec}\rm{)}
For any vertex algebra object $V$ in $\KL_{\kappa}$,
we have
$$\on{Zhu}(H_{DS,f}^0(V))\cong H_{DS,f}^0(\on{Zhu}(V)),$$
where in the right-hand side $H_{DS,f}^0(?)$ is the finite-dimensional analogue of the quantized 
Drinfeld-Sokolov reduction as described in \cite{A2012Dec}.

\end{enumerate}

\end{Th}

\begin{Co}\label{Co:DS}
For a    vertex algebra object $V$ in $\KL_\kappa$,
 the vertex algebra $ H_{DS,f}^0(V)$ is  a chiral quantization of 
$X_V\times_{\g^*} \Slo_f$.
If in addition $V$ is a strict chiral quantization of  
$X=\tilde{X}_V$,
then 
$H_{DS,f}^0(V)$ is a strict chiralization of 
$X\times_{\g^*}\Slo_f$.
\end{Co}

Define
the {\em equivariant affine $W$-algebra $\eqW^\kappa_{G,f}$ associated with $(G,f)$ }at level $\kappa$
by
\begin{align}
\eqW^\kappa_{G,f}:=H_{DS,f}^0(\D_{G,\kappa}^{ch}).
\end{align}
By Corollary \ref{Co:DS},
$\eqW^\kappa_{G,f}$ is a   strict chiralization of
\begin{align}
\tilde{X}_{\eqW^\kappa_{G,f}}=X_{\Dch_{G,\kappa}}\times_{\g^*} \Slo_f\cong \mathbf{S}_{G,f}.
\end{align}
Since $\mathbf{S}_{G,f}$ is 
 a smooth symplectic variety,
$\eqW^\kappa_{G,f}$ is 
simple 
by Theorem \ref{Th:AMoreau}.
Because $\Dch_{G,\kappa}$ is conformal with central charge $2\dim \g$,
$\eqW^\kappa_{G,f}$ is conformal
with central charge
\begin{align}
%2 \dim \g+\dim \g_0-\frac{1}{2}\dim \g_{1/2}-12(\frac{|\rho|^2}{h^{\vee}}-(\rho|h)+\frac{(k+h^{\vee})}{4}|h|^2)\\=
\dim \g+\dim \g_0-\frac{1}{2}\dim \g_{1/2}+12(\rho|h)-3(k+h^{\vee})|h|^2
%\nonumber
\label{eq:cc-of-eq-W}
\end{align}
(see \eqref{eq:cc-of-reduction}).

The equivariant affine $W$-algebra
is $\frac{1}{2}\Z$-graded:
$\eqW_{G,f}^\kappa=\bigoplus_{\Delta\in \frac{1}{2}\Z}
(\eqW_{G,f}^\kappa)_{\Delta}$. 
Note that the conformal weight of $\eqW_{G,f}^\kappa$ is not bounded from the below.
(See the formula \eqref{eq:conformal-dim-eq-W} below.)

We have 
$\eqW_{G,f}^\kappa\cong (\eqW_{G,f}^\kappa)^{op}$
since 
$\Dch_{G,\kappa}\cong (\Dch_{G,\kappa})^{op}$,
and so we do not 
distinguish
between them.

For $G=SL_2$ and a generic $\kappa$,
the equivariant affine $W$-algebra $\eqW^\kappa_{G,f}$
was  studied earlier 
by Frenkel and  Styrkas in \cite{FreSty06}.

Let $ \mathbf{U}_{G,f}$
denote the {\em equivariant finite  $W$-algebra} introduced by Losev \cite{Los07}.
The vertex algebra
$\eqW_{G,f}^\kappa$ should be regarded as an affine analogue of
$ \mathbf{U}_{G,f}$ is the following sense.
\begin{Th}\label{Th:Zhu-eq-W}
 We have $\on{Zhu}(\eqW^\kappa_{G,f} ))\cong \mathbf{U}_{G,f}$
 for any $\kappa$.
\end{Th}
\begin{proof}
The statement  follows  from
\cite[Remark 3.1.4]{Los07},
 Theorem \ref{Th:IMRN2016} (3)
and the description of the  finite-dimensional analogue of the quantized 
Drinfeld-Sokolov reduction given in \cite{A2012Dec}.
\end{proof}

By Theorem \ref{Th:IMRN2016},
the vertex algebra homomorphism
$\pi_L:V^\kappa(\g) \hookrightarrow \Dch_{G,\kappa}$
induces the embedding
\begin{align}
\W^\kappa(\g,f)\hookrightarrow \eqW_{G,f}^\kappa,
\end{align}
which we  denote also by $\pi_L$.
On the other hand,
the vertex algebra homomorphism
$\pi_R:V^{\kappa^*}(\g) \hookrightarrow \Dch_{G,\kappa}$
induces the vertex algebra homomorphism
\begin{align}
V^{\kappa^*}(\g)\ra  \eqW_{G,f}^\kappa,
\end{align}
which is denoted also by $\pi_R$.
The induced actions of 
$\W^\kappa(\g,f)$ and $\affg_{\kappa^*}$ on $ \eqW_{G,f}^\kappa$ obviously commute,
and hence 
we have a vertex algebra homomorphism
\begin{align}
\W^\kappa(\g,f)\* V^{\kappa^*}(\g)\overset{\pi_L\* \pi_R}{\ra}\eqW_{G,f}^\kappa.
\end{align}

The map
$\gr \pi_L: \gr\W^\kappa(\g,f)=\mc{O}(J_{\infty}\Slo_f)\ra
\gr \eqW_{G,f}^\kappa=\mc{O}(J_{\infty} \mathbf{S}_{G,f})$
is induced by  the projection
$J_{\infty}\mathbf{S}_{G,f}\cong J_\infty (G\times \Slo_f)\cong J_{\infty}G\times J_{\infty}\Slo_f\ra J_{\infty}\Slo_f$,
while 
the 
map
$\gr \pi_R: \gr V^{\kappa^*}(\g)=\mc{O}(J_{\infty}\g^*)\ra  \gr \eqW_{G,f}^\kappa=\mc{O}(J_{\infty}\mathbf{S}_{G,f})$
is induced by 
the chiral moment map 
\begin{align}
J_{\infty}\mu: J_{\infty}\mathbf{S}_{G,f}=J_{\infty}G\times J_\infty \Slo_f\ra J_{\infty} \g^*,\quad (g,f)\mapsto gf.
\label{eq:ch:moment-map-for-equov-Slo}
\end{align}

\begin{Pro}\label{Pro:free}
The $\affg_{\kappa^*}$-module
$\eqW_{G,f}^\kappa$ is free over
$U(\g[t^{-1}]t^{-1})$.
Therefore,
there exists a 
filtration
$0=M_0\subset M_1\subset M_2\subset \dots $,
$\eqW_{G,f}^\kappa=\bigcup_p M_p$,
as a  $\affg_{\kappa^*}$-module
such that
each successive  quotient is isomorphic to $\mathbb{V}_{\lam}^{\kappa^*}$ for some $\lam\in P_+$.
\end{Pro}
\begin{proof}
Since $\eqW_{G,f}^\kappa$
is an object of $\KL_{\kappa^*}$,
it is sufficient to show that
we have
%the Lie algebra homology vanishes
$H_i(\g[t^{-1}]t^{-1},\eqW_{G,f}^\kappa)=0$ for $i>0$.
We consider the spectral sequence
$E_r\Rightarrow H^\bullet(t^{-1}\g[t^{-1}],\eqW_{G,f}^\kappa)$
such that 
the $E_1$-term is isomorphic to 
$H_\bullet(t^{-1}\g[t^{-1}],\gr \eqW_{G,f}^\kappa)
\cong H_{\bullet}(t^{-1}\g[t^{-1}], \mc{O}(J_{\infty}\mathbf{S}_{G,f}))$,
that is,
 the Koszul homology of $\mc{O}(J_{\infty}\mathbf{S}_{G,f})$ as a $\mc{O}(J_{\infty}\g^*)$-module.
 Since $J_{\infty}\mu$ is a flat,
 we have that $ H_{i}(\g[t^{-1}]t^{-1}, \mc{O}(J_{\infty}\mathbf{S}_{G,f}))=0$ for $i>0$.
Hence the spectral sequence collapses at $E_1=E_{\infty}$,
and we get that 
$H_i(\g[t^{-1}]t^{-1},\eqW_{G,f}^\kappa)=0$ for $i>0$ as required.
\end{proof}
By Proposition \ref{Pro:free},
it follows that 
the vertex algebra homomorphism
$\pi_R:V^{\kappa^*}(\g)\ra  \eqW_{G,f}^\kappa$ is injective.

\begin{Pro}\label{Pro:g[t]t-invariant-part}
The $\affg_{\kappa^*}$-module
$\eqW_{G,f}^\kappa$ is cofree over $U(t\g[t])$
and
we have
%$H^i(t\g[t],\eqW_{G,f}^\kappa)=0$ for $i\ne 0$
%and
$$(\eqW^{\kappa}_{G,f})^{t\g[t]}\cong \bigoplus_{\lam\in P_+}H_{DS,f}^0(\mathbb{V}_{\lam}^\kappa)\otimes V_{\lam^*}$$
as $\W^\kappa(\g,f)\otimes U(\g)$-module.
In particular,
$$ \W^\kappa(\g,f)\cong \on{Com}(V^{\kappa^*}(\g),
 \eqW^{\kappa}_{G,f})= (\eqW_{G,f}^\kappa)^{\g[t]}$$
as vertex algebras.
\end{Pro}
\begin{proof}
Consider the spectral sequence
$E_r\Rightarrow H^\bullet(t\g[t],\eqW_{G,f}^\kappa)$
such that 
the $E_1$-term is 
$H^\bullet(t\g[t],\gr \eqW_{G,f}^\kappa)$.
Since $\gr \eqW_{G,f}^\kappa\cong \mc{O}(\mathbf{S}_{G,f})$,
we have
$H^i(t\g[t],\gr \eqW_{G,f}^\kappa)\cong \delta_{i,0}\mc{O}(J_{\infty}\Slo_f)\* \mc{O}(G)$.
It follows that the spectral sequence
 collapses at $E_1=E_{\infty}$,
 and we get that
\begin{align}
\gr  H^i(t\g[t],\eqW_{G,f}^\kappa)\cong \delta_{i,0}\mc{O}(J_{\infty}\Slo_f)\* \mc{O}(G).
\label{eq:2018-10-17-15-27}
\end{align}
In particular,
$\eqW_{G,f}^\kappa$ is cofree over  $U(t\g[t])$.

On the other hand, 
by Theorem \ref{Th:IMRN2016},
the embedding 
$$(\Dch_{G,\kappa})^{\pi_R(t\g[t])}\cong \bigoplus_{\lam\in P_+}\mathbb{V}_{\lam}^\kappa\* V_{\lam^*}\hookrightarrow
\Dch_{G,\kappa}$$
(see \eqref{eq:g[t]t-invariants-of-cdo})
induces the embedding
\begin{align*}
H_{DS,f}^0((\Dch_{G,\kappa})^{\pi_R(t\g[t])})\cong 
\bigoplus_{\lam\in P_+}H_{DS,f}^0(\mathbb{V}_{\lam}^\kappa)\* V_{\lam^*}\hookrightarrow
\eqW_{G,f}^{\kappa},
\end{align*}
and the image 
is contained in 
$(\eqW_{G,f}^\kappa)^{t\g[t]}$.
Since 
$\gr H_{DS,f}^0(\mathbb{V}_{\lam}^\kappa)\cong \mc{O}(J_{\infty}\Slo_f)\* V_{\lam}$
(see \cite{Ara09b}),
we have
$\gr \left(\bigoplus_{\lam\in P_+}H_{DS,f}^0(\mathbb{V}_{\lam}^\kappa)\* V_{\lam^*}\right)
\cong \mc{O}(J_{\infty}\Slo_f)\* \mc{O}(G)$.
By comparing with \eqref{eq:2018-10-17-15-27},
we get that 
$(\eqW_{G,f}^\kappa)^{t\g[t]}\cong \bigoplus_{\lam\in P_+}H_{DS,f}^0(\mathbb{V}_{\lam}^\kappa)\* V_{\lam^*}$.
\end{proof}

%Suppose that $\kappa$ is generic.
%By the chiral Peter-Weyl Theorem \eqref{eq:chiral-Peter-Weyl},
%we get
%\begin{align*}
%\eqW_{G,f}^\kappa \cong  \bigoplus_{\lam\in P_+}  H_{DS,f}^0(V_{\lam}^\kappa)\otimes  V_{\lam^*}^{\kappa^*}.
%\end{align*}

\begin{Pro}\label{Pro:eqW-finite}
The $\affg_{\kappa^*}$-module
$\eqW_{G,f}^\kappa$ is a direct sum of objects in $\on{KL}^{ord}_{\kappa^*}$.
\end{Pro}
\begin{proof}
First, 
suppose that 
$\kappa=\kappa_c$.
Consider the decomposition
 \begin{align}
\eqW_{G,f}^{\kappa_c}=\bigoplus_{\lam\in P_+}(\eqW_{G,f}^{\kappa_c})_{[\lam]}
\label{eq:dec-eqW-critical}
\end{align}
(see \eqref{eq:dec-KL-critical}).
We have
$(\eqW_{G,f}^{\kappa_c})_{[\lam]}=H_{DS,f}^0((\Dch_{G,\kappa_c})_{[\lam]}$.
In the same way  as in the proof of Proposition \ref{Pro:cdo-is-good},
we find using  Proposition \ref{Pro:g[t]t-invariant-part} that
$(\eqW_{G,c})_{[\lam]}\cong H_{DS,f}^0(\mathbb{V}_{\lam^*}^\kappa)\* \mathbb{V}_{\lam}^{\kappa^*}$
as graded vector spaces
and we are done.

So suppose that
 $\kappa\ne \kappa_c$.
Let $\omega$ be the conformal vector of $\eqW_{G,f}^\kappa$.
Then $\omega=\omega_{\W}+\omega_{\affg}$,
where 
$\omega_W$ and $\omega_{\affg}$
are conformal vectors of $\W^k(\g,f)$ and $V^{\kappa^*}(\g)$,
respectively.
Let
 $L^{\W}(z)=\sum_{n\in \Z}L^{\W}_nz^{-n-2}$
and 
 $L^{\affg}(z)=\sum_{n\in \Z}L^{\affg}_nz^{-n-2}$
 be the fields corresponding
 to $\omega_{\W}$ and $\omega_{\affg}$,
 respectively.
 Set
  $\eqW_{G,f}^{\kappa}[d_1,d_2]=\{v\in \eqW_{G,f}^{\kappa}\mid
 L_0^{\W}v=d_1v, \ L_0^{\affg}v=d_2v\} $,
 so that
 $(\eqW_{G,f}^{\kappa})_{\Delta}=\bigoplus_{d_1+d_2=\Delta}\eqW_{G,f}^{\kappa}[d_1,d_2]$.
 We have
 \begin{align*}
\eqW_{G,f}^\kappa=\bigoplus_{d\in \C}\eqW_{G,f}[d,\bullet]
\end{align*}
as a $\affg_{\kappa^*}$-module,
where
$\eqW_{G,f}^\kappa[d,\bullet]=\bigoplus_{d'\in \C}\eqW_{G,f}^\kappa[d,d']$.
So
it is sufficient to show  that
$\eqW_{G,f}^\kappa[d,\bullet]$ 
is  an  object
of $\on{KL}_{\kappa^*}^{ord}$,
or equivalently,
 $\eqW_{G,f}^\kappa[d_1,d_2]$ is finite-dimensional for all $d_1,d_2\in \C$.
If
$(\kappa-\kappa_c)/\kappa_\g\not \in \Q_{\leq 0}$,
set $N_p=H_{DS}^0(M_p)$,
where $M_p$
is as in \eqref{eq:min-zhu-fil}.
By Theorem \ref{Th:IMRN2016},
$N_{\bullet}$ defines a filtration of $\eqW_{G,f}^\kappa$
such that
\begin{align*}
\bigoplus_o N_p/N_{p-1}\cong \bigoplus_{\lam\in P_+}H_{DS,f}^0(\mathbb{V}_{\lam}^\kappa)\* D(\mathbb{V}_{\lam}^{\kappa^*}).
\end{align*}
We have 
$H_{DS,f}^0(\mathbb{V}_{\lam}^\kappa)=\bigoplus_{d\in \Z_{\geq 0}}H_{DS,f}^0(\mathbb{V}_{\lam}^\kappa)_{d+h_{\lam}}$,
and $\dim H_{DS,f}^0(\mathbb{V}_{\lam}^\kappa)_d<\infty $ for all $d$,
where $H_{DS,f}^0(\mathbb{V}_{\lam}^\kappa)_{d}$
denotes the $L_0^{\W}$-eigenspace
of $H_{DS,f}^0(\mathbb{V}_{\lam}^\kappa)$ with eigenvalue $d$
and
\begin{align}
h_{\lam}=\frac{(\lam+2\rho\mid \lam)}{2(k+h^{\vee})}-\frac{1}{2}\lam(h)
=\frac{|\lam-\frac{k+h^{\vee}}{2}h+\rho|^2-|\rho|^2}{2(k+h^{\vee})}+\frac{1}{2}(\rho|h)-\frac{(k+h^{\vee})}{8}|h|^2 .
\label{eq:conformal-dim-eq-W}
\end{align}
It follows that,
 for a given $d_1$,
there exists only finitely many $\lam\in P_+$ such that 
$H_{DS,f}^0(\mathbb{V}_{\lam}^\kappa)_{d_1}\ne 0$,
and therefore 
$\eqW_{G,f}^\kappa[d_1,d_2]$ is finite-dimensional for all $d_1,d_2\in \C$.
If
$(\kappa-\kappa_c)/\kappa_\g\not \in \Q_{\geq 0}$,
there exists \cite{Zhu11} a decreasing filtration
\begin{align*}
\Dch_{G,\kappa}=M_1\supset M_1\supset \dots\dots ,\quad \ \bigcap_p M_p=0,
\end{align*}
such that
\begin{align*}
\bigoplus_p M_p/M_{p+1}\cong \bigoplus_{\lam\in P_+}\mathbb{V}_{\lam}^\kappa\* D(\mathbb{V}_{\lam}^{\kappa^*}).
\end{align*}
Setting
 $N_p=H_{DS}^0(M_p)$,
we get a filtration of $\eqW_{G,f}^\kappa$
such that
\begin{align*}
\bigoplus_o N_p/N_{p+1}\cong \bigoplus_{\lam\in P_+}H_{DS,f}^0(\mathbb{V}_{\lam}^\kappa)\* D(\mathbb{V}_{\lam}^{\kappa^*}).
\end{align*}
Since
there exist only finitely many $\lam\in P_+$ such that 
$D(\mathbb{V}_{\lam}^\kappa)_{d_1}\ne 0$
for a given $d_1$,
it follows that
$\eqW_{G,f}^\kappa[d_1,d_2]$ is finite-dimensional for all $d_1,d_2\in \C$.
This completes the proof.
\end{proof}

By Proposition \ref{Pro:eqW-finite},
$\eqW_{G,f}^\kappa$ is a conformal vertex algebra object in $\on{KL}^{ord}_{\kappa^*}$
with the chiral quantum moment map $\pi_R$.

By Theorem \ref{Th:vanishing-BRST-g-rel},
we have the following assertion.
\begin{Th}\label{Th:secDS}
Let $M\in \on{KL}_\kappa$.
Then
$H^{\frac{\infty}{2}+i}(\affg_{-\kappa_\g},\g, \eqW_{G,f}^{\kappa}\* M)=0$
for $i\ne 0$,
where 
the BRST cohomology is taken with respect to the action of
$\affg_{-\kappa_\g}$ on
$ \eqW_{G,f}^{\kappa}\* M$
given by  $x(w\*m)=\pi_R(x)w\* m+w\* xm$.
If $M\in \on{KL}^{ord}$
we have
\begin{align*}
\gr H^{\frac{\infty}{2}+0}(\affg_{-\kappa_\g},\g, \eqW_{G,f}^{\kappa}\* M)\cong 
\mc{O}(J_{\infty}\Slo_f)\*_{\mc{O}(J_{\infty}\g^*)}\gr M.
\end{align*}
\end{Th}

\begin{Th}\label{Th:DS-realization}
For $M\in \on{KL}_{\kappa}$,
we have the  isomorphism
\begin{align*}
H_{DS,f}^\bullet(M)\cong H^{\frac{\infty}{2}+\bullet}(\affg_{-\kappa_\g},\g, \eqW_{G,f}^{\kappa}\* M)
\end{align*}
as $\W^\kappa(\g,f)$-modules,
where
%where on the right-hand-side
%the BRST cohomology is taken with respect to the action of
%$\affg_{-\kappa_\g}$ on
%$ \eqW_{G,f}^{\kappa}\* M$
%given by  $x(w\*m)=\pi_R(x)w\* m+w\* xm$,
 $\W^\kappa(\g,f)$ acts on the first factor $\eqW_{G,f}^\kappa$ on the right-hand side.
 If $M$ is a vertex algebra object in $\KL_{\kappa}$,
 this is an isomorphism of vertex algebras.
\end{Th}
\begin{proof}
First, note that
we already know that
the cohomology of both 
sides vanishes 
for nonzero cohomological degrees
(Theorem \ref{Th:IMRN2016}
 and Theorem \ref{Th:secDS}).

We may assume that $M\in \KL_{\kappa}^{ord}$
since the cohomology functor commutes with injective limits.
We may also assume that $M$ is $\Z_{\geq 0}$-graded:
$M=\bigoplus_{d\in \Z_{\geq 0}}M_d$:
% be  a grading
%is a $\Z_{\geq 0}$-grading of $M$ compatible with the grading of $V^\kappa(\g)$,
%that is,
$(xt^n)M_d\subset M_{d-n}$ for $x\in \g$, $n\in \Z$.
Note that
\begin{align}
\gr H_{DS,f}^\bullet(M)\cong \gr H^{\frac{\infty}{2}+\bullet}(\affg_{-\kappa_\g},\g, \eqW_{G,f}^{\kappa}\* M)
\label{eq:2018:10:17}
\end{align}
by Theorem \ref{Th:IMRN2016}
 and Theorem \ref{Th:secDS}.

Let
\begin{align*}
C=\bigoplus_{i\in \Z}C^i,\quad
&C^i=\bigoplus_{p.q\atop p+q=i}C^{p,q}, \\
&C^{p,q}=M\* \Dch_{G,\kappa}\* \bw{\semiinf+p}(\g)\*  F_{\chi}\*\bw{\semiinf+q}(\g_{>0}),
\end{align*}
where $F_{\chi}$ is the vertex algebra
generated by fields $\Psi_{\alpha}(z)$,
$\alpha\in \Delta_{1/2}$,
with the OPEs $\Psi_{\alpha}(z)\Psi_{\alpha}(w)\sim (f|[x_{\alpha},x_{\beta}])/(z-w)$,
and $\bw{\semiinf+q}(\g_{>0})$
is the vertex superalgebra
generated by odd fields 
$\psi_{\alpha}(z)$,
$\psi_{\alpha}^*(z)$,
$\alpha\in \Delta_{>0}$,
with the OPEs
$\psi_{\alpha}(z)\psi_{\beta}^*(w)\sim \delta_{\alpha,\beta}/(z-w)$.
Here, $x_{\alpha}$ is the root vector of $\g$ corresponding to the root $\alpha$,
$\Delta_j=\{\alpha\in \Delta\mid x_{\alpha}\in \g_j\}$,
$\Delta_{>0}=\bigsqcup_{j>0}\Delta_j$,
and we have taken the Cartan subalgebra $\h$ in $\g_0$
and $\Delta$ is the corresponding set of roots of $\g$.

Set
 $Q_{(0)}=(Q_{\affg})_{(0)}+(Q_{DS})_{(0)}$,
where 
\begin{align*}
&Q_{\affg}(z)=\sum_{i=1}^{\dim \g} (\pi_M(x_i(z))+\pi_R(x_{i}(z)))\psi_{i}^*(z)-\frac{1}{2}
\sum_{i,j,k}c_{i,j}^k :\psi_{i}^*(z)\psi_{j}^*(z)\psi_{k}(z):,\\
&Q_{DS}(z)=\sum_{\alpha\in \Delta_{>0}} (\pi_L(x_{\alpha}(z))+\Phi_{\alpha}(z))\psi_{\alpha}^*(z)-\frac{1}{2}
\sum_{\alpha,\beta,\gamma\in \Delta_{>0}}c_{\alpha,\beta}^\gamma :\psi_{\alpha}^*(z)\psi_{\beta}^*(z)\psi_{\gamma}(z):,
\end{align*}
where
$\pi_M$ is the action of $\affg_{\kappa}$ on $M$,
$\{x_i\}$ is a basis of $\g$,
 $c_{ij}^k$ is the corresponding structure constant,
$c_{\alpha,\beta}^\gamma$ is the structure constant of $\g_{>0}$ with respect to the basis
$\{x_{\alpha}\}$,
and
we have set $\Phi_{\alpha}(z)=(f|x_{\alpha})$ for $\alpha\in \Delta_j$, $j\geq 1$,
and %we have 
omitted the tensor product symbol.
Then
$(Q^{\affg}_{(0)})^2=(Q^{DS}_{(0)})^2=0$,
$\{Q^{\affg}_{(0)},Q^{DS}_{(0)}\}=0$,
and $C$ has the structure of a double complex.

Let 
$(E_r,d_r)$ be
the spectral 
sequence for the total cohomology $H^\bullet(C,Q_{(0)})$
associated with the decreasing filtration $\{\sum_{i\geq p}C^{i,\bullet}\}$ of $C$.
By definition the zeroth differential $d_0$ is
$(Q_{DS})_{(0)}$ and the first differential 
$d_1$ is $(Q_{\affg})_{(0)}$.
We claim that this spectral sequence converges to $H^\bullet(C,Q_{(0)})$. 
To see this,
first
suppose that 
$\kappa=\kappa_c$.
We have the direct sum decomposition
$(\Dch_{G,\kappa_c})_{[\lam]}=\bigoplus_{\lam\in P_+}(\Dch_{G,\kappa_c})_{[\lam]}$
as a $\affg_{\kappa_c}^{\oplus 2}$-modules,
see \eqref{eq:dec-cod-critical}.
By \eqref{ch-of-Dch},
each $(\Dch_{G,\kappa_c})_{[\lam]}$ admits
a bi-grading
$(\Dch_{G,\kappa_c})_{[\lam]}=\bigoplus\limits_{d_1,d_2\in \Z_{\geq 0}}(\Dch_{G,\kappa_c})_{[\lam]}[d_1,d_2]$
such that 
$\pi_L(xt^n)(\Dch_{G,\kappa_c})_{[\lam]}[d_1,d_2]\subset (\Dch_{G,\kappa_c})_{[\lam]}[d_1-n,d_2]$,
$\pi_R(xt^n)(\Dch_{G,\kappa_c})_{[\lam]}[d_1,d_2]\subset (\Dch_{G,\kappa_c})_{[\lam]}[d_1,d_2-n]$,
for $x\in \g$, $n\in \Z$.
It follows that
$C$ is a direct sum of subcomplexes
\begin{align*}
C_{\lam,D}=\bigoplus_{d\in \Z_{\geq 0}\atop D=d+\Delta}
M\* (\Dch_{G,\kappa})_{[\lam]}[\bullet,d]\* \bw{\semiinf+p}(\g)_{\Delta}\*  F_{\chi}\*\bw{\semiinf+q}(\mf{n}),
\end{align*}
$\lam\in P_+$,
$D\in \Z_{\geq 0}$,
and
we have
$C_{\lam,D}\cap C^{p,\bullet}=0$
for a sufficiently large $p$.
Hence the restriction of the filtration 
on each subcomplex is regular, and
this gives the convergency of  the  spectral sequence.
Next consider the case that $\kappa\ne \kappa_c$.
The increasing
filtration
$M_{\bullet}$ of $\Dch_{G,\kappa}$ in 
\eqref{eq:min-zhu-fil}
induces a filtration 
$C_{\bullet}$ of the complex $C$.
As above, we find that the spectral sequence converges on each $C_p$.
Since the cohomology functor commutes with the injective limits,
we get a vector space isomorphism
$H^{\bullet}(C)\cong \lim\limits_{\rightarrow}H^{\bullet}(C_p)\cong \lim\limits_{\rightarrow}E^{(p)}_{\infty}\cong E_{\infty}$,
where
$E_r^{(p)}$ is the spectral sequence for $H^\bullet(C_p)$.

By definition, 
we have
\begin{align*}
E_1^{p,q}=M\* H^q_{DS,f}(\Dch_{G,\kappa})\* \bw{\semiinf+p}(\g)
\cong \delta_{q,0}M\* \eqW_{G,f}^\kappa\* \bw{\semiinf+p}(\g),
\end{align*}
and
\begin{align*}
E_2^{p,q}\cong \delta_{q,0}
H^{\frac{\infty}{2}+0}(\affg_{-\kappa_\g},\g, \eqW_{G,f}^{\kappa}\* M)\* H^p(\g,\C).
\end{align*}
Since
elements of 
$H^{\frac{\infty}{2}+0}(\affg_{-\kappa_\g},\g, \eqW_{G,f}^{\kappa}\* M)$ are $\g$-invariant,
it follows that $d_r=0$ for all $r\geq 2$.
Therefore,
the spectral sequence collapses at $E_2=E_{\infty}$,
and we obtain the isomorphism
\begin{align}
H^i(C, Q_{(0)})\cong H^{\frac{\infty}{2}+0}(\affg_{-\kappa_\g},\g, \eqW_{G,f}^{\kappa}\* M)\* H^i(\g,\C).
\label{eq:iso-11}
\end{align}

On the other hand, by  \eqref{eq:AG-iso},
we have the canonical isomorphism
\begin{align*}
H^\bullet(C, Q_{(0)})\cong H^\bullet(C', Q'_{(0)}),
\end{align*}
where 
$$C'=\bigoplus_{p.q}(C')^{p,q}, 
\quad (C')^{p,q}=M\* \Dch_{G,0}\* \bw{\semiinf+p}(\g)\*  F_{\chi}\*\bw{\semiinf+q}(\mf{n}),
$$
\begin{align*}
Q'(z)=Q_{\affg}'(z)+Q_{DS}'(z),
\end{align*}
\begin{align*}
&Q_{\affg}'(z)=\sum_{i=1}^{\dim \g} \pi_R(x_{i})(z)\psi_{i}^*(z)-\frac{1}{2}
\sum_{i,j,k}c_{i,j}^k :\psi_{i}^*(z)\psi_{j}^*(z)\psi_{k}(z):,\\
&Q_{DS}'(z)=\sum_{\alpha\in \Delta_{>0}} (\pi_M(x_{\alpha})(z)+\pi_L(x_{\alpha}(z))+\Phi_{\alpha}(z))\psi_{\alpha}^*(z)\\
&\quad \quad \qquad \qquad \qquad \qquad -\frac{1}{2}
\sum_{\alpha,\beta,\gamma\in \Delta_{>0}}c_{\alpha,\beta}^\gamma :\psi_{\alpha}^*(z)\psi_{\beta}^*(z)\psi_{\gamma}(z):.\\
\end{align*}
Consider the  map
\begin{align}
H_{DS,f}^0(M)\ra H^0(C', Q'_{(0)}),\quad
[c]\mapsto [c\* \mathbf{1}_{\Dch_{G,\kappa}}\*  \mathbf{1}_{\bw{\semiinf+\bullet}(\g)}],
\label{eq:2018:10:17:01}
\end{align}
where 
$\mathbf{1}_{\Dch_{G,\kappa}}$ and $\mathbf{1}_{\bw{\semiinf+\bullet}(\g)}$
denotes the vacuum vectors of 
$\Dch_{G,\kappa}$ and $\bw{\semiinf+\bullet}(\g)$,
respectively.
We claim that this is an isomorphism.
Indeed,
this respects the filtration on each side,
and gives rise to a homomorphism
\begin{align*}
\gr H_{DS,f}^0(M)\ra \gr H^0(C', Q'_{(0)})\cong \gr H^0(C, Q_{(0)})\cong \gr H^{\frac{\infty}{2}+0}(\affg_{-\kappa_\g},\g, \eqW_{G,f}^{\kappa}\* M),
\end{align*}
which
 is an isomorphism
by \eqref{eq:2018:10:17}.
Therefore \eqref{eq:2018:10:17:01}
is an isomorphism as well.
We conclude that
$$
H^{\frac{\infty}{2}+0}(\affg_{-\kappa_\g},\g, \eqW_{G,f}^{\kappa}\* M) \cong H^0(C, Q_{(0)})
\cong H^0(C', Q'_{(0)})\cong H_{DS,f}^0(M).$$
If $M$ is a vertex algebra object in $\KL_\kappa$
this is vertex algebra isomorphism
since all maps are vertex algebra isomorphisms.
\end{proof}

%\begin{Pro}
%For any $\kappa$,
%$\W^\kappa(\g,f)$ and $V^{\kappa^*}(\g)$ form a dual pair in 
%$\eqW_{G,f}^\kappa$.
%\end{Pro}
%\begin{proof}
%WRITE MORE
%\end{proof}

%
%Recall that
%\begin{align*}
%\on{Zhu}\W^\kappa(\g,f)\cong U(\g,f)
%\end{align*}
%(\cite{Ara07,De-Kac06}, see also \cite{Ara09b}),
%where $U(\g,f)$ is the {\em finite $W$-algebra} associated with $(\g,f)$ (\cite{Pre02}).

In terms of the notation \eqref{eq:MTproduct},
we have
\begin{align}
\eqW_{G,f}^{\kappa_c}\circ M\cong  H_{DS,f}^0(M)\cong M\circ \eqW_{G,f}^{\kappa_c}
\label{eq:eqW-vs-DS}
\end{align}
for $M\in \KL_{\kappa_c}$.

\section{Chiral universal centralizer}
\label{sec:ch-uv}
Let $f,f'$ be nilpotent elements of $\g$.
Define
\begin{align*}
\mathbf{I}_{G,f,f'}^\kappa:=H_{DS,f'}^0(\eqW_{G,f}^\kappa),
\end{align*}
where the Drinfeld-Sokolov reduction 
of $\eqW_{G,f}^\kappa$
is taken with respect to the action $\pi_R$.
By Corollary \ref{Co:DS},
$\mathbf{Z}_{G,f,f'}^\kappa$ is a conformal vertex algebra
that is a strict chiralization of 
$$\mathbf{S}_{G,f}\times_{\g^*}\Slo_f=(G\times \Slo_f)\times_{\g^*}\Slo_{f'},$$
where
$\mathbf{S}_{G,f}\ra \g^*$ is given by the moment map \eqref{eq:moment-map-for-equov-Slo}.
If $f=f'$,
the central charge of 
$\mathbf{I}_{G,f,f}$
is independent of $\kappa$ and is given by
\begin{align*}
2 \dim  \g_0-\dim \g_{1/2}+24(\rho|h).
\end{align*}

Set
\begin{align*}
\mathbf{I}_{G}^\kappa:=\mathbf{I}_{G,f_{prin},f_{prin}}^\kappa,
\end{align*}
where $f_{prin}$ is a principal nilpotent element of $\g$.
Then 
$\mathbf{I}_{G}^\kappa$ is a strict chiral quantization of 
the {\em universal centralizer}
\begin{align}
\mathbf{I}_G^{cl}=   (G\times \Slo_{f_{prin}})\times_{\g^*}\Slo_{f_{prin}}.
\end{align}
The central charge 
of $\mathbf{I}_{G}^\kappa$ 
is given by
\begin{align*}
%2\on{rk}\g+4h^{\vee}\dim \g.
2\on{rk}\g+48 (\rho|\rho^{\vee}).
\end{align*}
%Here we have used the  strange formula
% $\dim \g=12|\rho|^2/h^{\vee}$.

By \cite{BezFinMir05,BezFin08},
Theorem \ref{Th:IMRN2016} (3),
and the description of the  finite-dimensional analogue of the quantized 
Drinfeld-Sokolov reduction given in \cite{A2012Dec},
we have the isomorphisms
\begin{align}
R_{\mathbf{I}_{G}^\kappa}\cong H^{\check{G}[[t]]}_\bullet(\on{Gr}_{\check{G}}),
\quad \on{Zhu}(\mathbf{I}_{G}^\kappa)\cong 
H^{\check{G}[[t]]\ltimes \C^*}_\bullet(\on{Gr}_{\check{G}}),
\end{align}
where $\check{G}$ is the 
Langlands dual group of $G$
and $\on{Gr}_{\check{G}}$ is the affine Grassmannian 
for $\check{G}$ as in Introduction.

By \cite{Gin18},
$\on{Zhu}(\mathbf{I}_{G}^\kappa)$ is also isomorphic to
the {\em spherical subalgebra of the nil-Hecke algebra}.

The vertex algebra $\mathbf{I}_{G}^\kappa$
is called the {\em chiral universal centralizer}
associated with $G$
at level $\kappa$.

For $G=SL_2$ and a generic $\kappa$,
the chiral universal centralizer
$\mathbf{I}_{G}^\kappa$ 
was introduced earlier 
by Frenkel and Styrkas \cite{FreSty06}
as the 
{\em modified regular representation} of the Virasoro algebra
(see also  \cite{FreZhu12}).

\section{Drinfeld-Sokolov reduction at the Critical level}
\label{section:Drinfeld-Sokolov reduction at the Critical level}

For the rest of this article we restrict to ourselves the case that $\kappa=\kappa_c$.
Set $$\KL=\KL_{\kappa_c},
\quad \Dch_{G}=\Dch_{G,\kappa_c}.
%\quad \mathbb{V}_{\lam}=\mathbb{V}_{\lam}^{\kappa_c}.
$$

Let $\lam\in P_+$.
We have
\begin{align}
\ch \mathbb{V}_{\lam}=\frac{\sum_{w\in W}\epsilon(w)e^{w\circ \lam}}{\prod_{\alpha\in \hat{\Delta}^{re}_+}(1-e^{-\alpha})\prod_{j=1}^{\infty}
(1-q^j)^{\on{rk}\g}},
\label{eq:ch-weyl}
\end{align}
where 
$W$ is the Weyl group of $\g$,
$\widehat{\Delta}_+^{re}$ is the set of positive real roots of $\affg_{\kappa_c}$.
Also, we have
\begin{align}
\ch \mathbb{L}_{\lam}=\frac{\sum_{w\in W}\epsilon(w)e^{w\circ \lam}}{
\prod_{\alpha\in \Delta_+}
(1-q^{\bra \lam+\rho.\alpha^{\vee}\ket})
\prod_{\alpha\in \hat{\Delta}^{re}_+}(1-e^{-\alpha})}
\label{eq:ch-simple}
\end{align}
(\cite{Ara07-3,AraMal}).

Fix a principal nilpotent element $f=f_{prin}$ of $\g$,
and set
$H_{DS}^\bullet(?)=H_{DS,f}^{\bullet}(?)$.
By \cite{FeiFre92},
we have the isomorphism
\begin{align}
\mf{z}(\affg)\isomap \W^{\kappa_c}(\g,f)
%\quad z\mapsto [z\* \mathbf{1}_{\text{Fermionic ghost}}]
\label{eq:FF-iso}
\end{align}
of vertex algebras.
Thus, for $M\in \KL$,
$H_{DS}^0(M)$ is naturally a module over $\LZ$
(see \eqref{eq:LZ}).

By Theorem \ref{Th:IMRN2016},
we have the exact functor
$$\KL^{[\lam]}\ra \LZ\Mod^{[\lam]},
\quad M\mapsto H_{DS}^0(M).$$

For $\lam\in P_+$,
define the $\LZ$-module
\begin{align}
\mf{z}_{\lam}:=H_{DS}^0(\mathbb{V}_{\lam}).
\end{align}
Note that $\LZ_{>0}$  trivially acts on $\mf{z}_{\lam}$.
The grading of $\mathbb{V}_\lam$ induces a grading
on
 $\mf{z}_{\lam}$ such that
\begin{align*}
\mf{z}_{\lam}=\bigoplus_{d\in \Z_{\geq 0}}(\mf{z}_{\lam})_{-\lam(\rho^\vee)+d},
\quad (\mf{z}_{\lam})_{-\lam(\rho^\vee)}=\C,
\end{align*}
see \cite[4.6]{Ara07} for the details.
By \eqref{eq:ch-weyl}
and the Euler-Poincar\'{e} principle,
we have
\begin{align}\label{eq:ch-zlam}
\ch \mf{z}_{\lam}=\frac{q^{-\lam(\rho^{\vee})}\prod_{\alpha\in \Delta_+}(1-q^{\bra \lam+\rho,\alpha^{\vee}\ket})}{\prod_{j=1}^{\infty}(1-q^j)^{\on{rk}\g}},
\end{align}
that is,
\begin{align}
\ch \mathbb{V}_{\lam}=\ch \mf{z}_{\lam}. \ch \mathbb{L}_{\lam}.
\label{eq:ch-is-product}
\end{align}

Consider Li's canonical filtration $F^{\bullet}\mf{z}_{\lam}$ of $\mf{z}_{\lam}$.
Since $\mf{z}(\affg)$ is commutative,
$P_{i,(n)}F^{p}\mf{z}_{\lam}\subset F^{p}\mf{z}_{\lam}$ for all $i$, $p$ and $n\geq 0$.
Hence,
each graded subspace
$F^{p}\mf{z}_{\lam}/F^{p+1}\mf{z}_{\lam}$ of $\gr{\mf{z}_{\lam}}$ 
is naturally a module over $\LZ$.
In particular,
\begin{align*}
\overline {\mf{z}_{\lam}}:=\mf{z}_{\lam}/F^1 \mf{z}_{\lam}=\mf{z}_{\lam}/\LZ_{(<0)}\mf{z}_{\lam}
\end{align*}
is a  $\LZ$-module.

%Let $\chi:\g\ra \C$, $x\mapsto (x|f)$.
Since $\chi=(f|?)\in \g^*$ is regular,
there is a surjective algebra homomorphism
\begin{align}
\LZ\twoheadrightarrow \mc{A}_{\chi}\subset U(\g),
\label{eq:Rybnikov}
\end{align}
constructed by Rybnikov \cite{Ryb06},
where $\mc{A}_{\chi}$ denotes the {\em Mishchenko-Fomenko
subalgebra} of $U(\g)$ associated with $\chi$.

The following assertion was essentially proved in 
\cite{FeiFreRyb09,FeiFreTol10}.
\begin{Pro}\label{Pro:cyclic}
Let $\lam\in P_+$.
\begin{enumerate}
\item 
$\mf{z}_{\lam}$ is a free $\LZ_{(<0)}$-module.
\item There is an isomorphism
$\overline {\mf{z}_{\lam}}\cong V_{\lam}$ of $\LZ$-modules,
where $\LZ$ acts on $V_{\lam}$ via the map \eqref{eq:Rybnikov}.
\item The $\LZ$-module $\mf{z}_{\lam}$ is  cyclic,
and is generated by the image  of the highest weight vector of $\mathbb{V}_{\lam}$.
\item
The $\LZ$-module $\overline{\mf{z}}_{\lam}$ has a unique socle spanned by a homogeneous vector 
of maximal degree $\lam(\rho^{\vee})$.
\end{enumerate}
\end{Pro}
\begin{proof}
We have
$\gr \mathbb{V}_{\lam}=\mc{O}(J_{\infty}\g^*)\otimes V_{\lam} $,
and
 by  \cite{Ara09b},
\begin{align*}
\gr \mf{z}_{\lam}\cong (\mc{O}(J_{\infty}(\chi+\mf{n}^{\bot}))\otimes V_{\lam})^{J_{\infty}N},
\end{align*}
where 
$N$ is the unipotent subgroup of $G$ corresponding to the maximal nilpotent
subalgebra $\mf{n}=\bigoplus_{j>0}\g_{j}$ of $\g$.
(Here we have taken the $\mf{sl}_2$-tripe in Section \ref{section:Quantum Drinfeld-Sokolov reduction and equivariant affine $W$-algebras}
as the one associated with $f_{prin}$.)
Since
the action map gives the isomorphism
$J_{\infty}N\times J_{\infty}\mc{S}_f\isomap J_{\infty}(\chi+\mf{n}^{\bot})$,
we have
\begin{align*}
\gr \mf{z}_{\lam}\cong (\mc{O}(J_{\infty}(\chi+\mf{n}^{\bot}))^{J_{\infty}N}\otimes V_{\lam}
\cong \mc{O}(J_{\infty}\mc{S}_f)\otimes V_{\lam}
\end{align*}
as $\mc{O}(J_{\infty}\mc{S}_f)$-modules.
Hence, 
$\gr \mf{z}_{\lam}$ is free over $\mc{O}(J_{\infty}\mc{S}_f)$,
and thus,
$\mf{z}_{\lam}$ is free over 
$\LZ_{(<0)}$.
Also,
we get that
\begin{align}
\overline {\mf{z}_{\lam}}\cong \gr \mf{z}_{\lam}/F^1 \gr \mf{z}_{\lam}\cong V_{\lam} 
\label{eq:iso-fdr}
\end{align}
 as vector spaces.
Therefore,
$\LZ$ acts on $V_{\lam}$ 
via the identification \eqref{eq:iso-fdr},
 and one finds that this action is identical to
  the action
  defined by
  \eqref{eq:Rybnikov},
  see
 \cite[\S 2.1]{FeiFreRyb09}.
Since 
$V_{\lam}$ is a cyclic $\mc{A}_{\chi}$-module generated by the highest weight vector (\cite{FeiFreRyb10}),
it follows that  $\mf{z}_{\lam}$ is  a cyclic $\LZ_{<0}$-module generated by the image of the highest weight vector of $\mathbb{V}_{\lam}$.
According to \cite{FeiFreTol10},
the $\mc{A}_{\chi}$-module $V_{\lam}$ has a unique socle,
and hence, so does $\overline{\mf{z}_{\lam}}$.
As $\overline{\mf{z}_{\lam}}$ is cyclic,
the socle must be
spanned by a homogeneous vector 
of maximal degree,
which is
$-(w_0 \lam)(\rho^{\vee})=\lam(\rho^{\vee})$.
\end{proof}
For a $\mc{Z}$-module $M$, we set
\begin{align*}
M^*:=\on{Hom}_{\LZ_{(<0)}}(M,\LZ_{(<0)})
\end{align*}
(see \eqref{eq:Z(<0)}),
and consider $M^*$ as a $\mc{Z}$-module by the action
$(zf)(m)=f(zm)$.
Also,
let $\tau^* M$ denote the $\LZ$-module 
obtained by twisting the action of $\LZ$ on $M$
as $z.m=\tau(z)m$.

\begin{Pro}\label{Pro:dual-Z}
For $\lam\in P_+$,
we have the following:
\begin{enumerate}
\item $\mf{z}_{\lam}^*\cong \mf{z}_{\lam}$,
\item $\tau^*(\mf{z}_\lam)\cong \mf{z}_{\lam^*}$.
\end{enumerate}
\end{Pro}
\begin{proof}
(i) Clearly,
$I_{\lam}$ annihilates $\mf{z}_{\lam}^*$.
Since $\mf{z}_{\lam}$ is free over 
$\LZ_{(<0)}$,
$\mf{z}_{\lam}^*$ is free over $\LZ_{(<0)}$ as well.
We have
\begin{align*}
\overline{\mf{z}_{\lam}^*}\cong \Hom_{\C}(\overline{\mf{z}_{\lam}},\C)
\end{align*}
as $\LZ$-modules,
where $\overline {\mf{z}_{\lam}^*}=\mf{z}_{\lam}^*/\LZ_{(<0)}\mf{z}_{\lam}^*$.
Since $\overline{\mf{z}_{\lam}}$ has a unique socle by Proposition \ref{Pro:cyclic},
$\overline{\mf{z}_{\lam}^*}$ is cyclic.
Therefore,
$\mf{z}_{\lam}^*$ is  cyclic,
and 
we have a surjection
$\mf{z}_{\lam}= \LZ_{<0}/I_{\lam}\ra \mf{z}_{\lam}^*$.
The assertion follows since  $\mf{z}_{\lam}^*$  has the same rank
as $\mf{z}_{\lam}$ as a  $\LZ_{(<0)}$-module.
(ii) By \eqref{eq:FG-theorem},
$I_{\lam}=\on{Ann}_{\LZ_{<0}}\mathbb{V}_{\lam}$.
Consider the subspace 
$M=V_{\lam}\* V_{\lam^*}\subset \mc{O}(G)\subset \Dch_{G}$
The $\affg_{\kappa_c}$-submodule
of $\Dch_{G}$
generated by $M$
by the action $\pi_L$ (resp.\ $\pi_R$)
is isomorphic
to $\mathbb{V}_{\lam}\* V_{\lam^*}$
(resp.\ $V_{\lam}\* \mathbb{V}_{\lam^*}$).
It follows that 
$I_{\lam}=\{z\in \LZ_{<0}\mid \pi_L(z)M=0\}$
and 
$I_{\lam^*}=\{z\in \LZ_{<0}\mid \pi_R(z)M=0\}$.
Hence
$\tau(I_{\lam})=I_{\lam^*}$
by \eqref{eq:effect-of-tau-to-FF}.
\end{proof}

Let $v_{\lam}\in \mf{z}_{\lam}$
be the image of the highest weight vector of $\mathbb{V}_{\lam}$.
By Proposition \ref{Pro:cyclic},
the map
$\LZ_{<0}\ra \mf{z}_{\lam}$, $z\mapsto zv_{\lam}$, is surjective,
and thus,
\begin{align}
\mf{z}_{\lam}\cong \LZ_{<0}/I_{\lam},
\end{align}
where $I_{\lam}$ is the kernel of the above map.
In particular,
$\mf{z}_{\lam}$ has the structure of a unital commutative algebra.
In \cite{FreGai04,FeiFreRyb09} it was shown that
\begin{align}
\mf{z}_{\lam}\cong \End_{\affg_{\kappa_c}}(\mathbb{V}_{\lam}),
\quad \mathbb{L}_{\lam}= \mathbb{V}_{\lam}/(\mf{z}_{\lam})_+ \mathbb{V}_{\lam},
\label{eq:FG-theorem}
\end{align}
where $(\mf{z}_{\lam})_+$ is the argumentation ideal of $\mf{z}_{\lam}$.
%and $\mathbb{V}_{\lam}$ is free over $\mf{z}_{\lam}$.

By \eqref{eq:ch-zlam} and Proposition \ref{Pro:cyclic},
we have
\begin{align}
\ch \overline{\mf{z}}_{\lam}=\frac{q^{-\lam(\rho^{\vee})}\prod_{\alpha\in \Delta_+}(1-q^{\bra \lam+\rho,\alpha^{\vee}\ket})}{
\prod_{i=1}^{\on{rk}\g}\prod_{j=1}^{d_i}(1-q^j)}=q^{-\lam(\rho^{\vee})}\prod_{\alpha\in \Delta_+}
\frac{(1-q^{\bra \lam+\rho,\alpha^{\vee}\ket})}{
(1-q^{\bra\rho,\alpha^{\vee}\ket})}.
\label{eq:chbarz}
\end{align}
%which can also be  obtained by the standard method \cite[10.9]{Kac90}.

Let 
$\C_{[\lam]}$
be the 
unique simple graded quotient of $\mf{z}_{\lam}$,
that is,
  the one-dimensional graded $\LZ$-module corresponding to $\chi_{\lam}$,
which is concentrated in degree $-\lam(\rho^{\vee})$.

\begin{Pro}\label{Pro:simple-go-simple}
For $\lam\in P_+$, we have
$H_{DS}^0(\mathbb{L}_{\lam})\cong \C_{[\lam]}$.
 \end{Pro}
\begin{proof}
The surjection $\mathbb{V}_{\lam}\ra \mathbb{L}_{\lam}$ 
induces a surjection $\mf{z}_\lam\ra H_{DS}^0(\mathbb{L}_{\lam})$.
Hence $H_{DS}^0(\mathbb{L}_{\lam})$ is cyclic 
by Proposition \ref{Pro:cyclic}.
Since 
$\ker \chi_{\lam}$ annihilates
it,
$H_{DS}^0(\mathbb{L}_{\lam})$ must be one-dimensional.
\end{proof} 
\begin{Pro}\label{Pro:ch-DS}
For $\lam\in P_+$,
the functor
$\on{KL}^{[\lam]}\ra \LZ\Mod^{[\lam]}$ is faithfull.
For $M\in \on{KL}_{[\lam]}$,
we have
\begin{align*}
\ch M=q^{\lam(\rho^{\vee})} \ch \mathbb{L}_\lam . \ch H_{DS}^0(M).
\end{align*}
\end{Pro}
\begin{proof}
The first statement follows from the exactness of the functor 
and Proposition~\ref{Pro:simple-go-simple}.
To see the second assertion,
note that we have
\begin{align*}
\ch M=\sum_{d\in \C}[M:\mathbb{L}_{\lam}[-d]]q^d\ch \mathbb{L}_{\lam},
\end{align*}
where
%$\mathbb{L}_{\lam}[-d]$ is the module $\mathbb{L}_{\lam}$ whose grading is shifted so that its highest weight vector has degree $d$,
%and 
$[M:\mathbb{L}_{\lam}[-d]]$ is the multiplicity of $\mathbb{L}_{\lam}[-d]$ 
as a graded representation.
Since $H_{DS}^0(?)$ is exact,
we have
\begin{align*}
\ch H_{DS}^0( M)=\sum_{d\in \C}[M:\mathbb{L}_{\lam}[-d]]q^d
\ch H_{DS}^0(\mathbb{L}_{\lam})=q^{-\lam(\rho^{\vee})}\sum_{d\in \C}[M:\mathbb{L}_{\lam}[-d]]q^d.
\end{align*}
This completes the proof.
\end{proof}

Let $\Gamma_-^p \mathbb{V}_{\lam} = (\LZ_{<0}^*)^p \mathbb{V}_{\lam}$,
where $\LZ_{<0}^*$ is the  augmentation ideal of $\LZ_{<0}$.
Then $\Gamma_-^\bullet  \mathbb{V}_{\lam}$ defines a decreasing filtration of $\mathbb{V}_{\lam}$ as
$\affg_{\kappa_c}$-modules.
By \eqref{eq:FG-theorem},
we have 
\begin{align}
\gr^{\Gamma_-}  \mathbb{V}_{\lam}\cong \mf{z}_\lam\*  \mathbb{L}_{\lam}
\end{align}
as $\LZ\* U(\affg_{\kappa_c})$-modules. 
Dually, for $p<0$ set 
$\Gamma_+^p D(\mathbb{V}_{\lam})=\{v\in D(\mathbb{V}_{\lam}) \mid (\LZ_{>0}^*)^{-p}v=0\}$,
where 
%$D(\mathbb{V}_{\lam})$ is the contragredient dual of $\mathbb{V}_{\lam}$ and
$\LZ_{>0}^*$ is the  augmentation ideal of $\LZ_{>0}$.
Then $\Gamma_+^\bullet  D(\mathbb{V}_{\lam})$ defines a decreasing filtration of $D(\mathbb{V}_{\lam})$ as
$\affg_{\kappa_c}$-modules,
and we have
\begin{align}
\gr^{\Gamma_+}  D(\mathbb{V}_{\lam})\cong D(\mf{z}_\lam)\*  \mathbb{L}_{\lam},
\end{align}
where 
$D(\mf{z}_\lam)$ is the contragredient dual of $\mf{z}_\lam$.

\begin{Pro}\label{Pro:reduction-of-the-dual}
For $\lam\in P_+$ we have
$H_{DS}^0(D(\mathbb{V}_{\lam}))\cong D(\mf{z}_\lam)$.
\end{Pro}
\begin{proof}
Consider the spectral sequence $E_r\Rightarrow H_{DS}^0(D(\mathbb{V}_{\lam}))$ associated with the filtration 
$\Gamma_+^p D(\mathbb{V}_{\lam})$.
We have
$E_1^{\bullet,q}=D(\mf{z}_\lam)\* H_{DS}^q(\mathbb{L}_{\lam})\cong \delta_{q,0}D(\mf{z}_\lam)$
by Proposition \ref{Pro:simple-go-simple}.
The spectral sequence collapses at $E_1=E_\infty$,
and we get that $H_{DS}^0(D(\mathbb{V}_{\lam}))\cong D(\mf{z}_\lam)$.
\end{proof}

For $b\geq 1$,
define
the vertex algebra
\begin{align}
\Vb:=\overbrace{V^{\kappa_c}(\g)\otimes _{\mf{z}(\affg)}V^{\kappa_c}(\g)\otimes _{\mf{z}(\affg)} \dots \otimes _{\mf{z}(\affg)}V^{\kappa_c}(\g)}^{b\text{ times}},
\label{eq;b-copy}
\end{align}
where
the tensor product $\otimes _{\mf{z}(\affg)}$
is taken with respect to the action
$z\* 1-1\* \tau(z)$ for $z\in \mf{z}(\affg)$.
Let 
\begin{align}
\iota_i: V^{\kappa_c}(\g)\hookrightarrow \Vb
\label{eq:iota}
\end{align}
 be the 
vertex algebra embedding that sends $u\in V^{\kappa_c}(\g)$ to the $i$-th factor of 
$\Vb$.
This induces a Lie algebra homomorphism
$\affg_{\kappa_c}\ra \End_{\C}(M)$
for any $\Vb$-module $M$,
which we denote also by $\iota_i$.

Let $\KL_b$ 
be the category of $\Vb$-modules 
consisting of objects 
$M$ that belongs to $\KL$
as a $\iota_i(\affg_{\kappa_c})$-module for all 
 $i=1,\dots, b$.

 By Proposition \ref{Pro:cdosub},
 we have a vertex algebra embedding
 $V^{\kappa_c}(\g)^{\*_{\mf{z}(\g)} b=2}\hookrightarrow \Dch_G$.
 In particular,
 $\Dch_G$ is an object of  $\KL_{b=2}$.
 
Set
\begin{align*}
&\mathbb{V}_{\lam,b}:=
\overbrace{\mathbb{V}_{\lam}\otimes _{\LZ}\mathbb{V}_{\lam^*}\otimes _{\LZ} \dots \otimes _{\LZ}\mathbb{V}_{\lam'}}^{b\text{ times}}
=\overbrace{\mathbb{V}_{\lam}\otimes _{\mf{z}_{\lam}}\mathbb{V}_{\lam^*}\otimes _{\mf{z}_{\lam^*}} \dots \otimes _{\mf{z}_{\lam'}}\mathbb{V}_{\lam'}}^{b\text{ times}}\in \KL_b,
\end{align*}
where  $\lam'=\lam$ if $b$ is odd and $\lam'=\lam^*$ of $b$ is even.

\begin{Pro}\label{Pro:r-weyl}
Let  $b\geq 2$.
\begin{enumerate}
\item We have
$H_{DS}^0(\mathbb{V}_{\lam,b})\cong \mathbb{V}_{\lam^*,b-1}$ 
(resp.\ $\cong \mathbb{V}_{\lam,b-1}$),
where the Drinfeld-Sokolov reduction is taken with respect to the action
$\iota_1$ (resp.\ $\iota_r$) of $\affg_{\kappa_c}$.
\item
We have
$H_{DS}^0(D(\mathbb{V}_{\lam,b}))\cong D(\mathbb{V}_{\lam^*,b-1})$ 
(resp.\ $\cong \mathbb{V}_{\lam,b-1}$),
where $D(\mathbb{V}_{\lam,b})$
is the contragradient dual of $\mathbb{V}_{\lam,b}$
and
 the Drinfeld-Sokolov reduction is taken with respect to the action
$\iota_1$ (resp.\ $\iota_r$) of $\affg_{\kappa_c}$.
\end{enumerate}
\end{Pro}
\begin{proof}
We only show the assertion for the action $\iota_1$.
(i)
By definition,
we have $\mathbb{V}_{\lam,b}=\mathbb{V}_{\lam}\*_{\mf{z}_{\lam}}\mathbb{V}_{\lam^*,b-1}$.
Since $H_{DS}^\bullet(\mathbb{V}_{\lam})\cong \mf{z}_{\lam}$ is obviously free over $\mf{z}_{\lam}$,
$H_{DS}^0(\mathbb{V}_{\lam,b})\cong H_{DS}^0(\mathbb{V}_{\lam})\*_{\mf{z}_{\lam}}\mathbb{V}_{\lam^*,b-1}
\cong \mf{z}_{\lam}\*_{\mf{z}_{\lam}}\mathbb{V}_{\lam^*,b-1}\cong \mathbb{V}_{\lam^*,b-1}$
by  K\"{u}nneth's theorem.
(ii) We have $D(\mathbb{V}_{\lam,b})=\{v\in D(\mathbb{V}_{\lam})\* D(\mathbb{V}_{\lam^*,b-1})\mid
(z\* 1)v=(1\* \tau(z))v\}$.
By Proposition \ref{Pro:reduction-of-the-dual},
the embedding $D(\mathbb{V}_{\lam,b})\ra D(\mathbb{V}_{\lam})\* D(\mathbb{V}_{\lam^*,b-1})$
induces the embedding 
$$H_{DS}^0(D(\mathbb{V}_{\lam,b}))\hookrightarrow D(\mf{z}_{\lam})\* D(\mathbb{V}_{\lam^*,b-1}),$$
and the image is contained in 
$\{v\in D(\mf{z}_{\lam})\* D(\mathbb{V}_{\lam^*,b-1})\mid
(z\* 1)v=(1\* \tau(z))v\}$,
which is contragredient dual of  
$\mf{z}_{\lam}\*_{\mf{z}_{\lam}}\mathbb{V}_{\lam^*,b-1}\cong \mathbb{V}_{\lam^*,b-1}$.
Since  $H_{DS}^0(D(\mathbb{V}_{\lam,b}))$
has the same
 character as  $H_{DS}^0(\mathbb{V}_{\lam,b})=\mathbb{V}_{\lam^*,b-1}$,
 we get that $H_{DS}^0(D(\mathbb{V}_{\lam,b}))=D(\mathbb{V}_{\lam^*,b-1})$.
\end{proof}

Let $\KL_b^{\Delta}$ be the full subcategory of $\KL_b$ consisting of objects $M$ that admits an increasing
 filtration
 $0=M_0\subset M_1\subset \dots $,
 $M=\bigcup_p M_p$,
 such that each successive quotient $M_p/M_{p-1}$ is isomorphic to $\mathbb{V}_{\lam,b}$ for some $\lam\in P_+$.
 Dually,
 let  $\KL_b^{\nabla}$
 be the full subcategory of $\KL_b$ consisting of objects $M$ that admit a decreasing
 filtration
 $M=M_0\supset M_1\supset \dots $,
 $\bigcap_p M_p=0$,
 such that each successive quotient $M_p/M_{p+1}$ is isomorphic to the 
 contragredient dual $D(\mathbb{V}_{\lam,b})$ of $\mathbb{V}_{\lam,b}$ for some $\lam\in P_+$.
 Let $\KL_r^{\Delta}\cap \KL_r^{\nabla}$
 be the full subcategory of $\KL_r$ consisting of objects $M$ that belong to both $\KL_r^{\Delta}$  and $\KL_r^{\nabla}$.

%In Proposition \ref{Pro:freeness}
%we will show that 
%$\Dch_G\in \KL_2^{\Delta}\cap \KL_2^{\nabla}$.

\section{The BRST cohomology $H^{\semiinf+\bullet}(\LZ,?)$}
\label{Sec:main-construction}
%From now on we set $\kappa=\kappa_c=-\frac{1}{2}\kappa_\g$,
%$f=f_{prin}$.
Let $\bigwedge^{\semiinf+\bullet }({\mf{z}(\affg)})$
be the charged free Fermions generated by odd fields
\begin{align*}
\psi_1(z),\dots, \psi_{\on{rk}\g}(z),\psi^*_1(z),\dots, \psi^*_{\on{rk}\g}(z)
\end{align*}
with OPEs
\begin{align*}
\psi_i(z)\psi_j^*(z)\sim \frac{\delta_{ij}}{z-w},\quad
\psi_i(z)\psi_j(z)\sim\psi_i^*(z)\psi_j^*(z)\sim0.
\end{align*}
Define
\begin{align*}
T^{gh}(z)=\sum_{n\in \Z}L^{gh}_nz^{-n-2}:=\sum_{i=1}^{\on{rk}\g} d_i:\psi_i^*(z)\partial_z \psi_i(z):-\sum_{i=1}^{\on{rk}\g} (d_i+1):\psi_i(z)\partial \psi_i^*(z):\\
=\sum_{i=1}^{\on{rk}\g} :(\partial \psi_i(z))\psi_i^*(z):-\sum_{i=1}^{\on{rk}\g} (d_i+1):\partial(\psi_i(z)\psi_i^*(z)):.
\end{align*}
Then $T^{gh}(z)$ defines a conformal vector of $\bigwedge^{\semiinf+\bullet }({\mf{z}(\affg)})$
with central charge $$-2(\on{rk}\g+24(\rho|\rho^{\vee})).$$
The conformal weight 
of $\psi_i$ and $\psi_i^*$ are
$d_i+1$
and $-d_i$,
respectively.
%It follows that 
%$\bigwedge^{\semiinf+\bullet }({\mf{z}(\affg)})$
%is not conical,
%but still
%We have
%$\Delta_{\psi_i}=d_i+1$,
%$\Delta_{\psi_i^*}=-d_i$.
%Let
% $\bigwedge^{\semiinf+\bullet }({\mf{z}(\affg)})=\bigoplus_{\Delta\in \Z} \bigwedge^{\semiinf+\bullet }({\mf{z}(\affg)})_{\Delta}$.
%Nevertheless,   
The conformal weight of $\bigwedge^{\semiinf+\bullet }({\mf{z}(\affg)})$ is bounded from below 
 and each homogeneous space $\bigwedge^{\semiinf+\bullet }({\mf{z}(\affg)})_{\Delta}$ is finite-dimensional.
% \begin{align}
%\on{ch}\bigwedge^{\semiinf+\bullet }({\mf{z}(\affg)})=\prod_{i=1}^r\prod_{n=1}^\infty (1+zq^{n-d_i-1})(1+z^{-1}q^{n+d_i-2}).
%\end{align}

Define the odd field
\begin{align}
\QZ(z)=\sum_{n\in \Z}\QZ_{(n)}z^{-n-1}:=\sum_{i=1}^{\on{rk}\g}P_i(z)\otimes \psi_i^*(z)
\label{eq:QZ}
\end{align}
on the vertex algebra $\mf{z}(\affg)\otimes \bigwedge^{\semiinf+\bullet }({\mf{z}(\affg)})$.
So
\begin{align}
\QZ_{(0)}=\sum_{i=1}^{\on{rk}\g}\sum_{n\in \Z}(p_i)_{(-n)}\*(\psi_i^*)_{(n-1)}.
%=\sum_{i=1}^{\on{rk}\g}\sum_{n\in \Z}(p_i)_{-n-d_i}(\psi_i^*)_{n}
\end{align}
Since $\mf{z}(\affg)$ is commutative, we have $\QZ(z)\QZ(w)\sim 0$, and thus,
$(\QZ_{(0)})^2=0$.
Hence, for $M\in \LZ\Mod$,
$(M\otimes \bigwedge^{\semiinf+\bullet }({\mf{z}(\affg)}), \QZ_{(0)})$  is a cochain complex,
where the cohomological degree is defined by
$\deg (\psi_i)_{(n)}=-1$, $\deg (\psi_i)^*_{(n)}=1$,
$\deg m\* |0\rangle=0$ for $m\in M$.
Denote by
$H^{\semiinf+\bullet}(\LZ,M)$ the corresponding cohomology.
If $V$   is a vertex algebra object in $\KL$
and $M$ is a $V$-module,
then $H^{\semiinf+\bullet}(\LZ,V)$ is naturally a vertex algebra
and $H^{\semiinf+\bullet}(\LZ,M)$  is a $H^{\semiinf+\bullet}(\LZ,V)$-module.

\begin{Lem}\label{Lem:vanihing-negative}
Let $M\in\LZ\Mod$
and suppose that 
$M$ is free as a $\LZ_{(<0)}$-module.
Then
$H^{\semiinf+i}(\LZ, M)=0$
for $i<0$.
\end{Lem}
\begin{proof}
Immediate from \cite[Theorem 2.3]{Vor93}
(cf.\ the proof of Lemma \ref{Lem:BRST-cohomology-and-ext}
below).
\end{proof}

In the followings, for  $M,N\in \LZ\Mod$ we consider $M\*N$ as a $\LZ$-module by the action
\begin{align}
P_i(z)\mapsto P_i(z)\*1-1\*\tau(P_i)(z)
\end{align}
unless otherwise stated.
%Also,
%for a $\LZ$-module $M$,
%let
%\begin{align*}
%M^*:=\Hom_{\FF}(M,\mf{z}(\affg)),
%\end{align*}
%and regard $M^*$ as a $\LZ$-module by the action
%\begin{align*}
%((P_i)_{(n)}f)(m)=f(\tau(P_i)_{(n)}m),\quad n\in\Z, \ f\in \Hom_{\FF)}(M,\mf{z}(\affg)),\
%m\in M.
%\end{align*}

\begin{Lem}\label{Lem:BRST-cohomology-and-ext}
Let $M,N\in \LZ\Mod$  
and suppose that $M$ is free of finite rank as a $\LZ_{(<0)}$-module.
Then
$H^{\semiinf+i}(\LZ,M\*N)=0$ for $i<0$
and
\begin{align*}
H^{\semiinf+0}(\LZ,M\*N)\cong \Hom_{\LZ}( M^*,\tau^* N).
\end{align*}
\end{Lem}
\begin{proof}
Consider the  Hochschild-Serre spectral sequence
$E_r\Rightarrow H^{\semiinf+\bullet}(\LZ,M\*N)$
 for the subalgebra
$\LZ_{(<0)}\subset \LZ$ as in \cite[Theorem 2.3]{Vor93}.
By definition, the $E_1$-term is 
the opposite Koszul homology
$H_\bullet^{op}(\LZ_{(<0)},M\*N)$
of the $\LZ_{(<0)}$-module $M\* N$.
Since $M$ is free over $\LZ_{(<0)}$,
$M$ is also free over $\LZ_{(<0)}$,
and so is $M\* N$.
Hence 
$H_i^{op}(\LZ_{(<0)},M\*N)=0$ for $i\ne  0$
and $H_0^{op}(\LZ_{(<0)},M \*N)\cong M\*_{\LZ_{(<0)}} N$.
It follows that we have
$E_2^{p,q}=\delta_{q,0}H^p(\LZ_{(\geq 0)}, M\*_{\LZ_{(<0)}} N)$,
the Lie algebra cohomology 
of $M$ as a module over the commutative Lie algebra 
$\bigoplus_{i=1}^{\on{rk}\g}\bigoplus_{n\geq 0}\C P_{i,(n)}$
with respect to the action $(P_i)_{(n)}\mapsto (P_i)_{(n)}\*1-1\* \tau(P_i)_{(n)}$.
Hence the spectral sequence collapses at $E_2=E_{\infty}$, and we have
\begin{align*}
H^{\semiinf+i}(\LZ,M\*N)\cong \begin{cases}H^i(\LZ_{(\geq 0)}, M\*_{\LZ_{(<0)}} N)&\text{for }i\geq 0\\
0&\text{for }i<0.
\end{cases}
\end{align*}
Since $M$ is free $\LZ_{(<0)}$-module of finite rank,
we have
\begin{align*}
 M\*_{\LZ_{(<0)}}N%=\Hom_{\mf{z}(\affg)}(M,\mf{z}(\affg))\*_{\mf{z}(\affg)}N
 \cong 
 \Hom_{\LZ_{(<0)}}(M^*,\tau^* N).
\end{align*}
It follows from the definition that 
$$H^{\semiinf+0}(\LZ,M \*N)\cong 
%H^0(\LZ_{(\geq 0)}, \Hom_{\LZ_{(<0)}}(M,N))=
\Hom_{\LZ}(M^*,\tau^* N)\subset \Hom_{\LZ_{(<0)}}(M^*,\tau^* N).$$
%Moreover, since 
%$H^{\semiinf+i}(\LZ,M^*\* ?)$ is the 
%$i$-th right derived functor of 
%$H^{\semiinf+0}(\LZ,M^*\* ?)=\Hom_{\LZ}(M,?)$,
%we get that 
%$$H^{\semiinf+i}(\LZ,M^* \* N)\cong 
%\on{Ext}^i_{\LZ}(M,N)$$
%for $i\geq 0$.
\end{proof}

For $M,N \in \LZ\Mod$,
the complex
$C(\LZ,M\* N)=M\* N\*
\bigwedge^{\semiinf+\bullet }({\mf{z}(\affg)})$
is a direct sum of subcompelxes
%$C(\LZ,M\* N)=\bigoplus_{d\in \C}
%C(\LZ,M\* N)_d$,
$C(\LZ,M\* N)_d=\bigoplus\limits_{d_1+d_2+\Delta=d}
M_{d_1}\* N_{d_2}\*
\bigwedge^{\semiinf+\bullet }({\mf{z}(\affg)})_{\Delta}$,
$d\in \C$,
and so $H^{\semiinf+\bullet}(\LZ,M\*N)$ is graded:
$$H^{\semiinf+\bullet}(\LZ,M\*N)=
\bigoplus_{d\in \C}
H^{\semiinf+\bullet}(\LZ,M\*N)_d.$$
Set
$\ch_q H^{\semiinf+\bullet}(\LZ,M\*N)
=\sum_{d\in \C}q^d \dim H^{\semiinf+\bullet}(\LZ,M\*N)_d $
when it is well-defined.

\begin{Pro}\label{Pro:vanishing-1}
 Let $\lam \in P_+$,
 $M\in \LZ\on{-Mod}$.
 We have
 $$H^{\frac{\infty}{2}+0}(\LZ,\mf{z}_{\lam}\otimes M)\cong
 \Hom_{\LZ}(\mf{z}_{\lam},\tau^* M).$$
 If $\tau^* M$ is a quotient of $\mf{z}_{\lam}$,
 we have
$$\ch H^{\frac{\infty}{2}+0}(\LZ,\mf{z}_\lam\otimes M)\cong 
q^{\lam(\rho^{\vee})}\ch M.
%\begin{cases}M[d]&\text{for }i=0
%\text{ and }\mu=\lam^*\\
%0&\text{otherwise}\end{cases}
$$
%with $d=\bra \lam,\rho^{\vee}\ket$
%as graded vector space.
\end{Pro}
\begin{proof}
By Lemma \ref{Lem:BRST-cohomology-and-ext} and Proposition \ref{Pro:dual-Z},
$H^{\frac{\infty}{2}+0}(\LZ,\mf{z}_\lam\otimes M)\cong
\Hom_{\LZ}(\mf{z}_{\lam},\tau^* M)$.
The latter is isomorphic to $\tau^* M$ if $\tau^* M$ is quotient 
of $\mf{z}_{\lam}$ as a $\LZ$-module.
Since 
under the isomorphism
$\mf{z}_\lam^*\cong \mf{z}_{\lam}$ in Lemma \ref{Lem:BRST-cohomology-and-ext},
the generator  $v_{\lam}$ of $\mf{z}_{\lam}$ corresponds to the dual element
of $\mf{z}_{\lam}$ of conformal weight $\lam(\rho^{\vee})$,
we get the assertion.
\end{proof}

Set
$$
\eqW=\eqW_G=\eqW^{\kappa_c}_{G,f_{prin}},\quad
%\mathbf{I}_G=\mathbf{I}_{G}^{\kappa_c},\quad
  \mathbf{S}=\mathbf{S}_G=\mathbf{S}_{G,f_{prin}},\quad \Slo=\Slo_{f_{prin}}.$$
%so that  $\gr \eqW\cong \mc{O}(J_{\infty}\mathbf{S})=\mc{O}(J_{\infty}G\times J_{\infty}\Slo)$.
%Unless otherwise stated,
% we consider $\eqW$ as $\LZ$-module by the action $\pi_L$.

 By \eqref{eq:cc-of-eq-W},
 $\eqW$ is conformal with central charge
 \begin{align}
\dim \g+\on{rk}\g+24(\rho|\rho^{\vee}).
\end{align}
Let $\omega_{\eqW}$ be the 
 conformal vector of $\eqW$,
 $L^{\eqW}(z)=\sum_{n\in \Z}L_n^{\eqW}z^{-n-2}$
 be the corresponding field.
 \begin{Lem}\label{Lem:action-of-conformal-v}
The subspace $\mf{z}(\affg)\subset \eqW$
is preserved by the action of 
$L_n^{\eqW}$ with $n\geq -1$.
\end{Lem}
\begin{proof}
Clearly,
$\mf{z}(\affg)$ is preserved by the action of the translation operator $L_{-1}^{\eqW}$.
Let $z\in \mf{z}(\affg)=\eqW^{\g[t]}$
(see Proposition \ref{Pro:g[t]t-invariant-part} and \eqref{eq:FF-iso}).
Since $[L_m^{\eqW},x_n]=(m-n)x_{m+n}$ for $x\in \g$,
we have $x_m L_n^{\eqW} z=-(n-m)x_{m+n}z=0$
for $m,n\geq 0$.
Hence $L_n^{\eqW} z\in \mf{z}(\affg)$ for $n\geq 0$.
\end{proof}
 
  For $N\in \LZ\on{-Mod}_{\on{reg}}$,
 $H^{\frac{\infty}{2}+0}(\LZ,\eqW\otimes  N)$
 is a graded $V^{\kappa_c}(\g)$-module,
 where $V^{\kappa_c}(\g)$ acts on the first factor $\eqW$.
 Hence we have a functor
 \begin{align}
 \LZ\on{-Mod}
%\LZ\on{-Mod}_{\on{reg}}
\ra \on{KL},\quad
N\mapsto H^{\frac{\infty}{2}+0}(\LZ,\eqW\otimes  N).
\label{eq:the-main-functor}
\end{align}
\begin{Lem}\label{Lem:vnd}
We have
$H^{\frac{\infty}{2}+i}(\LZ,\eqW\otimes  N)=0$ for $i<0$,
$N\in \LZ\Mod$.
Therefore,
the functor \eqref{eq:the-main-functor}
is left exact.
\end{Lem}
\begin{proof}
Since $SS(\eqW)\cong J_{\infty}\mathbf{S}\cong J_{\infty}\Slo\times J_{\infty}G$,
$\gr \eqW$ is   free over $\mc{O}(J_{\infty}\Slo)$.
Hence,
$\eqW$ is free as
a $\LZ_{(<0)}$-module,
% $\mf{z}(\affg)$ in the usual associative sense,
 and so is $\eqW\* N$.
The assertion follows from Lemma \ref{Lem:vanihing-negative}.
\end{proof}

\begin{Pro}\label{Pro:gt-inv-part}
For $N\in \LZ\Mod_{reg}$,
we have
$$H^{\frac{\infty}{2}+0}(\LZ,\eqW\otimes N)^{t\g[t]}\cong 
H^{\frac{\infty}{2}+0}(\LZ,\eqW^{t\g[t]}\otimes N)\\
\cong \bigoplus_{\lam\in P_+}V_{\lam}\* %N[-\lam(\rho^{\vee})]
\Hom_{\LZ}(\mf{z}_{\lam},N)
$$
as $\g$-modules.
\end{Pro}
\begin{proof}
Set
\begin{align*}
C=\bigoplus_{i\in \Z}C^i,\quad
C^i=\bigoplus_{p+q=i}C^{p,q},\quad C^{p,q}=\eqW\otimes N\*  \bigwedge\nolimits^{\semiinf+p }({\mf{z}(\affg)})\otimes  \bigwedge\nolimits^{q }(t^{-1}\g^*[t^{-1}]).
\end{align*}
Let $Q_{t\g[t]}$ be the differential of the Chevalley complex  
$\eqW\* \bigwedge\nolimits^{\bullet }(t^{-1}\g^*[t^{-1}])$
for the computation of the Lie algebra cohomology
 $H^{\bullet}(t\g[t],\eqW)$.
We extend $Q_{t\g[t]}$ 
to a differential of $C$  by letting it  act trivially  on the  factor $N\*  \bigwedge\nolimits^{\semiinf+i }({\mf{z}(\affg)})$.
Let
$Q_{\mf{z}}$
be the 
 differential of the complex 
$\eqW\* N\* \bigwedge\nolimits^{\semiinf+i }({\mf{z}(\affg)})$
which defines $H^{\semiinf+\bullet}(\LZ,\eqW\* N)$.
We extend $Q_{\mf{z}}$ to a differential 
of $C$ by letting it  act trivially  on the factor $\bigwedge\nolimits^{\bullet }((t\g[t])^*)$.
Since $\{Q_{t\g[t]},Q_{\mf{z}}\}=0$,
 $C$ is equipped with the structure of a double complex.

Consider the spectral sequence
 $E_r\Rightarrow H^\bullet_{tot}( C)$  
 such that $d_0=Q_{t\g[t]}$ and $d_1=Q_\mf{z}$.
 Since $\eqW$ is cofree over $U(t\g[t])$,
 we have
 \begin{align*}
E_1^{p,q}&\cong  H^q(t\g[t],\eqW)\* N\* \bigwedge\nolimits^{\semiinf+p }({\mf{z}(\affg)})
\cong \delta_{q,0}\eqW^{t\g[t]}\* N\* \bigwedge\nolimits^{\semiinf+p }({\mf{z}(\affg)})\\
&\cong \delta_{q,0}\bigoplus_{\lam\in P_+}(\mf{z}_{\lam^*}\* V_{\lam})\* N\* \bigwedge\nolimits^{\semiinf+p }({\mf{z}(\affg)})
\end{align*}
by Proposition \ref{Pro:g[t]t-invariant-part}.
It follows  that
\begin{align*}
E_2^{p,q}& \cong \delta_{q,0}\bigoplus_{\lam\in P_+}V_{\lam}\*H^{\semiinf+p}(\LZ, \mf{z}_{\lam^*}\* N).
\end{align*}
Therefore, the spectral sequence collapses at $E_2=E_{\infty}$,
and we obtain the  isomorphism
$H_{tot}^i(C)\cong \bigoplus_{\lam\in P_+}V_{\lam}\*H^{\semiinf+i}(\LZ, \mf{z}_{\lam^*}\* N)$.
In particular,
by Proposition \ref{Pro:vanishing-1},
%by Lemma \ref{Lem:BRST-cohomology-and-ext},
we have
\begin{align}
H_{tot}^0 (C)\cong  \bigoplus_{\lam\in P_+}V_{\lam}\*  \Hom_{\LZ}(\mf{z}_{\lam}, N).
%\begin{cases}  \mathbf{X}^n\* V_{\lam}
%&\text{for }p\geq 0\\
%&\text{for }p<0
%\end{cases}.
\label{eq:iso22}
\end{align}

 Next,
let $E'_r\Rightarrow H^\bullet_{tot}( C)$ 
 be the spectral sequence 
 such that $d_0=Q_\mf{z}$ and $d_1=Q_{t\g[t]}$.
We have
\begin{align}
&(E_1')^{p,q}\cong H^{\semiinf+q}(\LZ,\eqW\* N)\* \bigwedge\nolimits^{p}(t^{-1}\g^*[t^{-1}]),\\
&(E_2')^{p,q}\cong  H^p(t\g[t], H^{\semiinf+q}(\LZ,\eqW\* N)).
\label{eq:second-E2}
\end{align}
By Lemma \ref{Lem:vnd},
$(E_2')^{p,q}=$ for $p<0$ or $q<0$.
It follows that
$H^0_{tot}( C)\cong H^{\semiinf+0}(\LZ,\eqW\* N)^{t\g[t]}$.
Comparing this with  \eqref{eq:iso22},
we obtain that
$$H^{\semiinf+0}(\LZ,\eqW\* N)^{t\g[t]}\cong  \bigoplus_{\lam\in P_+}V_{\lam}\* \Hom(\mf{z}_{\lam}, N).$$
\end{proof}
 
 \begin{Pro}\label{Pro:image-of-simple}
For $\lam\in P_+$,
we have
$$H^{\frac{\infty}{2}+0}(\LZ,\eqW\otimes \C_{[\lam]})\cong 
\mathbb{L}_{\lam}$$
as graded $\affg_{\kappa_c}$-modules.
\end{Pro}
\begin{proof}
Note that
$z\in \LZ$ acts as the constant $\chi_{\lam}(z)$ on  $H^{\frac{\infty}{2}+0}(\LZ,\eqW\otimes \C_{[\lam]})$.
Therefore,
$H^{\frac{\infty}{2}+0}(\LZ,\eqW\otimes \C_{[\lam]})$ is a direct sum of $\mathbb{L}_{\lam}$ 
as $\affg_{\kappa_c}$-modules
(\cite{FreGai04}).
 On the other hand,
 by Proposition \ref{Pro:gt-inv-part},
 $H^{\frac{\infty}{2}+0}(\LZ,\eqW\otimes \C_{[\lam]})^{t\g[t]}\cong V_{\lam}$,
 which is concentrated in degree $0$,
and hence the assertion.
 \end{proof}
 
 \begin{Pro}\label{Pro:ch-from-above}
For $N\in \LZ\Mod_{[\lam]}$,
\begin{align*}
\on{ch}H^{\frac{\infty}{2}+0}(\LZ,\eqW\otimes  N) 
\leq \on{ch}N\on{ch}\mathbb{L}_{\lam},
\end{align*}
that is,
 the dimension of each weight space of
 $\on{ch}H^{\frac{\infty}{2}+0}(\LZ,\eqW\otimes  N) $
 is equal to or smaller than that of $N\* \mathbb{L}_{\lam}$.
\end{Pro}
\begin{proof}
%We may assume that
%$N\in \LZ\Mod_{[\lam]}$ for some $\lam\in P_+$.
First, consider the case that
$N$ is a quotient of $\mf{z}_{\lam}$.
There exists a decreasing filtration
$N=N_0\supset N_1\supset N_2\supset \dots $,
$\bigcap_p N_p=0$,
of $\LZ$-modules such that
each successive quotient is a direct sum of $\C_{[\lam]}$.
By considering the induced filtration of the complex 
we get a spectral sequence 
$E_r\Rightarrow H^{\frac{\infty}{2}+\bullet}(\LZ,\eqW\otimes  N)$.
As the $E_1$-term is 
$\bigoplus_i H^{\frac{\infty}{2}+\bullet}(\LZ,\eqW\otimes  N_i/N_{i+1})
\cong N\* H^{\frac{\infty}{2}+\bullet}(\LZ,\eqW\otimes \C_{[\lam]})
$,
Lemma \ref{Lem:vnd}
gives that 
$E_1^{p,q}=0$ for $p+q< 0$.
As a result,
as graded vector spaces,
$H^{\frac{\infty}{2}+0}(\LZ,\eqW\otimes N)$ is 
a quotient of  $N\* H^{\frac{\infty}{2}+0}(\LZ,\eqW\otimes \C_{[\lam]})\cong
N\* \mathbb{L}_{\lam}$
by Proposition \ref{Pro:image-of-simple}.
Hence the assertion holds.

Next, let $N$ be an arbitrary  object in  $\LZ\Mod_{\on{reg}}$.
There exists a filtration
$0=N_0\subset N_1\subset N_2\subset \dots $,
$N=\bigcup_i N$,
such that each successive quotient is a quotient of $\mf{z}_{\lam}$ for some $\lam\in P_+$.
By Lemma \ref{Lem:vnd},
we get that
$\ch H^{\frac{\infty}{2}+\bullet}(\LZ,\eqW\otimes  N)\leq 
\sum_i \ch H^{\frac{\infty}{2}+\bullet}(\LZ,\eqW\otimes  N_i/N_{i+1})
\leq 
\ch N.\ch \mathbb{L}_{\lam}$.
\end{proof}

Define the vertex algebra $\V_{G,2}$ by
\begin{align}
\V_{G,2}:=H^{\semiinf+0}(\LZ,\eqW\* \eqW).
\label{eq:const-cdo}
\end{align}
%By definition.
%we have the vertex algebra homomorphism
%\begin{align*}
%V^{\kappa_c}(\g)\*_{\mf{z}(\affg)}V^{\kappa_c}(\g)\ra \V_{G,2},
%\quad u_1\* u_2\mapsto [\pi_R(u_1)\* \pi_R(u_2)\* \mathbf{1}].
%\end{align*}

\begin{Th}\label{Th:cdo}
We have
$$\V_{G,2}\cong \mc{D}_G^{ch}$$
as vertex algebras.
\end{Th}
\begin{proof}
The vertex algebra embedding $V^{\kappa_c}(\g)\hookrightarrow \eqW$
to the first factor (resp.\ the second factor) of $\eqW\* \eqW$  induces the vertex algebra homomorphism
\begin{align}
\pi_L:V^{\kappa_c}(\g)\rightarrow \V_{G,2}\quad \text{(resp.\ }\pi_R:V^{\kappa_c}(\g)\rightarrow \V_{G,2}\text{).}
\end{align}

Consider the decomposition
$\V_{G,2}=\bigoplus_{\lam\in P_+}(\V_{G,2})_{[\lam]}$
with respect to the action of $\affg_{\kappa_c}$ on the first factor of $\eqW\* \eqW$.
Clearly,
$(\V_{G,2})_{[\lam]}=H^{\semiinf+0}(\LZ,\eqW_{[\lam]}\* \eqW_{[\lam^*]})$.
By Proposition \ref{Pro:ch-from-above},
we have
\begin{align}
\ch (\V_{G,2})_{[\lam]}\leq \ch \eqW_{[\lam^*]} \ch \mathbb{L}_{\lam}\leq \ch (\Dch_C)_{[\lam]}
\label{eq:gcheneq}
\end{align}
for $\lam\in P_+$.
Therefore,
it is sufficient  to show that there is a non-trivial vertex algebra homomorphism
$\Dch_G\ra \V_{G,2}$
since $\Dch_G$ is simple.
Note that 
$\Dch_G$ is $\Z_{\geq 0}$-graded,
and
thus,
so is $\V_{G,2}$ by \eqref{eq:gcheneq}.

By using Proposition \ref{Pro:gt-inv-part} twice,
we obtain that
\begin{align}
(\V_{G,2})^{t\g[t]\+ t\g[t]}
\cong H^{\semiinf+0}(\LZ,\eqW^{t\g[t]}\* \eqW^{t\g[t]})
\cong \bigoplus_{\lam\in P_+}\on{End}_{\LZ}(\mf{z}_{\lam})\otimes V_{\lam}\*V_{\lam^*}\nonumber \\
\cong \bigoplus_{\lam\in P_+}\mf{z}_{\lam}\otimes V_{\lam}\*V_{\lam^*}.
\label{eq:13:38*}
\end{align}
In particular,
\begin{align}
(\V_{G,2})_0^{t\g[t]\times t\g[t]}\cong \bigoplus_{\lam}V_{\lam}\*V_{\lam^*}.
\label{eq:atleast}
\end{align}
It follows from \eqref{eq:gcheneq} that 
\begin{align}
(\V_{G,2})_0=(\V_{G,2})_0^{t\g[t]\times t\g[t]}\cong  \mc{O}(G)
\label{eq:isoV2}
\end{align}
as $\g\times \g$-modules.

Since $\V_{G,2}$ is $\Z_{\geq 0}$-graded,
$(\V_{G,2})_0$ generates a commutative vertex subalgebra 
of $\V_{G,2}$.
In particular, $(\V_{G,2})_0$ is a commutative associative algebra
by the multiplication $a.b=a_{(-1)}b$.

We may assume that
$P_{1,0}$ acts on 
$(\V_{G,2})_0$
as the quadratic Casimir element of $U(\g)$,
so that
it acts on 
$V_{\lam}\*V_{\lam^*}\subset (\V_{G,2})_0$
by the multiplication by $(\lam+2\rho|\rho)$.
This gives a $\Q_{\geq 0}$-grading
\begin{align*}
(\V_{G,2})_0=\bigoplus_{d\in \Q_{\geq 0}}(\V_{G,2})_0(d),
\quad (\V_{G,2})_0(d)=\{a\in (\V_{G,2})_0\mid P_{1,0}a=d a\}
\end{align*}
of the commutative algebra $(\V_{G,2})_0$
such that $(\V_{G,2})_0(0)=V_0\* V_0=\C$.
It follows that
the projection
$(\V_{G,2})_0\ra (\V_{G,2})_0(0)=\C$
is an algebra homomorphism.
\begin{Lem}\label{Lem;:function-on-G}
Let $A$ be a unital commutative $G\times G$-algebra,
that is,
unital commutative $\C$-algebra
equipped with an  action of $G\times G$ on $A$
such that the multiplication map $A\* A\ra A$ is a $G\times G$-module homomorphism,
where $G\times G$ diagonally  acts on $A\* A$.
Suppose that 
$A\cong \bigoplus_{\lam\in P_+}V_{\lam}\* V_{\lam^*}$ as  $G\times G$-modules,
and the projection $A\ra V_0\* V_0=\C$  is an algebra homomorphism.
Then $A\cong \mc{O}(G)$ 
as commutative $G\times G$-algebras.
\end{Lem}
\begin{proof}
For a $G$-module $M$,
let $\mu_M:M\ra \mc{O}(G)\* M$
be the comodule map.
Thus,
$\mu_M(m)=\sum_i f_i\* m_i$ if 
$g.m=\sum_i f_i(g)m_i$ for all $g\in G$.
Then $\tilde{\mu}_M: \mc{O}(G)\* M\ra 
\mc{O}(G)\* M$,
$f\* m\mapsto f\mu(m)$,
gives a linear isomorphism
such that $\tilde{\mu}_M\circ (g\* g)=(g\* 1)\circ \tilde{\mu}_M$ for all $g\in G$.
Set $\tilde{\nu}_M=(\tilde{\mu}_M)^{-1}$,
so that $\tilde{\nu}_M\circ (g\* 1)\circ =(g\*g)\circ \tilde{\nu}_M$.
Define
$\nu_M:M\ra \mc{O}(G)\* M$ by $\nu_M(m)=\tilde{\nu}_M(1\* m)$.
More concretely,
we have
 $\nu_M (m)=\sum_i f_i'\* m_i$ if $g^{-1}.m=\sum_i f'_i(g)m_i$ for all $g\in G$.
We have the linear isomorphism
\begin{align*}
M\isomap (\mc{O}(G)\* M)^{\Delta(G)},\quad m\mapsto \nu_M(m).
\end{align*}
We claim that if $M$ is a commutative $G$-algebra $R$,
this gives an isomorphism of algebras.
Indeed,
consider the algebra homomorphism $\epsilon\* 1:\mc{O}(G)\* R\ra \C\* R=R$,
where $\epsilon$ is a counit.
We have $(\epsilon\* 1)\nu_R(r)=r$.
Thus, %since $\epsilon\* 1$ is an algebra homomorphism,
$(\epsilon\* 1)(\nu_R(r)\nu_R(r')-\nu_R(rr'))
=(\epsilon\* 1)(\nu_R(r))(\epsilon\* 1)(\nu_R(r'))-(\epsilon\* 1)(\nu_R(rr'))=r r'-rr'=0$. 
Because $\nu_R(r)\nu_R(r')-\nu_R(r r')\in (\mc{O}(G)\* R)^{\Delta(G)}$
and the restriction of $(\epsilon\* 1)$ to $ (\mc{O}(G)\* R)^{\Delta(G)}$ is an injection,
it follows that  that $\nu_R(r)\nu_R(r')=\nu_R(r r')$.

Let $\Phi:A\isomap \mc{O}(G)$ be an isomorphism as $G\times G$-modules.
For a $G$-module $M$,
we have 
the linear isomorphism
\begin{align*}
M\isomap (A\* M)^{\Delta(G)},\quad m\mapsto (\Phi\* 1)(\nu_M^A(m)).
\end{align*}
If $M$ is a commutative $G$-algebra,
this is an isomorphism of algebras by the same argument as above.

We conclude that there are isomorphisms
\begin{align*}
A\cong (\mc{O}(G)\* A)^{\Delta(G)}\cong (A\* \mc{O}(G))^{\Delta(G)}\cong \mc{O}(G).
\end{align*}
This completes the proof.
\end{proof}
By Lemma \ref{Lem;:function-on-G},
we can identify
$(\V_{G,2})_0$ with $\mc{O}(G)$ as commutative $G\times G$-algebras.
Since
$x(z)f(w)\sim\frac{1}{z-w}(xf)(w)$
for $x\in \g$ and $f\in (\V_{G,2})_0$,
we have a non-trivial vertex algebra homomorphism
$\Dch_G\ra\V_{G,2}$ as required.
\end{proof}
\begin{Th}\label{Th:it-is-an-inverse-to-DS}
Let $V$ be a vertex algebra object in $\KL$.
\begin{enumerate}
\item We have $V\cong H^{\frac{\infty}{2}+0}(\LZ, \eqW\otimes H_{DS}^0(V))$ as vertex algebras.
\item 
Let $\on{KL}(V)$ be the category of $V$-module objects in $\KL$.
Then 
 $$M\mapsto H^{\frac{\infty}{2}+0}(\LZ, \eqW\otimes H_{DS}^0(M))$$ 
 gives  an identity functor in $\on{KL}(V)$. 
 \end{enumerate}
\end{Th}
\begin{proof}
(i)
%(i) We claim that there is the following commutative diagram.
%\begin{align}
%\begin{CD}
%\eqW\* \eqW\* V    @>1\* H^{\semiinf+0}(\affg,\g, ?)>>  \eqW\* H_{DS}^0(V)\\
%@VV  H^{\semiinf+0}(\LZ,?) \* 1V        @VV H^{\semiinf+0}(\LZ,?)V\\
%\Dch_G\*V     @>H^{\semiinf+0}(\affg,\g, ?)>>  V.
%\end{CD}
%\end{align}
%In order to show this,
%define the bigraded vertex algebra
Set
\begin{align*}
C=\bigoplus_{i\in \Z}C^i,\quad
C^i=\bigoplus_{p+q=i}C^{p,q},\quad C^{p,q}=\eqW\otimes \eqW \otimes V\*  \bigwedge\nolimits^{\semiinf+p }({\mf{z}(\affg)})\otimes  \bigwedge\nolimits^{\semiinf+q }(\affg).
\end{align*}
Let
$d_{\mf{z}}$
denote the 
 differential of the complex 
$\eqW\* \eqW\* \bigwedge\nolimits^{\semiinf+\bullet }({\mf{z}(\affg)})$
which defines $H^{\semiinf+\bullet}(\LZ,\eqW\* \eqW)$.
We regard $d_{\mf{z}}$ as a differential 
of $C$ which   acts trivially on the factors $V$ and $\bigwedge\nolimits^{\semiinf+\bullet }(\affg)$.
Let $d_{\affg}$ be the differential of the  complex  
$\eqW\* V\* \bigwedge\nolimits^{\semiinf+\bullet}(\affg)$
which defines $H^{\semiinf+\bullet}(\affg_{-\kappa_\g}, \eqW\* V)$.
We consider $d_{\affg}$ 
as a differential of $C$ which  trivially acts on the first factor $\eqW$ and 
$ \bigwedge\nolimits^{\semiinf+i }({\mf{z}(\affg)})$.
We have $\{d_{\mf{z}},d_{\affg}\}=0$,
and thus
 $C$ is equipped with the structure of a double complex.

 Let $H_{tot}^\bullet(C)$ be the total cohomology of the double complex $C$.
 Note that $C$ is a direct sum of finite-dimensional subcomplexes.
 In fact, 
 we have
 $C=\bigoplus_{\lam,\mu,\nu\in P_+}C_{\lam,\mu,\nu}$ as complexes,
 where 
 $C_{\lam,\mu,\nu}=\eqW_{[\lam]}\* \eqW_{[\mu]}\* V_{[\nu]}\* \bigwedge\nolimits^{\semiinf+\bullet }({\mf{z}(\affg)})\otimes  \bigwedge\nolimits^{\semiinf+\bullet }(\affg)
$,
and each 
$C_{\lam,\mu,\nu}$ further decomposes into a direct sum of finite-dimensional 
subcomplexes consisting of the
eigenspaces of the 
total action of the Hamiltonian.
%which are preserved by $d_{tot}=d_{\mf{z}}+d_{\affg}$.

Let $E_r\Rightarrow H^\bullet_{tot}( C)$ 
 be the spectral sequence 
 such that $d_0=d_\mf{z}$ and $d_1=d_{\affg}$.
 By definition,
%We have
\begin{align*}
&E_1^{p,q}\cong H^{\semiinf+q}(\LZ,\eqW\*\eqW)\* V\*  \bigwedge\nolimits^{\semiinf+p }(\affg),\\
%\cong \delta_{q,0}\Dch_G\*V\*  \bigwedge\nolimits^{\semiinf+q }(\affg),\\
&E_2^{p,q}\cong  H^{\semiinf+p}(\affg_{-\kappa_\g}, H^{\semiinf+q}(\LZ,\eqW\*\eqW)\*V).
%\cong %\cong \delta_{q,0}V\* H^p(\g,\C)
\end{align*}
By Theorem \ref{Th:cdo},
$ H^{\semiinf+q}(\LZ,\eqW\*\eqW)$ is 
a  module over 
$H^{\semiinf+}(\LZ,\eqW\*\eqW)\cong \Dch_G$.
Hence 
$ H^{\semiinf+q}(\LZ,\eqW\*\eqW)$ is
free over $U(\g[t^{-1}]t^{-1})$ by 
Corollary \ref{Co:free-cofree-modules},
and hence 
$E_2^{p,q}=0$ if $p<0$ or $q<0$.
It follows that
\begin{align}
H_{tot}^0(C)\cong
E_2^{0,0}=  H^{\semiinf+0}(\affg_{-\kappa_\g}, H^{\semiinf+0}(\LZ,\eqW\*\eqW)\*V)
\cong H^{\semiinf+0}(\affg_{-\kappa_\g}, \Dch_G\* V)\cong V\label{eq:iso1}
\end{align}
as vertex algebras.

Next, consider the spectral sequence
 $E_r'\Rightarrow H^\bullet_{tot}( C)$ 
 such that $d_0=d_{\affg}$ and $d_0=d_\mf{z}$.
 By Proposition \ref{Pro:relative-coh}
 and Theorem \ref{Th:DS-realization},
 we have
 \begin{align*}
&(E_1')^{p,q}\cong  \eqW\* H_{DS}^0(V)\* \bigwedge\nolimits^{\semiinf+p }({\mf{z}(\affg)})\* H^q(\g,\C),\\
&(E_2')^{p,q} \cong H^{\semiinf+p}(\LZ, \eqW\* H_{DS}^0(V))\* H^q(\g,\C).
\end{align*}
Since $(E_2')^{p,q}=0$ for $p<0$ or $q<0$,
we obtain the  vertex algebra isomorphism
\begin{align}
H_{tot}^0(C)\cong  H^{\semiinf+0}(\LZ, \eqW\* H_{DS}^0(V)).\label{eq:iso2}
\end{align}
The assertion now immediately follows from  \eqref{eq:iso1} and \eqref{eq:iso2}.
(ii)
 The same proof as (1) applies.

\end{proof}

%Consider the functor
% \begin{align}
%\LZ\on{-Mod}_{reg}\ra \LZ\Mod_{reg},\quad
%N\mapsto H^{\frac{\infty}{2}+0}(\LZ,\mathbf{I}\otimes  N).
%\label{eq:the-cen-functor}
%\end{align}
%\begin{Lem}\label{Lem:leex}
%$H^{\frac{\infty}{2}+i}(\LZ,\mathbf{I}\otimes  N)=0$ for $i<0$.
%Hence the functor \eqref{eq:the-cen-functor} is left exact.
%\end{Lem}
%\begin{proof}
%By Proposition \ref{Pro:free},
%there exists an increasing filtration
%$0=N_0\subset N_1\subset \dots,$,
%$\mathbf{I}=\bigcup_p N_p$,
%such that 
%each successive quotient $N_p/N_{p-1}$ is isomorphic to 
%$H_{DS}^0(\mathbb{V}_{\lam})=\mf{z}_\lam$ for some $\lam\in P_+$.
%Therefore,
%$\mathbf{I}$ is free over $\mf{z}(\affg)$ in the usual associative sense and the assertion follows from 
%Lemma \ref{Lem:vanihing-negative}.
%\end{proof}
%
%\begin{Pro}
%We have 
%$H_{DS}^0(H^{\semiinf+0}(\LZ,\eqW\* N))\cong H^{\semiinf+0}(\LZ,\mathbf{I}\* N)$
%for $N\in \LZ\Mod_{reg}$.
%\end{Pro}
%\begin{proof}
%Using Lemma \ref{Lem:leex},
%as in the proof of Theorem \ref{Th:it-is-an-inverse-to-DS}
%we find that there is an isomorphism
%\begin{align*}
%H^{\semiinf+0}(\LZ,\mathbf{I}\* N)\cong 
%H^{\semiinf+0}(\LZ,H^{\semiinf+0}(\affg_{-2\kappa_{\g}},\eqW\* \eqW)\* N)\\\cong
%H^{\semiinf+0}(\affg_{-\kappa_\g},\eqW\* H^0(\LZ,\eqW\* N))
%\cong H_{DS}^0(H^{\semiinf+0}(\LZ,\eqW\* N)).
%\end{align*}
%\end{proof}
%
Let $\on{KL}_0$  be the 
full subcategory of $\LZ\Mod_{reg}$
consisting of objects $M$ that is isomorphic to $H_{DS}^0(\tilde{M})$ for some $\tilde{M}\in \on{KL}$.
\begin{Pro}\label{Lem:can-go-back}
For $M\in \KL_0$,
we have
$$M\cong  H_{DS}^0(H^{\semiinf+0}(\LZ,\eqW\*M)).$$
%
%\begin{enumerate}
%\item For $M\in \KL_0$,
%$$M\cong H^{\semiinf+0}(\LZ,\mathbf{I}\* M)\cong  H_{DS}^0(H^{\semiinf+0}(\LZ,\eqW\*M)).$$
%\end{enumerate}
%
%$M\mapsto H_{DS}^0(H^{\semiinf+0}(\LZ,\eqW\*M))$  
%gives an 
%identity functor in $\on{KL}_0$.
\end{Pro}
\begin{proof}
Let $\tilde{M}\in \on{KL}$ 
and set
$M\cong H_{DS}^0(\tilde{M})\in \KL_0$.
By Theorem \ref{Th:it-is-an-inverse-to-DS},
$\tilde{M}\cong H^{\semiinf+0}(\LZ,\eqW\*M)$.
Hence
$M\cong H_{DS}^0(H^{\semiinf+0}(\LZ,\eqW\*M))
$.
%
%Let $\psi:M\ra N$ be a morphism in $\KL_0$,
%and let
%$\tilde\psi:H^{\semiinf+0}(\LZ,\eqW\*M)\ra H^{\semiinf+0}(\LZ,\eqW\*M)$
%induced morphism.
%
%By the proof of Theorem \ref{Th:it-is-an-inverse-to-DS},
\end{proof}

\begin{Pro}\label{Pro:KL0}
Let 
\begin{align}
0\ra M_1\ra M_2\ra M_3\ra 0
\label{eq:ess}
\end{align}
be an exact sequence in $\LZ\Mod_{reg}$.
\begin{enumerate}
\item Suppose that $M_1, M_2\in \on{KL}_0$. Then $M_3\in \on{KL}_0$ and
\eqref{eq:ess} induces the exact 
 sequence
$$0\ra H^{\semiinf+0}(\LZ,\eqW\* M_1)\ra H^{\semiinf+0}(\LZ,\eqW\*M_2)\ra H^{\semiinf+0}(\LZ,\eqW\*M_3)\ra 0.$$
\item Suppose that $M_2, M_3\in \on{KL}_0$. Then $M_1\in \on{KL}_0$ and
\eqref{eq:ess} induces the exact 
 sequence
$$0\ra H^{\semiinf+0}(\LZ,\eqW\*M_1)\ra H^{\semiinf+0}(\LZ,\eqW\*M_2)\ra H^{\semiinf+0}(\LZ,\eqW\*M_3)\ra 0.$$
\end{enumerate}
\end{Pro}
\begin{proof}
By Lemma \ref{Lem:vnd},
\eqref{eq:ess}
induces  the exact sequence 
$0\ra H^{\semiinf+0}(\LZ,\eqW\*M_1)\ra H^{\semiinf+0}(\LZ,\eqW\*M_2)\overset{\psi}{\ra} H^{\semiinf+0}(\LZ,\eqW\*M_3)$.

(i) 
Let $N=\im \psi$. Applying the exact functor $H_{DS}^0(?)$ to the exact sequence
$0\ra H^{\semiinf+0}(\LZ,\eqW\*M_1) \ra H^{\semiinf+0}(\LZ,\eqW\*M_1) \ra N\ra 0$,
we obtain the exact sequence
$0\ra M_1\ra M_2\ra H_{DS}^0(N)\ra 0$ by Lemma \ref{Lem:can-go-back}.
Thus,
$H_{DS}^0(N)\cong M_2/M_1\cong M_3$,
and so $M_3\in \on{KL}_0$.
Moreover,
$N\cong H^{\semiinf+0}(\LZ,\eqW\* M_3)$ by Theorem \ref{Th:it-is-an-inverse-to-DS},
whence the second assertion.

(ii)
Let $N$ denote the cokernel of $\psi$. 
Applying  $H_{DS}^0(?)$ to the exact sequence $H^{\semiinf+0}(\LZ,\eqW\*M_2)\ra H^{\semiinf+0}(\LZ,\eqW\*M_3)\ra N\ra 0$,
we obtain the exact sequence
$M_2\ra M_3\ra H_{DS}^0(N)\ra 0$.
Since the first map is surjective, $H_{DS}^0(N)=0$,
and thus $N=0$ 
by Proposition \ref{Pro:ch-DS}.
The exact sequence
$0\ra H^{\semiinf+0}(\LZ,\eqW\*M_1)\ra H^{\semiinf+0}(\LZ,\eqW\*M_2)\overset{\psi}{\ra} H^{\semiinf+0}(\LZ,\eqW\*M_3)\ra 0$
induces
 an exact sequence
$0\ra H_{DS}^0(H(\LZ,\eqW\* M_1))\ra M_2\ra M_3\ra 0$,
and thus,
$N_1\cong H_{DS}^0(H(\LZ,\eqW\* M_1))$.
\end{proof}

\begin{Pro}\label{Pro:filtration}
Let $M\in \on{KL}$.
Suppose that $N=H_{DS}^0(M)$
admits a filtration
$0=N_0\subset N_1\subset N_2\subset \dots $,
$N=\bigcup_p N_p$ such that
each successive quotient 
is an object of $\KL_0$.
%$N_i/N_{i-1}$ is isomorphic to $\mf{z}_{\lam}$.
Then
$M_i=H^{\semiinf+0}(\LZ,\eqW\* N_i)$ defines an increasing filtration
of $M$ such that 
each successive quotient 
$M_i/M_{i-1}$ is isomorphic to $H^{\semiinf+0}(\LZ,\eqW\* N_i/N_{i-1})$.
%In particular,
%$M\in \on{KL}^{\Delta}$.
\end{Pro}
\begin{proof}
First, we show by induction on $i\geq 0$
that $N_i\in \KL_0$.
There is nothing to show for $i=0$.
So let $i>0$.
By the induction hypothesis and 
Proposition~\ref{Pro:KL0},
$N/N_{i}\in \KL_0$.
Hence $N/N_{i+1}\cong (N/N_i)/(N_{i+1}/N_i)\in \KL_0$.
Thus again by Proposition \ref{Pro:KL0},
$N_{i+1}\in \KL_0$.
Since $N_i\in \KL_0$ for all $i\geq 0$,
we have $M_{i}/M_{i-1}\cong H^{\semiinf+0}(\LZ,
\eqW\* (N_i/N_{i-1})$. This completes the proof.
\end{proof}

\section{Genus zero chiral algebras of class $\mc{S}$}
\label{section:chiral-algebras-of-classs-S}
%In this section we construct the genus zero chiral algebras of class $\mc{S}$.

For $b\geq 2$,
define the vertex algebra
\begin{align}
C_b=C(\bigoplus_{i=1}^{b-1}\LZ^{i,i+1},\eqW^{\otimes b}):=\eqW^{\otimes b}\otimes \left(\bw{\semiinf+\bullet}(\mf{z}(\affg))\right)^{\otimes b-1}
\label{eq:complex-Cb}
\end{align}
and the field 
\begin{align}
Q^{(b)}(z)=\sum_{i=1}^{b-1}Q^{(i,i+1)}(z)
\label{eq:diff-qr}
\end{align}
on $C_b$.
Here, 
\begin{align*}
Q^{(i,i+1)}(z)=\sum_{j=1}^{\on{rk}\g}(\rho_i(P_j(z))-\rho_{i+1}(\tau(P_j)(z)))\rho_{gh,i}(\psi_j^*(z)),
\end{align*}
where
$\rho_i$ denotes the action on the 
$i$-th factor of 
$\eqW^{\* r}$ and 
$\rho_{gh,i}$ denotes the action on the $i$-th factor of $\left(\bw{\semiinf+\bullet}(\mf{z}(\affg))\right)^{\otimes b-1}$.
Clearly
$Q^{(b)}(z)Q^{(b)}(w)\sim 0$,
 $(C_b,Q^{(b)}_{(0)})$ is the differential graded vertex algebra.
The corresponding cohomology is denoted by $H^{\frac{\infty}{2}+\bullet}(\bigoplus_{i=1}^{b-1}\LZ^{i,i+1},\eqW^{\otimes b})$.
The space 
$H^{\frac{\infty}{2}+\bullet}(\bigoplus_{i=1}^{b-1}\LZ^{i,i+1},\eqW^{\otimes b})$
is naturally a  vertex algebra,
which 
is $\frac{1}{2}\Z$-graded by the Hamiltonian
\begin{align}
H=\sum_{i=1}^{b}\rho_i(L^{\eqW}_0)+\sum_{i=1}^{b-1}\rho_{gh,i}(L^{gh}_0).
\label{eq:Hamiltonian-VGS}
\end{align}

Let
\begin{align}
%\V_{G,0}:=H_{DS}^0(\eqW_G),\quad 
\V_{G,1}=\eqW.
\label{eq;initial}
\end{align}
and define
\begin{align}
\V_{G,b}:=H^{\frac{\infty}{2}+0}(\bigoplus_{i=1}^{b-1}\LZ^{i,i+1},\eqW^{\otimes b})
\label{eq:main-def}
\end{align}
for $b\geq 2$.
%The vertex algebras
%$\V_{G,b}$, $b\geq 1$,
%are called the genus zerp {\em chiral algebras of class $\mc{S}$.
%where $\LZ^{i,j}$ is a copy of $\LZ$ that acts on $\eqW_G^{\otimes r}$
%by
%\begin{align}
%p(z)\mapsto \overbrace{1\otimes \dots \otimes  1}^{i-1}\otimes p(z)\otimes  1\otimes \dots \otimes  1 -\overbrace{1\otimes \dots \otimes  1}^{i}\otimes \tau(p)(z)\otimes  1\otimes \dots \otimes  1.
%\end{align}
For $b=2$,
this agree with the definition \eqref{eq:const-cdo}
and by Theorem \ref{Th:cdo},
we have
\begin{align}
\V_{G,2}\cong \mc{D}_{G}^{ch}.
\label{eq:cylinder}
\end{align}

The embedding
$V^{\kappa_c}(\g)\hookrightarrow \eqW_G$ induces
 the vertex algebra homomorphism
 \begin{align}
V^{\kappa_c}(\g)^{\* b}\ra \V_{G,b}\label{eq:homVVr}
\end{align}
By definition, 
\begin{align*}
\rho_i(z)\equiv \rho_{i+1}(\tau(z))
%\overbrace{1\otimes \dots \otimes  1}^{i-1}\otimes p(z)\otimes  1\otimes \dots \otimes  1 \equiv \overbrace{1\otimes \dots \otimes  1}^{i}\otimes \tau(p)(z)\otimes  1\otimes \dots \otimes  1
\end{align*}
for $z\in \LZ$, $1\leq i< b$
on $H^{\frac{\infty}{2}+0}(\bigoplus_{i=1}^{b-1}\LZ^{i,i+1},\eqW^{\otimes b})$.
Therefore,
\eqref{eq:homVVr}
factors through 
 the vertex algebra homomorphism
\begin{align*}
\Vb\ra \V_{G,b}
\end{align*}
(see \eqref{eq;b-copy}).
In particular,
$\V_{G,b}$ belongs to $\KL_b$.

\begin{Lem}\label{Lem:inductive-construction-of-Vr}
For each
$b\geq 2$,
we have the following.
\begin{enumerate}
\item $H^{\frac{\infty}{2}+n}(\bigoplus_{i=1}^{b-1}\LZ^{i,i+1},\eqW^{\otimes b})=0$ for $n< 0$.
\item  $\V_{G,b}\cong H^{\semiinf+0}(\LZ,\eqW\* \V_{G,b-1})\cong H^{\semiinf+0}(\LZ,\V_{G,b-1}\*\eqW)$.
\end{enumerate}
\end{Lem}
\begin{proof}
%Note that  (iii) follows from (ii) and 
%Theorem \ref{Th:it-is-an-inverse-to-DS},
We prove (i)  and (ii) by induction on $b\geq 2$.
For $b=2$,
(i) 
 has been proved in  Lemma \ref{Lem:vnd}
 and there is nothing to show for (ii).
Let $b>2$.
Consider the spectral sequence
$E_r\Rightarrow H^{\frac{\infty}{2}+\bullet}(\bigoplus_{i=1}^{b-1}\LZ^{i,i+1},\eqW^{\otimes b})$
such that
\begin{align*}
&E_1^{p,q}=\eqW\* H^{\frac{\infty}{2}+q}(\bigoplus_{i=1}^{b-2}\LZ^{i,i+1},\eqW^{\otimes b-1})\*
 \bigwedge\nolimits^{\semiinf+p }({\mf{z}(\affg)}),\\
 &E_2^{p,q}=H^{\frac{\infty}{2}+p}(\LZ,\eqW\* H^{\frac{\infty}{2}+q}(\bigoplus_{i=1}^{b-2}\LZ^{i,i+1},\eqW^{\otimes b-1})).
\end{align*}
 By  Lemma \ref{Lem:vnd}
and 
 the
 induction hypothesis, 
 $E_2^{p,q}=0$ if $q<0$ or $p<0$.
 Therefore, we get that
 $H^{\frac{\infty}{2}+n}(\bigoplus_{i=1}^{b-1}\LZ^{i,i+1},\eqW^{\otimes b})=0$
 for $n<0$ and
 $$\V_{G,b}
 \cong E_2^{0,0}\cong  H^{\semiinf+0}(\LZ,\eqW\* \V_{G,b-1}).$$
\end{proof}

\begin{Pro}\label{Pro:from-DG}
For any $b, b'\geq 1$,
we have
$H^{\frac{\infty}{2}+i}(\LZ,\V_{G,b}\* \V_{G,b'})=
0$ for $i< 0$ and 
$H^{\frac{\infty}{2}+0}(\LZ,\V_{G,b}\* \V_{G,b'}) \cong \V_{G,b+b'}$
as vertex algebras.
\end{Pro} 

\begin{proof}
The first statement follows from Lemma \ref{Lem:vanihing-negative}
and
Proposition \ref{Pro:freeness}.
We prove the second statement by induction on $b\geq 1$.
For $b=1$,
the statement has been proved in 
Lemma \ref{Lem:inductive-construction-of-Vr}.
So let $b\geq 2$,
and consider
the cohomology
$H^{\semiinf+\bullet}(\LZ^{\+2},
\V_{G,b-1}\* \eqW\* \V_{G,b'})$
of the complex 
$$(\V_{G,b-1}\* \eqW\* \V_{G,b'}\* \bw{\semiinf+\bullet}(\mf{z}(\affg))\* \bw{\semiinf+\bullet}(\mf{z}(\affg)),
Q_{(0)}),$$
with
$Q_{(0)}=Q^{(1)}_{(0)}+Q^{(2)}_{(0)}$,
where $Q^{(1)}_{(0)}$ is the differential of the complex for the cohomology
$H^{\semiinf+\bullet}(\LZ, \V_{G,b-1}\* \eqW)$ that   acts trivially on the factor 
$\V_{G,b'}$ and the second factor of $\bw{\semiinf+\bullet}(\mf{z}(\affg))^{\otimes 2}$,
and 
$Q^{(2)}_{(0)}$ is the differential of the complex for the cohomology
$H^{\semiinf+\bullet}(\LZ,  \eqW\* \V_{G,b'})$ that  acts  trivially on the factor 
$\V_{G,b-1}$ and the first factor of $\bw{\semiinf+\bullet}(\mf{z}(\affg))^{\otimes 2}$.
There is a spectral sequence
$$E_r\Rightarrow H^{\semiinf+\bullet}(\LZ^{\+2},
\V_{G,b-1}\* \eqW\* \V_{G,b'})$$
such that
\begin{align*}
E_2^{p,q}=H^
{\semiinf+p}(\LZ,H^{\semiinf+q}(\LZ,\V_{G,b-1}\* \eqW)\* \V_{G,b'}).
\end{align*}
By Lemma \ref{Lem:inductive-construction-of-Vr}
and Proposition \ref{Pro:freeness},
$E_2^{p,q}=0$ for $p<0$ or $q<0$.
Therefore,
\begin{align*}
H^{\semiinf+0}(\LZ^{\+2},
\V_{G,b-1}\* \eqW\* \V_{G,b'})\cong 
E_2^{0,0}\cong 
H^{\frac{\infty}{2}+0}(\LZ,\V_{G,b}\* \V_{G,b'}).
\end{align*}
%by Lemma \ref{Lem:inductive-construction-of-Vr}.
There is another spectral sequence
$$E_r'\Rightarrow H^{\semiinf+\bullet}(\LZ^{\+2},
\V_{G,b-1}\* \eqW\* \V_{G,b'})$$
such that
\begin{align*}
(E_2')^{p,q}=H^
{\semiinf+p}(\LZ,\V_{G,b-1}\* H^{\semiinf+q}(\LZ,\eqW\* \V_{G,b'})).
\end{align*}
Again
by Lemma \ref{Lem:inductive-construction-of-Vr}
and Proposition \ref{Pro:freeness},
$(E_2')^{p,q}=0$ for $p<0$ or $q<0$.
It follows that
\begin{align*}
H^{\semiinf+0}(\LZ^{\+2},
\V_{G,b-1}\* \eqW\* \V_{G,b'})\cong 
(E_2')^{0,0}\cong 
H^{\frac{\infty}{2}+0}(\LZ,\V_{G,b'-1}\* \V_{G,b+1})\\
\cong \V_{G,b+b'}
\end{align*}
by the induction hypothesis.
This completes the proof.
\end{proof}
\begin{Pro}\label{Pro:reduction-of-Vr-is-Vr-1}
We have
$H_{DS}^0(\V_{G,b+1})\cong \V_{G,b}$
for $b\geq 1$.
\end{Pro}
\begin{proof}
We have nothing to prove for $b=1$.
Let $b>1$.
We have
$H_{DS}^0(\V_{G,b+1})\cong H^{\semiinf+0}(\affg_{-\kappa_\g},\eqW\* \V_{G,b+1})
\cong  H^{\semiinf+0}(\affg_{-\kappa_\g},\eqW\* H^{\semiinf+0}(\LZ,\V_{G,b}\* \eqW))$.
As in the proof of Theorem \ref{Th:it-is-an-inverse-to-DS},
we find that 
\begin{align*}
&H^{\semiinf+0}(\affg_{-\kappa_\g},\eqW\* H^{\semiinf+0}(\LZ,\V_{G,b}\* \eqW))\\
&\cong H^{\semiinf+0}(\LZ,H^{\semiinf+0}(\affg_{-\kappa_\g},\eqW\*\V_{G,b})\* \eqW).
%\cong H^{\semiinf+0}(\LZ,\V_{G,r-1}\* \eqW),
\end{align*}
But the latter is isomorphic to $H^{\semiinf+0}(\LZ,\V_{G,b-1}\* \eqW)\cong \V_{G,b}$ by the induction hypothesis.
\end{proof}

\begin{Th}\label{Th:simple}
The  vertex algebra $\V_{G,b}$ is simple for all $b\geq 1$.
\end{Th}
\begin{proof}
For $b=1,2$, we already know that $\V_{G,b}$ is simple.
So let $b>2$ and let
 $I$ be a nonzero submodule of $\V_{G,b}$.
The embedding $I\hookrightarrow \V_{G,b}$ induces the embedding
$H_{DS}^0(I)\hookrightarrow H_{DS}^0(\V_{G,b})=\V_{G,b-1}$.
%by Theorem \ref{Th:compatibility-with-DS}.
Therefore, $H_{DS}^0(I)$ is a submodule of $\V_{G,b-1}$,
which is 
 is nonzero
since $I\cong H^{\frac{\infty}{2}+0}(\LZ,\eqW\otimes H_{DS}^0(I))$
by Theorem \ref{Th:it-is-an-inverse-to-DS}.
Because $\V_{G,b-1}$ is simple by induction hypothesis,
we get  that
$ H_{DS}^0(I)=\V_{G,b-1}$.
But then $I=H^{\frac{\infty}{2}+0}(\LZ, \eqW\otimes H_{DS}^0(I)))=H^{\frac{\infty}{2}+0}(\LZ, \eqW\otimes \V_{G,b-1})=\V_{G,b}$.
This completes the proof.
\end{proof}

Consider the decomposition
$\V_{G,b}=\bigoplus_{\lam\in P_+}(\V_{G,b})_{[\lam]}$
with respect to the action of $\iota_1$ (see \eqref{eq:iota}).
Let $T$ be a maximal torus of $G$ as in Introduction.

\begin{Pro}\label{Pro:conic}
For $b\geq 3$,
the vertex algebra $\V_{G,b}$ is conical.
We have for
$\lam\in P_+$,
$(z_1,\dots,z_b)\in T^b$
that
\begin{align*}
\on{tr}_{(\V_{G,b})_{[\lam]}}(q^{L_0}z_1z_2\dots z_b)=\left(\frac{q^{\bra \lam,\rho^{\vee}\ket}
\prod\limits_{j=1}^{\infty}(1-q^j)^{\on{rk}\g}}{\prod\limits_{\alpha\in \Delta_+}(1-q^{\bra \lam+\rho,\alpha^{\vee}\ket})}\right)^{b-2}
\prod_{k=i}^b\on{tr}_{\mathbb{V}_{\lam}} 
(q^{-D}z_k),
\end{align*}
where $D$ is the standard degree operator of the affine Kac-Moody algebra,
and
so
\begin{align*}
\on{tr}_{\mathbb{V}_{\lam}}(q^{-D} z)=
\frac{\sum_{w\in W}\epsilon(w)e^{w\circ \lam}(z)}{
\prod_{j=1}^{\infty}(1-q^j)^{\on{rk}\g}
\prod_{\alpha\in \Delta_+}\prod_{j=1}^{\infty}(1-q^{j-1}e^{-\alpha}(z))
(1-q^{j}e^{\alpha}(z)).}
\end{align*}
\end{Pro}
\begin{proof}
The first  assertion follows from
the  second assertion
since $\lam(\rho^{\vee})\geq 0$
and $\lam(\rho^{\vee})=0$ if and only if $\lam=0$ for $\lam\in P_+$.
By Proposition \ref{Pro:ch-DS} and
 Proposition \ref{Pro:reduction-of-Vr-is-Vr-1},
 we have
 for $zw\in  T^b$,
 $z\in T, w\in T^{b-1}$,
 \begin{align*}
\on{tr}_{ \V_{G,b}}(q^{L_0}z w)=q^{\lam(\rho^{\vee})}
\on{tr}_{\mathbb{L}_{\lam}}(q^{-D}z)\on{tr}_{\V_{G,b-1}}(q^{-D}w)\\
\end{align*}
which equals to
$$\left(\frac{q^{\bra \lam,\rho^{\vee}\ket}
\prod\limits_{j=1}^{\infty}(1-q^j)^{\on{rk}\g}}{\prod\limits_{\alpha\in \Delta_+}(1-q^{\bra \lam+\rho,\alpha^{\vee}\ket})}
\right)\on{tr}_{\mathbb{V}_{\lam}}(q^{-D}z)\on{tr}_{\V_{G,b-1}}(q^{-D}w)$$
by \eqref{eq:ch-is-product}.
On the other hand for $b=2$, we know that
\begin{align*}
\on{tr}_{ (\Dch_G)_{[\lam]}}(q^{L_0}z w)=
\on{tr}_{\mathbb{V}_{\lam}}(q^{-D}z)\on{tr}_{\mathbb{V}_{\lam}}(q^{-D}w)
\end{align*}
by \eqref{ch-of-Dch}.
Hence the assertion follows inductively.
\end{proof}

\begin{Co}
The vertex algebra $\V_{G,b}$ is separated for all $b\geq 1$.
\end{Co}
\begin{proof}
We already know the statement for $b=1,2$
and the statement for $b\geq 3$ follows from Proposition \ref{Pro:conic}.
\end{proof}
\begin{Pro}\label{Pro:central-charge}
The vertex algebra 
$\V_{G,b}$ is conformal with central charge
\begin{align*}
b\dim \g -(b-2)\on{rk}\g-24 (b-2)(\rho|\rho^{\vee}).
\end{align*}
\end{Pro}
\begin{proof}
Set
$$\omega_b=\sum_{i=1}^b\rho_i(\omega_{\eqW})+\sum_{i=1}^{b-1}\rho_{gh,i}(\omega_{gh}).$$
Then
$\omega_b$ defines a conformal vector of the complex
$C_b$ (see \eqref{eq:complex-Cb})
of central charge
$b\dim \g -(b-2)\on{rk}\g-24 (b-2)(\rho|\rho^{\vee})$.
Let $T(z)=\sum_{n\in\Z}L_n z^{-n-2}$ be the field corresponding to $\omega_b$.

By Lemma \ref{Lem:action-of-conformal-v},
we have
\begin{align*}
T(z)P_i(w)\sim \frac{1}{z-w}\partial P_i(w)+\frac{d_i+1}{(z-w)^2}P_i(w)+
\sum_{j=2}^{d_i+2}\frac{(-1)^jj!}{(z-w)^{j+1}}q^{(i)}_j(w),
\end{align*}
where 
$q^{(i)}_j$ is some  homogeneous vector of $\mf{z}(\affg)$ of conformal weight 
$d_i-j+2$.
Hence,
\begin{align*}
Q_{(0)}\omega_b=\sum_{s=1}^{b-1}\sum_{i=1}^{\on{rk}\g}
\sum_{j=2}^{d_i+1}\partial^j\left((\rho_s(q_j^{(i)})-\rho_{s+1}(\tau (q_j^{(i)})))\rho_{gh,s}(\psi_i^*)\right).
\end{align*}
As easily seen,
the element
$\sum_{s=1}^{b-1}\left((\rho_s(q_j^{(i)})-\rho_{s+1}(\tau (q_j^{(i)}))\right)\rho_{gh,s}(\psi_i^*)$ is a coboundary.
Thus,
there exists
$z_{ij}\in \mf{z}(\affg)\* \mf{z}(\affg)
\* \bw{\semiinf+0}(\mf{z}(\affg))$
such that
%$\left(q_j^{(i)}\* 1-1\* \tau (q_j^{(i)}))\right)\psi_i^*=-Q_{(0)}z_{ij}$.
%It follows that
$$\tilde{\omega_b}=\omega_b +\sum_{s=1}^{b-1}\sum_{i=1}^{\on{rk}\g}
\sum_{j=2}^{d_i+2}\partial^j(\rho_s\* \rho_{s+1}\* \rho_{gh,s})(z_{ij})$$
defines
a cocycle.
Note that 
$(\tilde{\omega}_b)_{(n)}=(\omega_b)_{(n)}$ for $n=0,1$.
Since $\V_{G,b}$ is non-negatively graded for $b\geq 2$,
$\tilde{\omega}_b$ defines a conformal vector of $\V_{G,b}$
if and only if 
$(\tilde{\omega}_b)_{(3)}\tilde{\omega_b}$ is a constant multiplication of 
the vacuum vector
(\cite[Lemma 3.1.2]{Fre07}).
For $b\geq 3$,
this is obvious
since $\V_{G,b}$ is conical   by Proposition \ref{Pro:conic}.
For $b=2$,
observe that 
$\tilde{\omega_b}$ is annihilated by the action
of $\g\+ \g$,
and so is $(\tilde{\omega}_b)_{(3)}\tilde{\omega_b}$.
Therefore,
$(\tilde{\omega}_b)_{(3)}\tilde{\omega_b}$
has to belong to $\C=V_0\* V_0\subset \mc{O}(G)=(\V_{G,b})_0$.
We have shown that
the image 
$\omega_{\V_{G,b}}$
of $\tilde{\omega}_b$
defines a conformal vector of $\V_{G,b}$,
which gives the same grading
as Hamiltonian \eqref{eq:Hamiltonian-VGS}.

Finally,
we will show by induction on $b\geq 2$
that the central charge 
of $\omega_{\V_{G,b}}$ is the same as that of ${\omega}_b$.
Let $b=2$.
It is enough to show that
$\omega_{\Dch_G}$ 
is the unique conformal vector $\omega$
of $\Dch_G$ 
such that
$(\omega)_{(1)}=(\omega_{\Dch_G})_{(1)}$
and $\omega\in (\Dch_G)^{\g\+\g}$.
Since $(\Dch_G)^{\g\+\g}$ is stable under the action
of the $\mf{sl}_2$-triple $\{L_{-1},L_0,L_1\}$  associated with $\omega_{\Dch_G}$
and $(\Dch_G)^{\g\+\g}_0=\C$,
it follows that 
$\partial=L_{-1}$ is injective on
the subspace
 $\bigoplus_{\Delta\geq 1}(\Dch_G)^{\g\+\g}_{\Delta}$,
 where $(\Dch_G)^{\g\+\g}_{\Delta}=(\Dch_G)^{\g\+\g}\cap (\Dch_G)_{\Delta}$.
Hence one can apply \cite[Lemma 4.1]{Moriwaki} twice to obtain that
$\omega- \omega_{\Dch_G}=\partial^2 a$ for some $a\in (\Dch_G)^{\g\+\g}_0$.
Since $a$ is a 
constant multiplication of 
the vacuum vector,
this gives that $\omega= \omega_{\Dch_G}$.
Let $b\geq 3$.
Since $\V_{G,b}$ is
conic,
the same argument using
 \cite[Lemma 4.1]{Moriwaki}
 implies that
$\omega_{\V_{G,b}}$ 
is the unique conformal vector of $\V_{G,b}$ such that
$(\omega_{\V_{G,b}})_{(1)}$ coincides with the Hamiltonian
\eqref{eq:Hamiltonian-VGS}.
Let 
 $c$ be the 
central charge of $\omega_{\V_{G,b}}$.
By Proposition \ref{Pro:reduction-of-Vr-is-Vr-1},
$\omega_{\V_{G,b}}$ gives rise to a conformal vector of 
of $\V_{G,b-1}$
with central charge
$c-\dim \g+\on{rk}\g+24(\rho|\rho^{\vee})$,
which is annihilated by  $\g^{\+ b-1}$.
Hence the uniqueness of the conformal vector of $\V_{G,b-1}$ and the induction hypothesis gives the required result.

\end{proof}

\begin{Rem}
We have $\dim W_G^b=b\dim \g-(b-2)\on{rk}\g$.
Hence the central charge of $\V_{G,b}$ can be expressed as 
$$\dim W_G^b-24(b-2)(\rho|\rho^{\vee}).$$
\end{Rem}
\begin{Rem}
In the case that $\g$ is simply laced,
using the strange formula  $\dim \g=12|\rho|^2/h^{\vee}$
one finds that 
the central charge of $\V_{G,b}$ is expressed also as
\begin{align*}
(b-2(b -2)h^{\vee})\dim \g-(b-2)\on{rk}\g.
%=\dim W_G^b-2(b-2)h^{\vee}\dim \g
\end{align*}

\end{Rem}

\begin{Pro}\label{Pro:freeness}
For any $b\geq 1$,
$\V_{G,b}\in \KL_b^{\Delta}\cap  \KL_b^{\nabla}$.
In particular,
$\V_{G,b}$ is free over 
$U(\g[t^{-1}])$ and cofree over $U(t\g[t])$
by the action $\iota_i$, $i=1,\dots,b$.
\end{Pro}
\begin{proof}
We prove the statement by induction on $b\geq 1$.

For $r=1$,
$\V_{G,1}=\eqW$ is free over $U(t^{-1}\g[t^{-1}])$ and cofree over $U(t\g[t])$
(Propositions \ref{Pro:free} and \ref{Pro:g[t]t-invariant-part}).
Hence,
$\V_{G,1}\in \KL_1^{\Delta}\cap  \KL_1^{\nabla}$.

%For $b\geq 2$,
%$\V_{G,b}$ is conical and so is non-negatively graded (Proposition \ref{Pro:conic}).
%We claim that
%$\V_{G,b}$ is self-dual for $b\geq 2$.
%For $b=2$ it was already proved in \cite{Zhu08} that $\V_{G,2}=\Dch_G$ is self-dual.
%So let $b\geq 3$.
%By \cite{Li94},
%it is sufficient to show
%that $L_1(\V_{G,b})_1=0$,
%where $L_1=(\omega_{\V_{G,b}})_{(2)}$ and 
%$\omega_{\V_{G,b}}$ is the conformal vector of $\V_{G,b}$.
%
%
%For $b\geq 3$,
%$\V_{G,b}$ is conical by  Proposition \ref{Pro:conic}.
%Hence \cite{Frenkel:1993aa} the Shapovalov form is defined on $\V_{G,b}$.
%By Theorem \ref{Th:simple},
%this form is non-degenerate,
%and therefore $\V_{G,b}$ is self-dual.
%Thus,
% it is sufficient to show that $\V_{G,b}\in \KL_b^{\Delta}$.
Next let $b\geq 2$.
By the induction hypothesis and  Proposition \ref{Pro:reduction-of-Vr-is-Vr-1},
$H_{DS}^0(\V_{G,b})=\V_{G,b-1}$ belongs to $\KL_{b-1}^{\Delta}\cap  \KL_{b-1}^{\nabla}$.
Therefore
Proposition \ref{Pro:r-weyl} and
Proposition \ref{Pro:filtration} 
give that
 $\V_{G,b}\in \KL_b^{\Delta}$.
\end{proof}

%
%For vertex algebra objects $V$, $W$ in $\on{KL}$,
%define
%\begin{align*}
%V\circ  W:=H^{\frac{\infty}{2}+\bullet}(\affg_{-\kappa_\g},\g, V\otimes W)
%\end{align*}

\begin{Th}\label{Th:fusion}
Let  $b,b'\geq 1$
such that $b+b'\geq 3$.
We have
$$H^{\semiinf+i}(\affg_{-\kappa_\g},\g, \V_{G,b}\*\V_{G,b'})=0
\quad \text{for }i\ne 0,
$$
and
\begin{align*}
\V_{G,b}\circ \V_{G,b'}\cong \V_{G,b+b'-2}
\end{align*}
as vertex algebras.
\end{Th}
\begin{proof}
The first statement follows immediately from Proposition \ref{Pro:vanishing-relative}
and Proposition \ref{Pro:freeness}.
For $b=1$ or $b'=1$,
the second statement has been proved in 
Proposition \ref{Pro:reduction-of-Vr-is-Vr-1}
since $\V_{G,1}\circ \V_{G,b}\cong H_{DS}^0(\V_{G,b})$
by \eqref{eq:eqW-vs-DS}.
So
suppose that $b\geq 2$.
Set
\begin{align*}
C=\V_{G,b-1}\* \eqW\* \V_{G,b'}\*  \bigwedge\nolimits^{\semiinf+\bullet }({\mf{z}(\affg)})
\* \bigwedge\nolimits^{\semiinf+\bullet}(\affg).
\end{align*}
Let
$Q_{\mf{z}}$
be the 
 differential  
 of the complex
 $\V_{G,b-1}\* \eqW\* \bigwedge\nolimits^{\semiinf+\bullet }({\mf{z}(\affg)})
$
for the cohomology $H^{\semiinf+\bullet}(\LZ,\V_{G,b-1}\* \eqW)$.
We consider $Q_{\mf{z}}$
as a differential on $C$
that   acts trivially on the factors $\V_{G,b'}$
and $\bigwedge\nolimits^{\semiinf+\bullet}(\affg)$.
Let $Q_{\affg}$ be the differential of the complex  
$\eqW\* \V_{G,b'}\* \bigwedge\nolimits^{\semiinf+\bullet}(\affg)$
for 
  $H^{\semiinf+\bullet}(\affg_{-\kappa_\g}, \eqW\* \V_{G,b'})$.
  We consider $Q_{\affg}$
as a differential on $C$
that   acts trivially on the factors $\V_{G,b}$
and $\bigwedge\nolimits^{\semiinf+\bullet}(\mf{z}(\affg))$.
Then $C$ has the structure of a double complex.
 Let $H_{tot}^\bullet(C)$ be the total cohomology of $C$
 with the differential 
 $Q= Q_{\mf{z}}+Q_{\affg}$.
 As in the proof of Theorem \ref{Th:it-is-an-inverse-to-DS},
 we find that 
the total complex $C$ is a direct sum of finite-dimensional subcomplexes.

% 
%We consider $d_{\affg}$ 
%as a differential of $C$
%by the embedding
%\begin{align*}
%&\Dch_G\*\eqW\*
% \bigwedge\nolimits^{\semiinf+i }(\affg)\ra C,\\
%&u\* w\* \omega\mapsto 
% \overbrace{1\* \dots \* 1}^{r-2}\* u\* w\* \overbrace{1\* \dots \* 1}^{s-1}\*\overbrace{1\* \dots \* 1}^{r-2}\* \omega\*\overbrace{1\* \dots \* 1}^{s-1}.
%\end{align*}
% Since $\{Q_{\mf{z}},Q_{\affg}\}=0$,
% $C$ is equipped with the structure of a double complex.

Let $E_r\Rightarrow H^\bullet_{tot}( C)$ 
 be the spectral sequence 
 such that $d_0=Q_\mf{z}$ and $d_1=Q_{\affg}$.
We have
\begin{align*}
%&E_1^{p,q}
%\cong \bigoplus_{i=0}^q H^i(\LZ^{\oplus {r-3}}, \eqW^{\otimes r-2})\* H^{q-i}%(\LZ^{\oplus {s}},\eqW^{\otimes s-1})\*  \bigwedge\nolimits^{\semiinf+q }(\affg),\\
&E_2^{p,q}\cong  H^{\semiinf+p}(\affg_{-\kappa_\g}, 
H^{\semiinf+q}(\LZ,\V_{G,b-1}\* \eqW)\*\V_{G,b'}).
%\cong \delta_{q,0}\left(\V_{G,r}\circ \V_{G,s}\right)\* H^p(\g,\C).\text{TO BE CHECKED}
\end{align*}
Then $E_2^{p,q}=0$ if $p<0$ or $q<0$
by 
Lemma \ref{Lem:vnd}
and Proposition \ref{Pro:freeness}.
Thus, we obtain the  vertex algebra isomorphism
\begin{align}
H_{tot}^0(C)\cong  
E_2^{0,0}\cong
H^{\semiinf+0}(\affg_{-\kappa_\g}, \V_{G,b}\*\V_{G,b'})
=\V_{G,b}\circ \V_{G,b'}\label{eq:one-comp}.
\end{align}

Next, consider the spectral sequence
 $E_r'\Rightarrow H^\bullet_{tot}( C)$ 
 such that $d_0=Q_{\affg}$ and $d_0=Q_\mf{z}$.
 We have
 \begin{align*}
%&(E_1')^{p,q}\cong  \eqW^{\* r+s-2}\* \left(\bigwedge\nolimits^{\semiinf+i }({\mf{z}(\affg)})\right)^{\* r+s-3} \*H^q(\g,\C),\\
&(E_2')^{p,q} \cong H^{\semiinf+p}(\LZ,\V_{G,b-1}\* H^q(\affg_{-\kappa_\g},\eqW\* \V_{G,b'})).
\end{align*}
Then $(E_2')^{p,q}=0$ if $p<0$ or $q<0$
by 
Lemma \ref{Lem:vnd}
and Proposition \ref{Pro:freeness},
and we obtain the the isomorphism
\begin{align}
H_{tot}^0(C)\cong 
(E_2')^{0,0}\cong  H^{\semiinf+0}(\LZ, \V_{G,b-1}\* \V_{G,b'-1})\cong \V_{G,b+b'-2}
\label{eq:anothe-comp}
\end{align}
by Proposition \ref{Pro:from-DG} and the induction hypothesis.
Finally, \eqref{eq:one-comp} and \eqref{eq:anothe-comp}
gives the required isomorphism.
\end{proof}

\begin{Rem}\label{Rem:uniquness}
We note that
$\V_{G,b}$ is the unique vertex algebra object in $\KL^{\* b}$ 
%(with respect to the first or the last action of $\affg_{\kappa_c}$)
such that $\V_{G,1}\circ \V_{G,b}\cong \V_{G,b-1}$.
Indeed,
let $V$ be a vertex algebra object in $\KL^{\* b}$ such that $\V_{G,1}\circ \V_{G,b}\cong H_{DS}^0(\V_{G,b})\cong \V_{G,b-1}$.
Then by Theorem \ref{Th:it-is-an-inverse-to-DS},
we have
$V\cong H^{\semiinf+0}(\LZ,\eqW\* \V_{G,b-1})\cong \V_{G,b}$.
\end{Rem}

\begin{Rem}\label{Rem:showing-isomorphisms}
\begin{enumerate}
\item Suppose that there  exists a  vertex algebra object $V\in \KL^{\* b}$  with  vertex algebra embeddings
\begin{align*}
V\hookrightarrow \V_{G,b}\quad \text{and}\quad \V_{G,b-1}\hookrightarrow H_{DS}^0(V),
\end{align*}
then $\V_{G,b}\cong V$.
Indeed,
the embedding $V\hookrightarrow \V_{G,b}$ induces the embedding
$H_{DS}^0(V)\hookrightarrow  H_{DS}^0(\V_{G,b})\cong \V_{G,b-1}$.
The existence of the embedding 
$\V_{G,b-1}\hookrightarrow H_{DS}^0(V)$ then forces that 
$H_{DS}^0(V)\cong \V_{G,b-1}$.
It follows that $V\cong \V_{G,b}$ by
Theorem \ref{Th:it-is-an-inverse-to-DS}.
\item
In the above,
the 
existence 
of the embedding $V\hookrightarrow \V_{G,b}$
 can replaced by another condition as follows:
Suppose that there  exists a  vertex algebra object $V\in \KL^{\* b}$  with  vertex algebra embeddings
\begin{align*}
\V_{G,b-1}\hookrightarrow H_{DS}^0(V).
%\label{eq;emb}
\end{align*}
If $V$ is generated
by the subspace $\bigoplus\limits_{\Delta\leq d}\dim V_\Delta$ for some $d$ 
and
  $\dim V_\Delta=\dim (\V_{G,b})_\Delta$ for all $\Delta\leq  d$,
then 
$V\cong \V_{G,b}$.
(Note that Proposition \ref{Pro:conic}
provides $\dim (\V_{G,b})_\Delta$.)
Indeed,  by the left exactness of the functor $H^{\semiinf+0}(\LZ,\eqW_G\* ?)$ (see Lemma \ref{Lem:vnd})
and Theorem \ref{Th:it-is-an-inverse-to-DS},
the above map induces an embedding $\V_{G,b}\hookrightarrow V$.
%\begin{align}
%\V_{G,b}\hookrightarrow V.
%\end{align}
But the assumption implies that
this must be an isomorphism.
\end{enumerate}

\end{Rem}

%\begin{Pro}
%For $b\geq 1$,
%we have
%$\dim X_{\V_G,b}=b \dim \g-(b-2)\on{rk}\g$.
%In particular,
%$\V_{G,b}$ is finitely strongly generated.
%\end{Pro}
%\begin{proof}
%By Theorem \ref{Th:IMRN2016}
%and Proposition \ref{Pro:reduction-of-Vr-is-Vr-1},
%we have
%$X_{\V_{G,b}}=\Slo\times_{\g^*}X_{\V_{G,b+1}}
%=X_{\V_{G,b+1}}
%/\!/\!/_{\chi}N$.
%Therefore,
%$\dim X_{\V_{G,b}}=\dim X_{\V_{G,b+1}}-(\dim \g-\on{rk}\g)$.
%Since we have $\dim X_{\V_{G,b=2}}=\dim T^*G=2 \dim G$,
%the assertion follows inductively.
%\end{proof}

\begin{Th}\label{Th:Higgs-branch-conj}
The associated variety 
$X_{\V_{G,b}}$ of the vertex algebra $\V_{G,b}$ is isomorphic to $W^b_G$
for all $b\geq 1$.
\end{Th}

\begin{proof}
We assume that $W^b_G$
is reduced.
Otherwise,
we replace $W^b_G$
by its reduced scheme.

Set ${X}_b={X}_{\V_{G,b}}$,
$R_b=R_{\V_{G,b}}/\mbox{\footnotesize$\sqrt{(0)}$}=\mc{O}({X}_b)$,
$A_b=\mc{O}(W^b_G)$.
We wish to show that $R_b\cong A_b$.
We regard both $R_b$ and $A_b$
as Poisson algebra objects of $\on{QCoh}^G(\g^*)$
by the moment maps $\mu:X_b\ra \g^*$ and $\mu:W^b_G\ra \g^*$ with respect to the action of the first copy of $G$.
Since preimages 
%Let  ${X}_{b,reg}$
%and $W^b_{reg}$ be the preimage
of $\g_{reg}^*$
by the moment maps
are open and dense 
in $X_b$ and $W^b_G$,
%respectively.
%Since ${X}_{b,reg}$ and $W^b_{reg}$ are open dense subsets of $X_b$ and $W^b$,
the restriction maps
$R_b\ra R_b|_{\g^*_{reg}}$ 
and $A_b\ra A_b|_{\g^*_{reg}}$ 
are injective.

We prove by induction on $b\geq 2$
that
$R_b|_{\g_{reg}}\in \on{Coh}^{G}(\g^*_{reg})$ is  a free  $\mc{O}_{\g^*_{reg}}$-module
and coincides with $A^b|_{\g_{reg}^*}$.
 Since the complement of $\g_{reg}^*$ has codimension $3$ in $\g^*$ (\cite{Vel72}),
this shows that
\begin{align}
R_b\cong R_b|_{\g^*_{reg}}\cong A_b|_{\g^*_{reg}}\cong A_b.
\label{eq:iso-ring}
\end{align}

We have
\begin{align}
\kappa(R_b)\cong R_{b-1},\quad \kappa(A_b)\cong A_{b-1},
\label{eq:reduction-bfn}
\end{align}
where $\kappa$ is the Kostant-Whittaker reduction,
that is,
$\kappa(M)=M|_{\Slo}=M\*_{\mc{O}(\g^*)}\mc{O}(\Slo)$.
On the other hand,
it was shown in \cite{Ric17} using the equivariant descent theory
that the Kostant-Whittaker reduction gives an equivalence
\begin{align*}
\on{QCoh}^G(\g^*_{reg})\isomap \on{Rep}(\mathbf{I}^{cl}_G),\quad M\mapsto \kappa(M),
\end{align*}
where $\on{Rep}(\mathbf{I}^{cl}_G)$
is the category of representations of the group scheme $\mathbf{I}^{cl}_G$ over $\Slo$,
which are quasi-coherent over $\Slo$.
Moreover,
from the proof it follows that 
the above equivalence restricts to the equivalence between the category of the Poisson algebra objects in
$\on{QCoh}^G(\g^*_{reg})$ and that of the Poisson algebra objects $\on{Rep}(\mathbf{I}^{cl}_G)$.
Here, a Poisson algebra object in $\on{Rep}(\mathbf{I}^{cl}_G)$
is a Poisson algebra $R$ equipped with a Poisson algebra homomorphism
$\nu:\mc{O}(\Slo)\ra R$, and 
a $\on{Rep}(\mathbf{I}^{cl}_G)$-module structure is compatible with the 
$\mc{O}(\Slo)$-module structure given by $\nu$.

Hence, by \eqref{eq:reduction-bfn} and the induction hypothesis,
it follows that
$R_b|_{\g_{reg}^*}\cong A_b|_{\g_{reg}^*}$.
Moreover,
since it is a free $\mc{O}(\g^*_{reg})$-module
by the induction hypothesis,
$R_{b-1}$ is free as $\mc{O}(\Slo)$-module.
Thus, the above equivalence show that 
$R_b|_{\g_{reg}^*}=A_b|_{\g_{reg}^*}$ is a free $\mc{O}_{\g^*_{reg}}$-module.
This completes the proof.
%
%
%We have shown the isomorphism $R_b\cong A_v$ 
%as objects in
%$\on{QCoh^G}(\g^*)$.
\end{proof}
%\begin{Rem}
%It is not true in general that
%the associated scheme
%$\tilde{X}_{\V_{G,b}}$
%is isomorphic to $W^b_G$
%as $\tilde{X}_{\V_{G,b}}$ is not reduced in general while 
%$W^b_G$ is reduced.
%\end{Rem}

The following follows immediately from  \cite{BFN17} and Theorem 
\ref{Th:Higgs-branch-conj}.
\begin{Co}\label{Co:symplectic}
The associated variety 
$X_{\V_{G,b}}$ of $\V_{G,b}$  is symplectic for all $b\geq 1$,
that is, the Poisson structure of  $X_{\V_{G,b}}$ is symplectic on its smooth locus.
\end{Co}
The Higgs branch of a $4d$ $\mathcal{N}=2$ SCFT is expected have a finitely many symplectic leaves  (\cite{BeeRas}).
Hence, it is natural to expect the following.
\begin{Conj}\label{Conj}
The vertex algebra
$\V_{G,b}$ is quasi-lisse \cite{Arakawam:kq} for all $b\geq 1$,
 that is,
 $X_{\V_{G,b}}=W^b_G$ has finitely many symplectic leaves.
\end{Conj}
\begin{Rem}
The vertex algebra
$\V_{G=SL_n,b=3}$ is called 
{\em chiral algebras for trinion theories}
and denote by $\chi(T_n)$ in the literature.
A conjectural description of the list of 
strong generators has been given in \cite{Beem:2015yu,LemPee15}.
\end{Rem}

\appendix
\section{Examples}\label{appendix}
In this appendix we compute some examples 
of $\V_{G,b}$.

Below
we identify the space of the
symmetric invariant bilinear forms on a simple Lie algebra $\g$
with $\C$ by the correspondence
$\C\ni k\mapsto k\kappa_\g/2h^{\vee}=k\kappa_\g/4$.

First, let us consider the case
$G=SL_2$.

Let $\tilde{\mathbf{W}}_{G}^k$ be the vertex algebra generated by
fields
 $S(z),a(z), b(z), c(z), d(z)$,
subjected to the following OPEs:
\begin{align*}
&a(z)a(w)\sim 0, \ \ c(z)c(w)\sim 0,
\ \ a(z)c(w)\sim 0,\\
&a(z)b(w)\sim \frac{1}{2(z-w)}:a(w)^2:,\quad
a(z)d(w)\sim \frac{1}{2(z-w)}:a(w)c(w):,
\\
&b(z)c(w)\sim -\frac{1}{2(z-w)}:a(w)c(w):,
\quad c(z)d(w)\sim \frac{1}{2(z-w)}:c(w)^2:,\\
&b(z)b(w)\sim \frac{2k+3}{4}\left(\frac{1}{z-w}:a'(w) a(w):+ \frac{1}{(z-w)^2}:a(w)^2:\right),\\
&d(z)d(w)
\sim \frac{2k+3}{4}\left(\frac{1}{z-w}:c'(w) c(w):+ \frac{1}{(z-w)^2}:c(w)^2:\right),
\end{align*}
\begin{align*}
&b(z)d(w)\sim  \frac{1}{z-w}\left(\frac{1}{2}(:a(z)d(z):-:b(z)c(z):)+\frac{2k+1}{4}:a'(z)c(z):\right)\\
&
\qquad \qquad +\frac{2k+3}{4(z-w)^2}:a(w)c(w):,\\
&S(z)S(w)\sim -\frac{k+2}{z-w}\partial S(w)-\frac{2(k+2)}{(z-w)^2}S(w)+\frac{(k+2)^2c_k }{2 (z-w)^4},
\ \ c_k=1-\frac{6(k+1)^2}{k+2},\\
&S(z)a(w)\sim \frac{1}{z-w}b(w)
+\frac{2k+1}{4(z-w)^2}a(w),
\\
&S(z)b(w)\sim -\frac{1}{z-w}:S(w)a(w):
-\frac{2k+7}{4(z-w)^2}b(w)-\frac{(k+2)(2k+1)}{2(x-w)^3}  a(w),
\\
&S(z)c(w)\sim\frac{1}{z-w}d(w)
+\frac{2k+1}{4(z-w)^2}c(w),
\\
&S(z)d(w)\sim -\frac{1}{z-w}:S(w)c(w):
-\frac{2k+7}{4(z-w)^2}d(w)-\frac{(k+2)(2k+1)}{2(x-w)^3}  c(w).
\end{align*}

%Then,
%$
% :a(z)d(z)-b(z)c(z):-: a'(z)c(z):$
% is a singular vector of $\tilde{\mathbf{W}}_G^k$.
The equivariant affine $W$-algebra
 $\mathbf{W}_G^k$ is the quotient of $\tilde{\mathbf{W}}_G^k$
 by the submodule generated by the singular vector
 $$
 :a(z)d(z)-b(z)c(z):-: a'(z)c(z):-1.$$
The conformal vector (the stress tensor) of $\mathbf{W}_G^k$ is given by
\begin{align*}
T(z)%=:S(z)(a(z)c'(z)-a'(z)c(z)):+:(b(z)d'(z)-b'(z)d(z))\\
%+\frac{2k+7}{2}(d'(z)a'(z)-b'(z)c'(z))
%\\+\frac{1}{24}
%\left((2k+7)(a(z)c'''(z)+3a'(z)c''(z))-2(2k+5)(3a''(z)c'(z)+a'''(z)c(z)
%\right)\\
=:S(z)(a(z)c'(z)-a'(z)c(z)):+:(b(z)d'(z)-b'(z)d(z)):\\
+\frac{2k+7}{2}(:a'(z)d'(z):-:b'(z)c'(z):)\\
-\frac{6k+17}{24}\left(3:a''(z)c'(z):+:a'''(z)c(z):\right).
\end{align*}
We have
\begin{align*}
T(z)S(w)\sim \frac{1}{z-w}\partial S(w)+\frac{1}{(z-w)^2}2 S(w)+\frac{(2k+1)(3k+4)}{2(z-w)^4},
\end{align*}
\begin{align*}
&T(z)a(w)\sim \frac{1}{z-w}\partial a(w)-\frac{1}{2(z-w)^2} a(w),\\
&T(z)b(w)\sim \frac{1}{z-w}\partial b(w)+\frac{1}{2(z-w)^2} b(w)+\frac{2k+1}{2(z-w)^3}a(w),\\
&T(z)c(w)\sim \frac{1}{z-w}\partial c(w)-\frac{1}{2(z-w)^2} c(w),\\
&T(z)d(w)\sim \frac{1}{z-w}\partial d(w)+\frac{1}{2(z-w)^2} b(w)+\frac{2k+1}{2(z-w)^3}c(w),
\end{align*}
and
\begin{align*}
T(z)T(w)\sim \frac{1}{z-w}\partial T(w)+\frac{1}{(z-w)^2}2 T(w)+\frac{2(2-3k)}{2(z-w)^4}.
\end{align*}
Thus,
the central charge of $\mathbf{W}_G^k$
is $2(2-3k)$,
and
\begin{align*}
\Delta_{a}=\Delta_{c}=-\frac{1}{2},\quad \Delta_{b}=\Delta_{d}=\frac{1}{2},
\quad \Delta_S=2.
\end{align*}

We have
an embedding
$\W^k(\mf{sl}_2)\hookrightarrow \mathbf{W}_G^k$,
and the image of $\W^k(\mf{sl}_2)$ is generated by $S(z)$.
Indeed,
for $k\ne -2$,
$-S(z)/(k+2)$ is a conformal vector of central charge $c_k=1-\frac{6(k+1)^2}{k+2}$,
and for $k=2$, $S(z)$ generates a commutative vertex subalgebra 
$\mf{z}(\affg)$
of $\W^{-2}(\mf{sl}_2)$.

The embedding
$$V^{k^*}(\mf{sl}_2)\hookrightarrow \mathbf{W}_G^k,\quad k^*=-k-4,$$
is given as follows:
\begin{align*}
e(z)&\mapsto :S(z)c(z)^2:+:d(z)^2:-\frac{2k+7}{2}(:c(z) d'(z):-:c'(z)d(z):)
\\&
+\frac{2k+7}{4}: c'(z)^2:
-\frac{2k+3}{8}:c''(z)c(z):,
\\h(z)&\mapsto 2:S(z)a(z)c(z):+2 :b(z)d(z):\\
&-\frac{2k+7}{2}\left(:a(z)d'(z)+:b'(z)c(z):-:a'(z)d(z):-:b(z) c'(z):
\right)\\
&+\frac{2k+7}{2}:a'(z)c'(z):-\frac{2k+3}{4}:a''(z)c(z):,\\
f(z)&\mapsto -S(z)a(z)^2:-:b(z)^2:+\frac{2k+7}{2}(:a(z) b'(z)-a'(z)b(z):)\\
&-\frac{2k+7}{4}:a'(z)^2:+\frac{2k+3}{8}: a''(z)a(z):.
\end{align*}
The vertex algebras $\W^k(\mf{sl}_2)$ and $V^{k^*}(\mf{sl}_2)$ from a dual pair inside $\mathbf{W}_G^k$.

We have
\begin{align*}
&e(z)a(w)\sim -\frac{1}{z-w}c(w),\ \
e(z)b(w)\sim -\frac{1}{z-w}d(w),\\
&h(z)a(w)\sim -\frac{1}{z-w}a(w),\quad
h(z)b(w)\sim -\frac{1}{z-w}b(w),\\
&h(z)c(w)\sim \frac{1}{z-w}c(w),\quad
h(z)c(w)\sim \frac{1}{z-w}d(w),\\
&f(z)c(w)\sim- \frac{1}{z-w}a(w),\ \
f(z)d(w)\sim -\frac{1}{z-w}b(w), \\
&e(z)c(w)\sim e(z)d(w)\sim 0,
\quad f(z)a(w)\sim f(z)b(w)\sim 0.
\end{align*}
In $R_{\mathbf{W}_G^k}=\mc{O}({\mathbf{S}_G})=\mc{O}(G)\otimes \mc{O}({\Slo})$,
the images of $a$, $b$, $c$, $d$ generate $\mc{O}(G)$
and the image of $S$ generates $\mc{O}({\Slo})$.

As in Section \ref{Sec:main-construction},
we set $\eqW_G=\eqW^{-2}_G$.
The complex $C_2$ equals to $\eqW_G\* \eqW_G\* \bw{\bullet}(\mf{z}(\affg))$.
For $u\in \eqW_G$,
we set $u_1=u\* 1, u_2=1\* u_2\in \eqW_G\* \eqW_G$.
The field
$Q^{(2)}(z)$ in
\eqref{eq:diff-qr}
is given by
\begin{align*}
Q^{(2)}(z)=S_1(z)\psi^*(z)-S_2(z)\psi^*(z).
\end{align*}
The isomorphism
\begin{align*}
\D_{G}^{ch}\isomap \V_{G,2}=H^{\frac{\infty}{2}+0}(\LZ,\mathbf{W}_G\otimes  \mathbf{W}_G)
\end{align*}
is given by
\begin{align*}
&u(z)\mapsto u_1(z)\quad (u\in \fing),\\
&x_{11}(z)\mapsto :c_1(z)b_2(z):+:d_1(z)a_2(z):+:c_1(z)a_2(z)\psi(z)\psi^*(z):,\\
&x_{12}(z)\mapsto -:a_1(z) b_2(z):-:b_1(z) a_2(z):-:a_1(z)a_2(z)\psi(z)\psi^*(z):,\\
&x_{21}(z)\mapsto -:c_1(z) d_2(z): -:d_1(z)c_2(z):-: c_1(z)c_2(z)\psi(z)\psi^*(z):,\\
&x_{22}(z)\mapsto -:a_1(z)d_2(z) -:b_1(z)c_2(z):-: a_1(z)c_2(z)\psi(z)\psi^*(z):.
\end{align*}
Here 
$x_{ij}(z)$,
$1\leq i,j\leq 2$, are generators of the commutative vertex algebra
$\mc{O}(J_{\infty}G)$ satisfying the relation
$x_{11}(z)x_{22}(z)-x_{12}(z)x_{21}(z)=1$.

\smallskip

By \cite{{BFN17}},
we have
$$W_{G=SL_2}^{b=3}=\C^2\* \C^2\* \C^2,$$
where the symplectic structure of 
$\C^2\* \C^2\* \C^2$ is induced by the  natural symplectic structure of 
$\C^2$.
The Hamiltonian $G$-action on $\C^2$
induces three commuting Hamiltonian $G$-action
$\C^2\* \C^2\* \C^2$.
Let $SB((\C^2)^{\otimes 3})$ be the $\beta\gamma$ system associated with the symplectic vector space
$\C^2\* \C^2\* \C^2$.

\begin{Th}\label{Th:SL23pt}
We have
$\V_{G=SL_2,b=3}\cong SB((\C^2)^{\otimes 3})
$.
\end{Th}
\begin{proof}
By Remark \ref{Rem:showing-isomorphisms} (i),
it is enough to construct vertex algebra homomorphisms
\begin{align*}
SB((\C^2)^{\otimes 3})
\ra H^{\semiinf+0}(\LZ, \eqW_G\* \Dch_G)\cong \V_{G,3},\\
\V_{G,2}\cong \Dch_{G}\ra H_{DS}^0(SB((\C^2)^{\otimes 3})),
\end{align*}
which can be done directly.
\end{proof}

We have
$X_{\V_{G=SL_2,b=3}}\cong \C^2\* \C^2\* \C^2$,
which agree with Theorem \ref{Th:Higgs-branch-conj}.

\begin{Th}[{\cite[Conjecture 1]{BeeLemLie15}}]\label{Th:SL24pt}
We have
$\V_{G=SL_2,b=4}\cong L_{-2}(D_4)$.
\end{Th}
\begin{proof}
By Remark \ref{Rem:showing-isomorphisms} (i),
it is enough to show that there are nonzero vertex algebra homomorphisms 
\begin{align*}
L_{-2}(D_4)\ra \V_{G,4}
\quad\text{and}\quad \V_{G,3}\ra H_{DS}^0(L_{-2}(D_4)).
\end{align*}
Both maps can be constructed directly.
We  explain the details only for the first map.

We have
\begin{align*}
\V_{G,4}\cong \V_{G,3}\circ \V_{G,3}=H^{\semiinf+0}(\affg_{-4},\g, \V_{G,3}\* \V_{G,3})
\\\cong H^{\semiinf+0}(\affg_{-4},\g, SB((\C^2)^{\otimes 3})\* SB((\C^2)^{\otimes 3})).
\end{align*}
Let $U=SB((\C^2)^{\otimes 3})\* SB((\C^2)^{\otimes 3}))^{\g[t]}$,
where the $\g[t]$ acts diagonally. As easily seen there is a non-trivial vertex algebra homomorphism
$\tilde{\psi}:V^{-2}(D_4)\ra U$.
Composing it with  
the vertex algebra homomorphism
$$U\ra H^{\semiinf+0}(\affg_{-4},\g, SB((\C^2)^{\otimes 3})\* SB((\C^2)^{\otimes 3})),\quad
u\mapsto [u\* \mathbf{1}],$$
we get a vertex algebra homomorphism
$$\psi:V^{-2}(D_4)\ra H^{\semiinf+0}(\affg_{-4},\g, SB((\C^2)^{\otimes 3})\* SB((\C^2)^{\otimes 3})).$$
This is nonzero,
because $\psi$ sends the vacuum vector to the vacuum vector.
By \cite{AM15},
$L_{-2}(D_4)$ is a quotient of 
$V^{-2}(D_4)$ by a submodule generated by three singular vectors of weight $2$.
The map $\tilde{\psi}$ sends two of the three singular vectors to zero,
but it does not kill the remaining  singular vector, say $w$.
Nevertheless,
one finds that
$\psi(w)$ is a cobundary, that is,
$\psi(w)=0$.
Hence $\psi$ induces a vertex algebra embedding
$$L_{-2}(D_4)\hookrightarrow H^{\semiinf+0}(\affg_{-4},\g, SB((\C^2)^{\otimes 3})\* SB((\C^2)^{\otimes 3}))\cong \V_{G,4}$$
as required.
\end{proof}

It was shown \cite{AraMal}
that $X_{L_{-2}(D_4)}$ is isomorphic to the minimal nilpotent orbit closure 
$\overline{\mathbb{O}_{min}}$ of $D_4$.
This agree with the fact (\cite{{BFN17}}) that
$$W_{G=SL_2}^{b=4}\cong \overline{\mathbb{O}_{min}}\quad\text{in $D_4$}.$$

\medskip

Next, let $G=SL_3$.
\begin{Th}[{\cite[Conjecture 4]{BeeLemLie15}}]\label{Th:E6}
We have
$\V_{G=SL_3,b=3}\cong L_{-3}(E_6)$.
\end{Th}
\begin{proof}
 It is enough to construct 
a vertex algebra embedding
\begin{align}
\Dch_{G=SL_3}\hookrightarrow  H_{DS}^0(L_{-3}(E_6)),
\label{eq:E6}
\end{align}
since the character formula (Proposition \ref{Pro:conic}) shows that
$\dim (\V_{G=SL_3,b=3})_1=3.8+2.3^3=78=\dim E_6=\dim L_{-3}(E_6)_1$.

So let us show the existence of the map \eqref{eq:E6}.
Since $\Dch_{G=SL_3}$ is simple,
it is sufficient to construct non-trivial vertex algebra homomorphism
$\Dch_{G=SL_3}\ra  H_{DS}^0(L_{-3}(E_6))$.

We adopt  the standard Bourbaki numbering for the simple roots 
$\{\alpha_1,\alpha_2,\dots,\alpha_6\}$
for $E_6$
and we denote by 
$\varpi_1,\dots,\varpi_6$
the corresponding fundamental weights.
Fix  roots vectors $x_{\alpha}$,
$\alpha\in \Delta$.
The Drinfeld-Sokolov reduction in \eqref{eq:E6}
is taken with respect to the action of $\widehat{\mf{sl}}_3\subset E_6$ generated by $x_{\pm\alpha_1}(z)$,
$x_{\pm\alpha_3}(z)$.

Since 
the fields
$x_{\pm\alpha_5}(z)$,
$x_{\pm\alpha_6}(z)$,
$x_{\pm\alpha_2}(z)$,
$x_{\pm\theta}(z)$
%$\alpha\in \{\pm \alpha_2,\pm \alpha_5, \pm \alpha_6, \pm \theta\}$,
commute with 
the subalgebra $\widehat{\mf{sl}}_3$
generated by $x_{\pm\alpha_1}(z)$,
$x_{\pm\alpha_3}(z)$,
they define elements of $H_{DS}^0(L_{-3}(E_6))$,
and we have a vertex algebra homomorphism
$V^{-3}(\mf{sl}_3)\ra H_{DS}^0(L_{-3}(E_6))$.

Set $S=\{\alpha_1+\alpha_3+\alpha_4,\alpha_1+\alpha_3+\alpha_4+\alpha_5,
\alpha_1+\alpha_3+\alpha_4+\alpha_5+\alpha_6,
\alpha_1+\alpha_2+\alpha_3+\alpha_4,
\alpha_1+\alpha_2+\alpha_3+\alpha_4+\alpha_5,
\alpha_1+\alpha_2+\alpha_3+\alpha_4+\alpha_5+\alpha_6,
-(\alpha_2+\alpha_3+2\alpha_4+2\alpha_5+\alpha_6),
-(\alpha_2+\alpha_3+2\alpha_4+\alpha_5+\alpha_6),
-(\alpha_2+\alpha_3+2\alpha_4+2\alpha_5)\}$.
The $9$-dimensional subspace
$\bigoplus_{\alpha\in S}\C x_{\alpha}$
form a commutative subalgebra
of $E_6$,
which commutes with the subalgebra $\widehat{\mf{sl}}_3$
generated by $x_{\pm\alpha_1}(z)$,
$x_{\pm\alpha_3}(z)$,
and
 is invariant under the adjoint action of 
$\widehat{\mf{sl}}_3^{\oplus 2}$
generated by
$x_{\pm\alpha_5}$,
$x_{\pm\alpha_6}$,
$x_{\pm\alpha_2}$,
$x_{\pm\theta}$.
It is isomorphic to $\C^3\* \C^3$ as a representation of  the latter subalgebra $\widehat{\mf{sl}}_3^{\oplus 2}$.

Let $U$ be the commutative vertex subalgebra of $L_{-3}(E_6)$
generated by 
$x_{\alpha}(z)$,
$\alpha\in S$.
Since $U$ commutes with 
$x_{\alpha_1}(z)$ and
$x_{\alpha_2}(z)$,
there is an obvious vertex algebra homomorphism
$U\ra H_{DS}^0(L_{-3}(E_6))$.
We claim that
this map factors through a 
vertex algebra homomorphism
$\mc{O}(J_{\infty}SL_3)\ra H_{DS}^0(L_{-3}(E_6))$.

%The vectors $x_{\alpha}$
%with $\alpha=$
%$[ 1, 0, 1, 1, 0, 0 ]$, %e12
%$[ 1, 0, 1, 1, 1, 0 ]$, %e18
%$[ 1, 0, 1, 1, 1, 1 ]$, %e23
%$[ 1, 1, 1, 1, 0, 0 ]$, %e17
%$[ 1, 1, 1, 1, 1, 0 ]$, %e22
%$[ 1, 1, 1, 1, 1, 1 ]$, %e27
%$-[ 0, 1, 1, 2, 2, 1 ]$, %f31
%$-[ 0, 1, 1, 2, 1, 1 ]$, %f28
%$-[ 0, 1, 1, 2, 1, 0 ]$, %f24
%defines cocycles in 

By \cite{AM15},
the maximal submodule of $V^{-3}(E_6)$ is generated by
a singular vector $v$ of weight $\varpi_1+\varpi_6$ in degree 2.
The $E_6$-module $V_{\varpi_1+\varpi_6}$ generated by $v$ has dimension $650$.
Using the relations
obtained from
vectors in  $V_{\varpi_1+\varpi_6}$
of 
weight $[2, 1, 0, -2, 1, 0]$,
$[2, -1, 0, -1, 1, 0]$,
$[2, 0, 0, -1, 1, 0]$,
$[1, 0, 1, -1, 0, 0]$,
we find that
the commutative subalgebra generated by
$x_{\alpha}$, $\alpha\in S$,
by the multiplication $x_{\alpha}.x_{\beta}=(x_{\alpha})_{(-1)}x_{\beta}$,
is isomorphic to
$\mc{O}(SL_3)$.
It follows that
the vertex algebra homomorphism
$U\ra H_{DS}^0(L_{-3}(E_6))$
 factors through a 
vertex algebra homomorphism
$\mc{O}(J_{\infty}SL_3)\ra H_{DS}^0(L_{-3}(E_6))$.

We conclude that
the fields
$x_{\pm\alpha_5}(z)$,
$x_{\pm\alpha_6}(z)$,
$x_{\pm\alpha_2}(z)$,
$x_{\pm\theta}(z)$,
$x_\alpha(z)$ ($\alpha\in S$)
generate a vertex subalgebra 
of $H_{DS}^0(L_{-3}(E_6))$ that is isomorphic to 
$\Dch_{G=SL_3}$.
\end{proof}

It was shown \cite{AraMal}
that $X_{L_{-3}(E_6)}$ is isomorphic to the minimal nilpotent orbit closure 
$\overline{\mathbb{O}_{min}}$ of $E_6$.
This agree with the fact (\cite{{BFN17}}) that
$$W_{G=SL_3}^{b=3}\cong \overline{\mathbb{O}_{min}}\quad\text{in $E_6$}.$$

\newcommand{\etalchar}[1]{$^{#1}$}

%\bibliographystyle{alpha}
%
%\bibliography{/Users/tomoyuki/Documents/Dropbox/bib/math}
%\bibliography{/Users/tomoyuki1/Documents/Dropbox/bib/math}

\end{document}